\documentclass[a4paper,12pt]{article}
\usepackage{amsmath} 
\usepackage{bm}
\usepackage{amssymb}
\usepackage{graphicx}
\usepackage{setspace}
\usepackage{enumitem}
\usepackage{lscape}

\usepackage{amsthm}
\newtheorem{thm}{Theorem}[section]  %[chapter]
\newtheorem{lemma}[thm]{Lemma}
\newtheorem{cor}[thm]{Corollary}
\newtheorem{prp}[thm]{Proposition}
\theoremstyle{definition}
\newtheorem{dfn}[thm]{Definition}
\newtheorem{example}[thm]{Example}

\theoremstyle{remark}
\newtheorem{remark}[thm]{Remark}

\numberwithin{equation}{section} %{chapter}

\makeatletter  
  \@addtoreset{equation}{section}
\makeatother

\usepackage[tight]{minitoc}               % shrink spacing
\mtcsetfeature{secttoc}{open}{\vspace{-6pt}}  % default=0pt
\mtcsetfeature{secttoc}{close}{\vspace{-3pt}} % default=0pt
\usepackage[usenames]{color}
  \newcommand{\blue}[1]{\textcolor{black}{#1}}
  \newcommand{\red}[1]{\textcolor{black}{#1}}
  
\def\comment#1{ }

\author{\large %Akihito Ebisu\\
Yoshishige Haraoka \\
Hiroyuki Ochiai\\
Takeshi Sasaki\\
Masaaki Yoshida}

\title{\bf  Fuchsian differential equations of order 3,...,6 with three singular points and an accessory parameter}

\date{\today}
\begin{document}\maketitle
\dosecttoc

\begin{abstract}
  Fuchsian differential equations $H_j$ of order $j=3,\dots,6$ with three singular points and one accessory parameter are presented. The shift operators for $H_6$ are studied. They lead to assign the accessory parameter of $H_6$ a cubic polynomial of local exponents so that the equation has several nice symmetries. The other equations will be studied in the forthcoming papers.
\end{abstract}
\tableofcontents
\vfill

\noindent
{\bf Subjectclass}[2020]: Primary 34A30; Secondary 34M35, 33C05, 33C20, 34M03.


\newpage
\section*{Introduction}\setcounter{stc}{1}%\nostcrule   
\addcontentsline{toc}{section}{\protect\numberline{}Introduction} 
A Fuchsian ordinary differential equation is called rigid
if it is uniquely determined
by the local behaviors at the regular singular points.
In other words, a Fuchsian ordinary differential equation is rigid
if it is free of accessory parameters.
For rigid Fuchsian ordinary differential equations,
we \blue{know how to obtain} integral representations of solutions,
monodromy representations, shift relations, irreducibility conditions,
connection coefficients and so on (cf. \cite{Osh,Hara}).
While for non-rigid differential equations,
we have no way to know those things in general.
\par\medskip

\red{In this paper and in the forthcoming paper \cite{HOSY2},}
we study several Fuchsian equations with three singular points $\{0,1,\infty\}$.
A most naive generalization of the Gauss hypergeometric equation $E_2$ with the Riemann scheme
$$    R_2:\left( \begin{array}{ccc}
    x=0:&0&a_1\\
    x=1:&0&a_2\\
    x=\infty:&a_4&a_3\end{array}\right),\quad a_1+\cdots+a_4=1,$$
    would be \blue{an equation} of order three with the Riemann scheme
$$    R_3:\left( \begin{array}{cccc}
    x=0:&0&b_1&b_2\\
    x=1:&0&b_3&b_4\\
    x=\infty:&b_7&b_5&b_6\end{array}\right),\quad b_1+\cdots+b_7=3,$$
\blue{which we denote by $H_3$.}    This has an expression as
$$H_3:x^2(x-1)^2\partial^3+x(x-1)p_2\partial^2+p_1\partial+p_0\in\mathbb{C}[x][\partial],\quad \partial=d/dx$$
where $p_2,\ p_1$ and $p_0$ are polynomials in $x$ at most of degree $1$, $2$ and 1, respectively.
The number of coefficients is 7, and the number of free local exponents is 6, thus one coefficient is not determined by the local exponents. Actually, the constant term of $p_0$ is not determined, which is often called the {\it accessory parameter}.
  
$H_3$ is connected via {\it addition and middle convolution}
%\footnote{\blue{There is no $H_7,\dots$, see \S \ref{fromH3toH654bymc}}}
 with equations $H_4,H_5$ and $H_6$ of order 4, 5 and 6, with respective Riemann schemes: 
$$R_4:\left( \begin{array}{ccccc}
  x=0:& 0&1&c_1&c_2\\
  x=1:& 0&1&c_3&c_4\\
  x=\infty:& c_8&c_5&c_6&c_7\end{array}\right), \quad
R_5:\left( \begin{array}{cccccc}
 x=0:&  0&1&d_1&d_2&d_3\\
 x=1:&  0&1&d_4&d_5&d_6\\
 x=\infty:&  d_9&\ d_9+1\ &d_9+2\ &d_7&d_8\end{array}\right), $$$$
R_6:\left( \begin{array}{ccccccc}
 x=0:&  0&1&2&e_1&e_2&e_3\\
 x=1:&  0&1&2&e_4&e_5&e_6\\
 x=\infty:&  e_0&\ e_0+1\ &\ e_0+2\ &e_7&e_8&e_9\end{array}\right),$$
 where $c_8,\ d_9$ and $e_0$ are determined by the Fuchs relation. We assume that these equations have no logarithmic solution at the singular points (except \S 2.4.2) unless otherwise stated.
   $H_j\ (j=3,4,5,6)$ has $j+3$ free local exponents and one accessory parameter.

   For example, $H_6$ is obtained from $H_3$ as follows: \par\smallskip\noindent(1) Compose $x(x-1)X$ from the left, and $X^{-1}$ from the right, where $X:=x^{g_0}(x-1)^{g_1}$. Then the head (top-order term) of the equation changes into $x^3(x-1)^3\partial^3$. \par\smallskip\noindent(2) Compose $\partial^3$ from the left to get $(\theta,\partial)$-form \blue{(refer to \S \ref{Genxz}), where $\theta:=x\partial$.} \par\smallskip\noindent(3) Replace $\theta$ by $\theta-u$ (middle convolution with parameter $u$). \par\smallskip\noindent
   Then the Riemann scheme of the resulting  equation is given as
$$ \left( \begin{array}{cccccc}
   0&1&2&g_0+u&b_1+g_0+u&b_2+g_0+u\\
   0&1&2&g_1+u&b_3+g_1+u&b_4+g_1+u\\
   1-u&2-u&3-u&b_5-g_0-g_1-u&b_6-g_0-g_1-u&b_7-g_0-g_1-u
     \end{array}\right).$$
We rename the local exponents as in $R_6$, and get the equation $H_6$.
The shifts of the three new parameters $g_0\to g_0\pm1,g_1\to g_1\pm1$ and $u\to u\pm1$
induce the shifts of the local exponents:
$$
 \begin{aligned}
  sh_1&:(e_1,e_2,e_3)\to(e_1\pm1,e_2\pm1,e_3\pm1),\\ 
  sh_2&:(e_4,e_5,e_6)\to(e_4\pm1,e_5\pm1,e_6\pm1),\\
  sh_3&:(e_1,\dots,e_7,e_8,e_9)\to
  (e_1\pm1,\dots,e_6\pm1,e_7\mp1,e_8\mp1,e_9\mp1).
 \end{aligned}
$$
For these shifts, we present the shift operators explicitly (Theorem \ref{shopH6}). 
When the equation is rigid, the construction of shift operators is known (\cite{Osh} Chapter 11).
 \par\medskip
   Since the equation $H_6$ has an accessory parameter, say $ap$, writing $H_6=H_6(e,ap)$, the shift operators for the shifts $sh_i$ send the solutions of $H_6(e,ap)$ to those of $H_6(sh_i(e),ap')$ for some $ap'$ not necessarily equal to $ap$.
   \par\medskip   
When $ap$ is a polynomial of $e$, say $f(e)$, if $H_6(e,f(e))$ admits a shift operator for each shift $sh_i$, then $f(e)$ has a specific form (Theorem 7.3). This is the main theorem in this paper. If $f(e)$ is a cubic polynomial of $e$, then $f(e)$ is explicitly written with coefficients $a=a(a_1,\dots,a_6)$ (Corollary 7.4). This is denoted by $f(e,a)$. 
\par\medskip
  {\sl We find a polynomial $f(e,a)$ of the local exponents $e$ with a set $a$ of parameters such that,
for every shift $sh_j$, the shift operator sends the solution of $H_6(e,f(e,a))$ to those of $H_6(sh_j(e),f(sh_j(e),a))$ (Theorem 7.3).
This is the main theorem in this paper.}
  \par\medskip
We set $G_6(e,a)=H_6(e,f(e,a))$. By operating a middle convolution to $G_6(e,a)$, we get the equation $G_3(e,a)$ of order 3.
 Then via addition and middle convolution, we get $G_4(e,a)$ and $G_5(e,a)$ from $G_3(e,a)$, where the accessory parameters are replaced by polynomials of the local exponents of $H_4,H_5$ and $H_3$, respectively. Finally, we get $E_j=E_j(e):=G_j(e,0)$, $(j=3,4,5,6)$.
 \par\medskip
\red{Codimension-2 specializations
\footnote{\red{For a Fuchsian equation $E$, a codimension-$k$ specialization of $E$ is an equation $E$ with $k$ linearly independent relations among the local exponents, apart from the Fuchsian relation}}
 $S\!E_3$ of $E_j$ $(j=3,4,5,6)$ having rich shift operators are studied in  \cite{HOSY2}; $S\!E_3$ is equivalent to the Dotsenko-Fateev equation.}

   \par\bigskip
   This paper is organized as follows. In Section 1, The equation $H_6$ is introduced. We tabulate the equations $H_5,H_4, H_3$ and define $G_j,E_j\ (j=3,4,5,6)$ without much explanation. This is to show the reader what kind of equations we treat. 
   
   In order to define equations and to study shift operators,
   we need various tools of investigation, which we prepare in Section 2. 
When a certain transformation such as a transformation caused by a
coordinate change is performed to an equation,
it may happen that the equation remains the same 
with certain change of parameters. In such a case,
the equation is said to be {\it symmetric} relative to this transformation.
We study the following symmetries
\begin{itemize}
\item adjoint symmetry; when the adjoint equation remains the same,
  with some change of parameters, 
\item differentiation symmetry; when derivatives of solutions satisfy the same equation, with some change of parameters, 
\item symmetry relative to the coordinate changes $x\to1/x$ and $x\to1-x$.
\end{itemize}
We recall the notion of {\it accessory parameters}, which plays a central role in this paper.
We see that each of $H_j\ (j=3,4,5,6)$ has one accessory parameter. 
\par\smallskip

In Section 3, we review the notion of {\sl addition and  middle convolution}, which is important to know how the equations are related among them. Explicit procedure of getting $H_6,H_5,H_6$ from $H_3$, and the inverse procedure are presented.
\par\smallskip\noindent

In general, for  shifts (Definition \ref{DefShift}) of local exponents $sh_\pm: e\to e_\pm$ of a differential equation $H(e,ap)$, \blue{where $e_\pm$ denote the shifted exponents,} 
if a non-zero differential operator $P_{\pm}=P_{\pm}(e)$ sends solutions of $H(e,ap)$
to those of $E(e_\pm,ap_\pm)$, we call the operator $P_{\pm}$ the {\it shift operators}
of $H$ for the shift of the local exponents $e\to e_\pm$. 
These operators are important tools to see the structure of the space
of solutions. If such operators $P_{\pm}$ exist, we define the operator
$Sv_{e}$ by $P_+(e_-)\circ P_-(e)$, which turns out to be a constant mod $H(e)$. \footnote{\blue{Composition of two differential operators $P$ and $Q$ is
  denoted by $P\circ Q$; we often  write it as $PQ$.}}
We call such a constant the {\it S-value} for the shifts $e\to e_\pm$.
When $Sv_{e}$ vanishes then $H(e)$ is {\it reducible}. These are discussed in Section 4.
\par\smallskip

In Section 5, we first present these procedures for 
the Gauss equation $E_2$, which plays the ideal model of our study:
we recall the well-known properties such as the shift operators, reducibility conditions,
and explicit decompositions when the equation is reducible, \dots,
which will be generalized later for the equations above.

\par\smallskip
In Section 6, we study shift operators of our main equation $H_6$.
We find shift operators for each shift $sh_j$, S-values, and reducibility conditions,
and when $H(\epsilon)$ is reducible for some $e=\epsilon$, we see how the factorization of $H(\epsilon)$ is inherited to $H(sh_j(\epsilon))$.

\par\smallskip
In section 7, we state the main theorem (Theorem \ref{shiftopG6}) in this paper: we find cubic polynomials $S_{10}, t_{2i}, t_{3i}\ (i=1,2,3)$ \blue{of the local exponents} such that if the accessory parameter $ap$ is assigned as
$$f(e,a)=S_{10}+a_0+a_1t_{21}(e)+\cdots+a_6t_{33}(e),$$
\blue{where $a_0,\dots,a_6$ are constants,} and if we put
$$G_6(e,a)=H_6(e,f(e,a)),$$
then the shift operator for the shift $sh_j$ sends the solution space of $G_6(e,a)$ to that of $G_6(sh_j(e),a)$.

\par\smallskip 
In Section 8, we finally reach the equation $E_6(e)=G_6(e,0)$, which enjoys fruitful symmetries (e.g. adjoint, differentiation, the coordinate changes $x\to1/x, x\to1-x$, ...).

\par\smallskip
In Section 9, the shift operators of $H_5$ is given; they are derived from the shift operators $P_{\pm00}$ and $P_{0\pm0}$ of $H_6$. The S-values and reducibility conditions are given.  For the equation $H_4$, we find only one shift operator $\partial$ and its inverse, which is in Section 10. No shift operator is found for the equation $H_3$. 
\par\medskip
\red{The equations we treat in this paper and the paper \cite{HOSY2}:
$$\begin{array}{cl}
{\rm this\ paper}&\quad H_j,\ G_j,\ E_j,\quad (j=6,5,4,3),\ {\rm and}\ E_2,\\[2mm]
{\rm  \cite{HOSY2}}&\quad S\!E_j, \ (j=6,5,4,3),
  \end{array}$$
where $E_2$ is the Gauss hypergeometric equation.
They are mutually related as in the following figure 
$$\begin{array}{cccccccc}
 H_6&\longrightarrow &G_6&\longrightarrow&E_6&\longrightarrow& S\!E_6&\\
\downarrow&&\downarrow&&\downarrow&&\downarrow \\
 H_5&\longrightarrow &G_5&\longrightarrow&E_5&\longrightarrow& S\!E_5&\\
\downarrow&&\downarrow&&\downarrow&&\downarrow \\
 H_4&\longrightarrow &G_4&\longrightarrow&E_4&\longrightarrow& S\!E_4&\\
\downarrow&&\downarrow&&\downarrow&&\downarrow \\
 H_3&\longrightarrow &G_3&\longrightarrow&E_3&\longrightarrow& S\!E_3&
  \end{array}$$
}
  Horizontal arrows stand for specializations keeping the spectral type,
      and vertical lines for factorizations. Every equation 
      has one accessory parameter.

\par\medskip
{\it Acknowledgment}: We used the software Maple, especially {\sl DEtools \!}-package for multiplication and division of differential operators.  Interested readers may refer to our list of data written in text files of Maple format
\footnote{ http://www.math.kobe-u.ac.jp/OpenXM/Math/FDEdata}
for the differential equations and the shift operators treated in this document.
\par\smallskip\noindent
We thank T. Oshima and the referee for critical comments.
We also thank N. Takayama for instructing us about computer systems as well as various computational skills.

\par\smallskip\noindent
We previously submitted to a journal a long paper that contains most of the results in \red{this paper and the paper \cite{HOSY2}.}  Two referees gave us kind and useful comments. These helped us rewrite the paper to make the reasoning much clear and the structure straight. We deeply appreciate their kindness. To clarify the story, we divided the long paper into \red{two} relatively short ones: this paper and \red{the paper \cite{HOSY2}. }

   \newpage
\section{Equations $H_j,G_j,E_j\ (j=3,4,5,6)$ and $E_2$}\label{TheTen}
\secttoc
In this section, we introduce Fuchsian ordinary differential equations $H_j,G_j,E_j\ (j=3,4,5,6)$ of order $3,\dots,6$, with three singular points $\{0,1,\infty\}$.

When we are studying a differential operator $E$, we often call $E$ a differential equation and speak about the solutions without assigning an unknown.

The {\it Riemann scheme} of an equation  is the table of local exponents
at the singular points. The {\it Fuchs relation} says that
the sum of all the exponents equals
\begin{equation}\label{Fuchsrelation}
  \frac12n(n-1)(m-2),\end{equation}
where $n$ is the order of the equation,
and $m$ is the number of singular points;
for our equations, $m=3$.

When an equation $E\in\mathbb{C}[x][\partial]$ of order $n$ is written as
$$E=p_n\partial^n+\sum_{i=0}^{n-1}p_i\partial^i,$$
where
$$ p_n=x^{n_0}(x-1)^{n_1},\  p_i=\sum_jp_{ij}x^j\ (i=0,\dots,n-1),\ \partial=d/dx,$$
for some integers $n_0$ and $n_1$, we assume the coefficients $p_0,\dots, p_n$ have no common factor. $p_n\partial^n$ is often called the {\it head} of the equation.

A subset $ap$ of coefficients $\{p_{ij}\}$ is called a set of {\it accessory parameters}, if all other coefficients are uniquely written in terms of $ap$ and the local exponents. The choice of $ap$ is not unique, but the cardinality of $ap$ is unique, which is called the {\it number of accessory parameters}. For $H_j$, it is 1, and we choose one and call it {\it the} accessory parameter.

When an equation is determined uniquely by the local exponents,
it is said to be {\it free of accessory parameters} or {\it rigid}.

\subsection{Equation $H_6$}\label{E6}
%\setcounter{stc}{2}%\nostcrule  
%\nostcrule
%\secttoc
We present a Fuchsian differential equation $H_6$ of order 6 with 9 free
local exponents, with 3 singular points, and with the Riemann scheme
$$R_6:\left(\begin{array}{lcccccc}x=0:&0&1&2&e_1&e_2&e_3\\x=1:&0&1&2&e_4&e_5&e_6\\x=\infty:&s&s+1&s+2&e_7&e_8&e_9\end{array}\right),\quad e_1+\cdots+e_9+3s=6,
  $$
  with spectral type\footnote{Any solution at the three singular points has no logarithmic terms; this is often called the {\sl no-logarithmic condition} (refer to \S \ref{GenLoc}).} $(3111,3111,3111)$ and with generic local exponents $e=(e_1,\dots,e_9)$.
This is the main equation in this article.
 
 Any equation with Riemann scheme $R_6$ and with the said spectral type
 has the following expression 
 \begin{equation}\label{eqTxdx}T=p_6(x)\partial^6+\cdots+p_1(x)\partial+p_0\in\mathbb{C}[x][\partial],\end{equation}
   where
 \begin{equation}\label{coeffT}
   \setlength{\arraycolsep}{2pt}
   \def\arraystretch{1.1}
   \begin{array}{lcllcl}
p_{6}&=& x^3(x-1)^3,&
p_{5}&=& (p_{50}+p_{51}x)x^2(x-1)^2,\\
p_{4}&=& (p_{40}+p_{41}x+p_{42}x^2)x(x-1),\ &   
p_{3}&=& p_{30}+p_{31}x+p_{32}x^2+p_{33}x^3,\\
p_{2}&=& p_{20}+p_{21}x+p_{22}x^2,&             
p_{1}&=& p_{10}+p_{11}x,
 \end{array}     \setlength{\arraycolsep}{3pt}
   \end{equation}
refer to Proposition \ref{nologcond}. \blue{We call such an expression by use of polynomial coefficients of $x$ and the differentiation $\partial$, the $(x,\partial)$-form (refer to \S \ref{Genxz} for related expressions).}
The indicial polynomial at $x=0$ is given by 
 \[\rho(\rho-1)(\rho-2)\{(\rho-3)(\rho-4)(\rho-5)
    +(\rho-3)(\rho-4)p_{50} +(\rho-3)p_{40}+p_{30}\}.\]
So the coefficients $p_{50},p_{40}$ and $p_{30}$ are expressed
as polynomials of the local exponents $\{e_1,e_2,e_3\}$.
Do the same at $x=1$. Then we find that 
most \blue{of the} coefficients (as well as $p_{31}-p_{32}$) can be expressed
by the local exponents $e_1,\dots,e_9$, except the following {\it four} coefficients:
  $$p_{10},\ p_{20},\ p_{21},\ p_{32}.$$ 
We next examine the no-logarithmic condition at $\infty$. Applying $T$ to the expression
  $$u(x)=x^{-\rho}\sum_{m=0}^\infty u_mx^{-m},$$
we see that $Tu$ is expanded as
  $$ f(\rho)u_0x^{-\rho}+[f(\rho+1)u_1+g(\rho)u_0]x^{-\rho-1}
+[f(\rho+2)u_2+g(\rho+1)u_1+h(\rho)u_0]x^{-\rho-2}+\cdots,
$$
where
\begin{equation}\label{inditialeq} \begin{aligned}
  f(\rho)
  &= \rho(\rho+1)\cdots(\rho+5)
  -p_{51}\rho(\rho+1)\cdots(\rho+4) \\
  &\quad +p_{42}\rho(\rho+1)\cdots(\rho+3) 
  -p_{33}\rho(\rho+1)(\rho+2) \\
  &\quad +p_{22}\rho(\rho+1)
  -p_{11}\rho
  +p_0 \end{aligned} 
\end{equation}
  is the indicial polynomial at infinity and
\begin{equation}\label{equationgh}  \begin{aligned}  g(\rho)
  &=
  -3\rho(\rho+1)\cdots(\rho+5)
  -(p_{50}-2p_{51})\rho(\rho+1)\cdots(\rho+4)\\
  &\quad
  +(p_{41}-p_{42})\rho\cdots(\rho+3)
  -p_{32}\rho(\rho+1)(\rho+2) \\
  &\quad +p_{21}\rho(\rho+1)
  -p_{10}\rho,\\
  h(\rho)
  &=
  3\rho(\rho+1)\cdots(\rho+5)
  -(-2p_{50}+p_{51})\rho(\rho+1)\cdots(\rho+4)\\
  &\quad
  +(p_{40}-p_{41})\rho\cdots(\rho+3)
  -p_{31}\rho(\rho+1)(\rho+2) \\
  &\quad +p_{20}\rho(\rho+1).
 \end{aligned} \end{equation}
The local exponents at infinity, the roots of $f(\rho)$,
are  $s,s+1,s+2,$ and the other three are generic; in particular,
\begin{equation} \label{nonzero} f(s+k)\neq0\quad(k\geq3).
\end{equation}
When $\rho=s+2$,
$u_m$ $(m\geq1)$ is determined by the recurrence relation
$$
 f(s+2+m)u_m=F_m(u_0,u_1,\dots,u_{m-1}),
$$
for some function $F_m$, thanks to \eqref{nonzero}.
When $\rho=s+1$, the equation for $u_1$ becomes
$$
 f(s+2)u_1+g(s+1)u_0=0
$$
with $f(s+2)=0$.
Therefore we need $g(s+1)=0$.
Then $u_m$ $(m\geq2)$ is determined thanks to \eqref{nonzero}.
When $\rho=s$,
the equation for $u_1$ becomes
$$
 f(s+1)u_1+g(s)u_0=0
$$
with $f(s+1)=0$, and so we need $g(s)=0$.
Moreover the equation for $u_2$ becomes
$$
 f(s+2)u_2+g(s+1)u_1+h(s)u_0=0
$$
with $f(s+2)=g(s+1)=0$. So we need $h(s)=0$.
\par
Hence  the no-logarithmic condition is given by the {\it three} equations:
\begin{equation}\label{equationggh}g(s)=0,\quad g(s+1)=0,\quad h(s)=0\end{equation}
  for the {\it four} coefficients $p_{10}$, $p_{20}$, $p_{21}$, $p_{32}$.
  Hence, it remains one freedom of choice of the coefficients. So we get

\begin{prp}\label{accpar1}
  The differential equation with the Riemann scheme $R_6$ such that any local solution at 0 and 1 does not have logarithmic terms can be written as \eqref{eqTxdx} with \eqref{coeffT}. This equation has four free coefficients $\{p_{10},p_{20},p_{21},p_{32}\}$. Defining three polynomials $\{f,g,h\}$ by \eqref{inditialeq} and \eqref{equationgh}, the condition that any local solution at $\infty$ does not have logarithmic terms is given by the system of three equations \eqref{equationggh}.
\end{prp}

\begin{prp}\label{eqwithR6}Let
  \begin{eqnarray}\label{eqT}T=T_0(\theta)+T_1(\theta)\partial+T_2(\theta)\partial^2+T_3(\theta)\partial^3\in\mathbb{C}[\theta,\partial],\quad \theta=x\partial\end{eqnarray}
be an equation with Riemann scheme $R_6$ and with the spectral type $(3111,3111,3111)$. Then most \blue{of the} coefficients can be expressed in terms of the local exponents as
{\setlength\arraycolsep{2pt} \def\arraystretch{1.1}
  \begin{eqnarray} 
\quad T_0&=&(\theta+2+s)(\theta+1+s)(\theta+s)B_0,\quad B_0=(\theta+e_7)(\theta+e_8)(\theta+e_9), \label{eqT0}\\
\quad      T_1&=&(\theta+2+s)(\theta+1+s)B_1,\quad B_1=T_{13}\theta^3+T_{12}\theta^2+T_{11}\theta+T_{10}, \label{eqT1}\\
\quad    T_2&=&(\theta+2+s)B_2,\quad B_2=T_{23}\theta^3+T_{22}\theta^2+T_{21}\theta+T_{20}, \label{eqT2}\\
\quad    T_3&=&(-\theta-3+e_1)(-\theta-3+e_2)(-\theta-3+e_3), \label{eqT3}
  \end{eqnarray}
  }
where
\[\def\arraystretch{1.2}\setlength\arraycolsep{3pt} \begin{array}{rcl}
  T_{13}&=&-3,\quad T_{23}=3,\quad  T_{12}=-9+s_{11}-2s_{13},\quad  T_{22}=18+s_{13}-2s_{11},\\
T_{11}&=&-8+(s_{11}^2+2s_{11}s_{13}-s_{12}^2+s_{13}^2)/3+s_{11}-5s_{13}-s_{21}+s_{22}-2s_{23},\\
T_{21}&=&35+(-s_{11}^2-2s_{11}s_{13}+s_{12}^2-s_{13}^2)/3-7s_{11}+5s_{13}+2s_{21}-s_{22}+s_{23},\\
T_{20}&=&-T_{10}+19+(s_{11}^2s_{13}-s_{11}s_{12}^2+s_{11}s_{13}^2-s_{12}^2s_{13})/9+(s_{13}^3+s_{11}^3-2s_{12}^3)/27\\
&&+(-2s_{11}^2-4s_{11}s_{13}+s_{11}s_{22}+2s_{12}^2+s_{22}s_{12}-2s_{13}^2+s_{22}s_{13})/3\\
&&
-5s_{11}+4s_{13}+3s_{21}-2s_{22}-s_{31}-s_{32}-s_{33},
\end{array}
\]
except $T_{10}$, which does not affect the local exponents. In this sense, we call this coefficient the {\rm accessory parameter}.
Here $s_*$ are symmetric polynomials of the local exponents:
\def\arraystretch{1.1}\setlength\arraycolsep{1pt}
\begin{equation}\label{e19tos}\begin{array}{rl}      %%# new linebreak
s_{11}&=e_1+e_2+e_3,\quad s_{12}=e_4+e_5+e_6,\quad s_{13}=e_7+e_8+e_9,\\
s_{21}&=e_1e_2+e_1e_3+e_2e_3,\quad s_{22}=e_4e_5+e_4e_6+e_5e_6,\\ 
s_{23}&=e_7e_8+e_7e_9+e_8e_9,\quad s_{31}=e_1e_2e_3,\quad s_{32}=e_4e_5e_6,\\ 
s_{33}& = e_7e_8e_9,\quad s=-(s_{11}+s_{12}+s_{13}-6)/3.
\end{array}\end{equation}
\setlength\arraycolsep{3pt}
\end{prp}
\begin{dfn}This equation (\ref{eqT}) is denoted by
  $$H_6=H_6(e,T_{10}),\quad e=(e_1,\dots,e_9).$$
  \end{dfn}
%\begin{proof} . %  
\subsection{Proof of Proposition \ref{eqwithR6}}
  Since the above operator \eqref{eqTxdx}: $T=x^3(x-1)^3\partial^6+\cdots$
  can be expressed in $(\theta,\partial)$-form,
  we write this equation as \eqref{eqT}: $T=T_0+T_1\partial+\cdots$.
  Since the head (top order term) of $T$ is
$$p_6\red{(x)}\partial^6=x^3(x-1)^3\partial^6=x^6\red{\partial^6}-3x^5\partial^5\ \partial+3x^4\partial^4\ \partial^2-x^3\partial^3\ \partial^3,$$
and $x^i\partial^i=\theta(\theta-1)\cdots(\theta-i+1)$, the terms  $T_0$ and $T_3$ are determined by local exponents at $x=\infty$
  and at $x=0$, as \eqref{eqT0} and \eqref{eqT3},
  thanks to Propositions \ref{localat0} and \ref{localat1}. In addition we have $$T_{13}=-3,\quad T_{23}=3.$$
  We could then \red{substitute these} into $(x,\partial)$-form
  $p_6(x)\partial^6+\cdots$,
  and follow the recipe in Proposition \ref{accpar1}.
  Instead,  we make a coordinate change $x\to 1/x$ to this equation.
  Perform the transformation $x=1/y,w=y\partial_y, \partial_y=d/dy$
  to \eqref{eqT}: 
$$T|_{x=1/y}=T_0(-w)-T_1(-w)yw+T_2(-w)y^2(w+1)w-T_3(-w)y^3(w+2)(w+1)w.$$
Multiply $y^{s}$ from the right, and $y^{-s}$ from the left:
\[ \def\arraystretch{1.2}\begin{array}{l}
  T_0(-(w+s))-T_1(-(w+s))y(w+s)+T_2(-(w+s))y^2(w+1+s)(w+s)\\
  -T_3(-(w+s))y^3(w+2+s)(w+1+s)(w+s);
\end{array}\]
Multiply $y^{-3}$ from the left:
\begin{equation}\label{atinfty}
  \def\arraystretch{1.2}
  \begin{array}{l}T_0(-(w+s+3))y^{-3}-T_1(-(w+s+3))y^{-2}(w+s)\\
\quad    +T_2(-(w+s+3))y^{-1}\times(w+1+s)(w+s)\\
\quad    -T_3(-(w+s+3))(w+2+s)(w+1+s)(w+s).
\end{array}\end{equation}
The first term is 
\[ \def\arraystretch{1.2} \begin{array}{ll}
  &(-(w+s+3)+s)(-(w+s+3)+s+1)(-(w+s+3)+s+2)\\
  &\quad \times(-(w+s+3)+e_7)(-(w+s+3)+e_8)(-(w+s+3)+e_9)y^{-3}\\
&=(w+3)(w+2)(w+1)(w+s+2-e_7)(w+s+2-e_8)(w+s+2-e_9)y^{-3}\\
  &=(w+s+2-e_7)(w+s+2-e_8)(w+s+2-e_9)\partial_y^3,
\end{array}\]
(by $\partial_y^3=(w+1)(w+2)(w+3)y^{-3})$ the last term is 
\[ \def\arraystretch{1.2} \begin{array}{ll}
  &(-(w+s+3)+3-e_1)(-(w+s+3)+3-e_2)(-(w+s+3)+3-e_3)\\
  &\times(w+s)(w+s+1)(w+s+2)\\
  &=(w+s)(w+s+1)(w+s+2)(w+e_1+s)(w+e_2+s)(w+e_3+s),
\end{array}\]
and the second term is (polynomial of $\theta$)$y^{-2}$
and the third term is (polynomial of $\theta$)$y^{-1}$;
these must be polynomials of $(w,\partial_y)$.
Since
$$\partial_y^2=(w+1)(w+2)y^{-2}\quad {\rm and}\quad \partial_y=(w+1)y^{-1},
$$
$(w+1)(w+2)$ divides $T_1(-(w+s+3))$, and $(w+1)$ divides $T_2(-(w+s+3)),$
that is,
  $$(\theta+2+s)(\theta+1+s)\,|\,T_1(\theta),\quad {\rm and}\quad (\theta+2+s)\,|\,T_2(\theta).$$
Now we are ready.
We put $T_1(\theta)$ and $T_2(\theta)$ as in \eqref{eqT1} and \eqref{eqT2},
and transform it to $(x,\partial)$-form: $T=p_6\partial^6+p_5\partial^5+\cdots$.
We have
$$p_6=x^3(x^3+T_{13}x^2+T_{23}x-1),
  \quad p_5=x^2((e_7+e_8+e_9+3s+18)+\cdots),\dots$$
  All the coefficients $p_{ij}$ are expressed in terms of
  $$e_1,\dots, e_9,s,\quad T_{13},T_{12},T_{11},T_{10}\ (=p_{10}),
  \  T_{23},T_{22},T_{21},T_{20}\ (=p_{20}),$$
where $s=(6-e_1-\cdots-e_9)/3$. 

\begin{itemize}[leftmargin=30pt]%[leftmargin=*]
\item As we saw already, $T_{13}=-3,\ T_{23}=3$.
\item $x^2(x-1)^2\,|\,p_5$ leads to
  $$\begin{array}{ll}T_{12} &=e_1+e_2 +e_3 -2e_7 -2e_8 -2e_9 -9=s_{11}-2s_{13}-9,\\
  T_{22} &= -2e_1 -2e_2 -2e_3 +e_7 +e_8 +e_9 +18=-2s_{11}+s_{13}+18.\end{array}$$

\item $x(x-1)\,|\,p_4$ leads to
  $$T_{11}+  T_{21}=s_{21}-s_{23}-6s_{11}+27.$$%e_1e_2 + e_1e_3 + e_2e_3 - e_7e_8 - e_7e_9 - e_8e_9- 6e_1 - 6e_2 - 6e_3 + 27.$$

\item The requirement that local exponents at $x=1$ are $\{e_4,x_5,e_6\}$ is equivalent to the system 
  $$\begin{array}{l}T_{11}+3s^2 - (-2s_{11} - 2s_{13} + 12)s
      - 5s_{11} + s_{13} + s_{21} + 2s_{23} + 20=0,\\
  T_{10} + T_{20}+s^3 + (s_{11} + s_{13} - 6)s^2 - (-T_{11} + 5s_{11} - s_{13} \\
  \quad  - s_{21} - 2s_{23} - 20)s + s_{32} + s_{33} + 9s_{11} - 3s_{21}
  + s_{31} - 27=0.\end{array}$$
\end{itemize}
Thus $T_{13},\ T_{12},\ T_{11},\ T_{23},\ T_{22},\ T_{21},$
and $T_{10}+T_{20}$ are expressed by the local exponents.\par
%\end{proof} % 
%\hfill$\square$ 

\subsection{Table of equations $H_j\ (j=6,5,4,3)$ and $E_2$}\label{TheTenTable}
We {\it always assume}, for $H_j$,  that local solutions
have no logarithmic term, and the exponents $e_1$, $e_2$, $\dots$ are generic. $R_j$ denotes the Riemann scheme of $H_j$.

We tabulate the equations $H_j\ (j=6,5,4,3)$: they are related to $H_6$ via addition-middle-convolutions and restrictions (see \S \ref{fromH3toH654bymc}, \ref{fromH654toH3bymc} and \ref{fromH6toH53byfac}).
\begin{itemize}[leftmargin=*] \setlength{\itemsep}{5pt}
\item$H_6=H_6(e,T_{10}),\qquad e=(e_1, \dots, e_9)$
  $$\begin{array}{l}=x^3(x-1)^3\partial^6+x^2(x-1)^2P_1\partial^5+x(x-1)P_2\partial^4+P_3\partial^3+P_2\partial^2+P_1\partial+P_0,\\[2mm]
     = T_0+T_1\partial+T_2\partial^2+T_3\partial^3,\quad \theta=x\partial,\end{array}$$
  \[R_6:\left(\begin{array}{lcccccc}
    x=0:&0&1&2&e_1&e_2&e_3\\
    x=1:&0&1&2&e_4&e_5&e_6\\
    x=\infty:&s&s+1&s+2&e_7&e_8&e_9\end{array}\right),
    \quad  s=(6-e_1-\cdots-e_9)/3,
    \]
    where $P_j$ \blue{is used symbolically} for a polynomial of degree $j$ in $x$,    and
\[\def\arraystretch{1.1}
\begin{array}{lcl}
 T_0&=&(\theta+s+2)(\theta+s+1)(\theta+s)B_0,\quad B_0=(\theta+e_7)(\theta+e_8)(\theta+e_9),\\
 T_1&=&(\theta+s+2)(\theta+s+1)B_1,\quad B_1=T_{13}\theta^3+T_{12}\theta^2+T_{11}\theta+T_{10},\\
 T_2&=&(\theta+s+2)B_2,\quad B_2=T_{23}\theta^3+T_{22}\theta^2+T_{21}\theta+T_{20},\\
 T_3&=&-(\theta+3-e_1)(\theta+3-e_2)(\theta+3-e_3),
 \end{array}
 \] %
 where $T_{13},T_{12},T_{11},T_{23},T_{22},T_{21}$ and $T_{20}+T_{10}$ are polynomials in $e_1$, $\dots$, $e_9$; they are given in Proposition \ref{eqwithR6}. We choose $T_{10}$ as the accessory parameter. Spectral type (3111,3111,3111).

\item $H_5=H_5(e_1,\dots,e_8,B_{510})$
  $$\begin{array}{l}=x^3(x-1)^3\partial^5+x^2(x-1)^2P_1\partial^4+x(x-1)P_2\partial^3+P_3\partial^2+P_2\partial+P_1\\[2mm]
  =x\overline{T}_0+\overline{T}_1+\overline{T}_2\partial +\overline{T}_3\partial ^2,\end{array}$$
where $P_j$ \blue{is used symbolically} for a polynomial of degree $j$ in $x$,     
\[R_5: \left(\begin{array}{ccccc}
    0&1&e_1-1&e_2-1&e_3-1\\
    0&1&e_4-1&e_5-1&e_6-1\\ 
    1+s&2+s&3+s&e_7+1&e_8+1\end{array}\right),
    \qquad s=(6-e_1-\cdots-e_8)/3,
    \]
\[\def\arraystretch{1.1} \begin{array}{rcl}
   \overline{T}_0 &=& (\theta+s+1)(\theta+s+2)(\theta+s+3)B_{50},\quad B_{50}=B_0(\theta=\theta+1, e_9=0), \\
   \overline{T}_1 &=& (\theta+s+1)(\theta+s+2)B_{51},\quad B_{51}:=B_1(e_9=0),\\
   \overline{T}_2 &=& (\theta+s+2)B_{52},\quad B_{52}:=B_2(e_9=0), \\
   \overline{T}_3 &=& -(\theta+3-e_1)(\theta+3-e_2)(\theta+3-e_3).
\end{array} \]
This is obtained from $H_6$ by putting $e_9=0$,
and dividing from the right by $\partial$. The accessory parameter %$B_{510}$ 
is the constant term $B_{510}$ of the polynomial $B_{51}$ in $\theta$. Spectral type (2111,2111,311).
\item $H_4=H_4(c_1,\dots,c_7,\mathcal T_{10})$
  $$\begin{array}{l}
  =x^2(x-1)^2\partial^4+x(x-1)P_1\partial^3+P_2\partial^2+P_1\partial+P_0,\\[2mm]
  =\mathcal T_0+\mathcal T_1\partial+\mathcal T_2\partial^2,\end{array}$$
  where $P_j$ \blue{is used symbolically} for a polynomial of degree $j$ in $x$,
\[R_4:\left(\begin{array}{llcll}x=0:&0&1&c_1&c_2\\ x=1:&0&1&c_3&c_4\\ x=\infty:&c_8&c_5&c_6&c_7\end{array}\right),\quad c_1+\cdots+c_8=4,\]
 \[\def\arraystretch{1.1} \begin{array}{rcl}
    \mathcal T_0&=&(\theta+c_5)(\theta+c_6)(\theta+c_7)(\theta+c_8),\\
    \mathcal T_1&=&-2\theta^3+\mathcal T_{12}\theta^2+\mathcal T_{11}\theta+\mathcal T_{10},\\
    \mathcal T_{12}&=&c_1+c_2-c_5-c_6-c_7-c_8-5,\\
    \mathcal T_{11}&=&3(c_1+c_2)-c_1c_2+c_3c_4-c_5c_6-c_5c_7-c_5c_8-c_6c_7-c_6c_8-c_7c_8-8,\\
    \mathcal T_2&=&(\theta-c_1+2)(\theta-c_2+2),
 \end{array}
 \]
 where $\mathcal T_{10}$ is the accessory parameter. Spectral type (211,211,1111).

\item $H_3=H_3(b_1,\dots,b_6,a_{00})$
  $$\begin{array}{l}=x^2(x-1)^2\partial^3+x(x-1)P_1\partial^2+P_2\partial+P_1\\[2mm]
  =xS_{-1}+S_0+S_1\partial,\end{array}$$
where $P_j$ \blue{is used symbolically} for a polynomial of degree $j$ in $x$,
  \[R_3:\left(\begin{array}{llll}x=0:&0&b_1&b_2\\ x=1:&0&b_3&b_4\\ x=\infty:&b_7&b_5&b_6\end{array}\right),\quad b_1+\cdots+b_7=3,\]
\[\def\arraystretch{1.1}\setlength\arraycolsep{2pt} \begin{array}{rcl}
    S_{-1}&=&(\theta+b_5)(\theta+b_6)(\theta+b_7),\\
    S_0&=& -2\theta^3+(2b_1+2b_2+b_3+b_4-3)\theta^2 \\
    &&\qquad +(-b_1b_2+(b_3-1)(b_4-1)-b_5 b_6 - b_5 b_7 - b_6 b_7)\theta+a_{00},\\
    S_1&=&(\theta-b_1+1)(\theta-b_2+1),
  \end{array}
  \]
  where $a_{00}$ is the accessory parameter. Spectral type (111,111,111).
\item $E_2=E_2(a_1,a_2,a_3)=E(a,b,c)$ (the Gauss hypergeometric equation)
$$\begin{array}{l}
=x(x-1)\partial^2+((a+b+1)x-c)\partial +ab\\
=(\theta+a)(\theta+b)-(\theta+c)\partial   \end{array}$$   
  \[
     R_2=\left(\begin{array}{lll}x=0:&0&a_1\\
     x=1:&0&a_2\\
     x=\infty:&a_3&a_4\end{array}\right)=
         \left(\begin{array}{llc}x=0:&0&1-c\\
       x=1:&0&c-a-b\\
       \xi=\infty:&a&b\end{array}\right)= R_{abc},      
       \]
       where $a_1+\cdots+a_4=1.$ This equation is rigid. Spectral type (11,11,11).
\end{itemize}
Summing up,
 \par\medskip
\[\begin{array}{cccccc  }
{\rm name\ of\ the\ equation}  \ &\ H_6\ &\ H_5\ &\ H_4\ &\ H_3\ &\qquad E_2 \\[1mm]
{\rm order\ of\ the\ equation} &6  &5  &4  &3 &\qquad2\\
{\rm number\ of\ the\ free\ local\ exponents}&9&8 &7  &6 &\qquad3\\
{\rm number\ of\ accessory\ parameters} &\ 1\ &\ 1\  &\ 1\   &\ 1\  &\qquad 0
\end{array}
\]
\addtolength{\arraycolsep}{1pt}   %%# temp change of \arraycolsep
%\begin{remark}As in Proposition 1.2, $H_5$ is the equation with the Riemann scheme $R_5$ with the spectral type (2111,2111,311); and $H_4$ with $R_4$ and with $(211,211,1111)$.\end{remark}
\subsection{Equations $G_j, E_j\ (j=6,5,4,3)$}
Each of the equations $H_j\ (j=6,5,4,3)$ has one accessory parameter.
The equations $G_j, E_j$ are equations $H_j$ with a specified cubic polynomials of the local exponents $e=(e_1,e_2,\dots)$ as the accessory parameter.
\subsubsection{$G_6(e,a)$}
The accessory parameter of $H_6$ is denoted by $T_{10}$. The equation $G_6$ is $H_6$ with a specific cubic polynomial $T_{10}(e)$ of $e$ as $T_{10}$. This polynomial is determined roughly as follows:
If the equation $G_6$ admits  shift operators for the  block shifts
   $$sh_j:(e_j,e_{j+1},e_{j+2},s)\to(e_j+1,e_{j+1}+1,e_{j+2}+1,s-1)\quad (j=1,4,7),$$ then $T_{10}(e)$ must be 
  $$T_{10}(e)=S_{10}+R,\quad R=a_0+a_1t_{21}+a_2t_{22}+a_3t_{23}+a_4t_{31}+a_5t_{32}+a_6t_{33},$$
   where $S_{10}$ and $t_{ij}$ are cubic polynomials in $e$ defined in Theorem \ref{shiftopG6} and Corollary \ref{t2t3bis}, and $a_0,\dots,a_6$ are free constants. We denote the equation with the above $T_{10}$ by $G_6(e,a)$.
   \subsubsection{$G_j(e,a)\ (j=3,4,5)$}
 \blue{  The equation $H_3$ is obtained from $H_6$ by middle convolution (\S \ref{fromH6toH3}). The equations $H_4$ and $H_5$ are obtained from $H_3$ by addition and middle convolution (\S \ref{fromH3toH654bymc}). We follow these procedure starting from $G_6(e,a)$ and get $G_3(e,a),G_4(e,a)$ and $G_5(e,a)$.}
   \subsubsection{$E_j(e)\ (j=6,5,4,3)$} As the most symmetric equation, $E_6(e)$ is defined as $G_6(e,0)$. Equations $E_3(e),E_4(e)$\footnote{The differential equation $Z(A)$ found and  studied in \cite{Z12} is a codimension-2 specialization of $E_4(e)$.}  and $E_5(e)$ are $G_3(e,0),G_4(e,0)$ and $G_5(e,0)$, respectively.

\newpage
\adjuststc
\section{Generalities}\label{Gen}%\nostcrule 
\setcounter{stc}{3}
\secttoc
In this section, we prepare tools that we need to study
our equations in the following sections and the following papers.

\subsection{ Symmetry}\label{GenSym}
In this subsection, $H(e,ap)\in\mathbb{C}[x][\partial]$ denotes a differential equation with free local exponents $e$ and accessory parameters $ap$, 
$G(e,a)\in\mathbb{C}[x][\partial]$ a differential equation with local exponents $e$ and accessory parameters $ap$ assigned as polynomials of $e$ with a set $a$ of parameters, and
$E(e)\in\mathbb{C}[x][\partial]$ a differential equation with local exponents $e$ where accessory parameters are assigned as polynomials of $e$. % without parameters.    
Examples are
$$H_j,\quad G_j,\quad E_j\quad (j=3,4,5,6).$$
\subsubsection{Shift symmetry}
For a shift (Definition \ref{DefShift})  $sh(e)$ of free (generic) local exponents $e$ of a differential equation, if a non-zero differential operator $P\in\mathbb{C}(x)[\partial]$ sends
\begin{itemize}
     \item    solutions of $H(e,ap)$  to those of $H(sh(e),ap')$, for some $ap'$, 
     \item   solutions of $G(e,a)$  to those of $G(sh(e),a)$,
 \item   sends solutions of $E(e)$  to those of $E(sh(e))$,
\end{itemize}
     the operator $P$ is called a {\it shift operator} (Definition \ref{DefShift})  for the shift $sh(e):e_i\to e_i+n_i\ (n_i\in\mathbb{Z})$. The equation with such a property is said to be symmetric with respect to the shift $sh(e)$. 
\subsubsection{Differentiation, adjoint and coordinate change }
   If derivatives of solutions satisfy  the same equation, with some change of 
     \begin{itemize}
     \item the local exponents $e$ and the accessory parameters $ap$, for $H(e,ap)$, 
     \item the local exponents $e$ and the parameters  $a$, for $G(e,a)$,
\item  the local exponents $e$, for $E(e)$,
     \end{itemize}
     the equation is said to enjoy differentiation symmetry. %This is a kind of shift symmetry.

     If the adjoint equation (defined in \S \ref{GenAdj}) of an equation remains the same, with some change of the exponents and the parameters as itemized above,  the equation is said to enjoy adjoint symmetry.

 If an equation after a coordinate change of $x$, remains the same, with some change of the exponents and the parameters as itemized  above, 
        the equation is said to be symmetric relative to this transformation.

\subsubsection{Symmetries of $H_j,G_j,E_j$}
   We tabulate the symmetries that $K_j=\{H_j,G_j,E_j\}$ enjoy (Y=yes, N=no):
   \[ \def\arraystretch{1.2}
   \begin{array}{ccccccc}
   \text{Symmetry }&\ &\ K_{6}\ &\ K_{5}\ &\ K_{4}\ &\ K_3\ &\qquad E_2 \\
{\rm Shift\ operators} &\ \ &{\rm Y}&{\rm Y}&{\rm Y}&{\rm N'}&\qquad{\rm Y}\\
{\rm Differentiation} &\ &{\rm Y}&{\rm N}&{\rm Y}&{\rm N}&\qquad{\rm Y}\\
{\rm Adjoint} &\ \ &{\rm Y}&{\rm Y}&{\rm Y}&{\rm Y}&\qquad{\rm Y}\\
x\to1/x \ &\   &{\rm Y}&{\rm N}&{\rm N}&{\rm Y}&\qquad{\rm Y}\\
x\to1-x \  &\ &{\rm Y}&{\rm Y}&{\rm Y}&{\rm Y}&\qquad{\rm Y}
\end{array}
\]
where N' stands for `no shift operator is found to the authors'.
   \subsubsection{Examples }\label{cochH6}%coch-H6-1,coch-H6-2
   \begin{itemize}
     \item Adjoint of $H_3(e,a_{00})$ is $H_3(-e_1,\dots,-e_4,2-e_5,2-e_6,a'_{00})$, where %$a'_{00}$ is 
$$\begin{array}{ll}a'_{00}&=-e_1 e_2 +(e_1+e_2+e_3+e_4)(e_5+e_6-2)\\&+(e_5-1)^2+(e_6-1)^2+(e_5-1)(e_6-1)-1-a_{00}
%a'_{00}&=-e_1e_2 + e_1e_5 + e_1e_6 + e_2e_5 + e_2e_6 + e_3e_4 + e_3e_5 + e_3e_6 + e_4e_5\\ &+ e_4e_6 + e_5^2 + e_5e_6 + e_6^2- 2e_1 - 2e_2 - 2e_3 - 2e_4 - 3e_5 - 3e_6 + 2-a_{00}.
\end{array}$$
%$$-e1*e2 + e1*e5 + e1*e6 + e2*e5 + e2*e6 + e3*e4 + e3*e5 + e3*e6 + e4*e5 + e4*e6 + e5^2 + e5*e6 + e6^2- 2*e1 - 2*e2 - 2*e3 - 2*e4 - 3*e5 - 3*e6 + 2-a00,$$
\item Adjoint of $H_6(e,T_{10})$ is $H_6(2-e_1,\dots,2-e_6,1-e_7,1-e_8,1-e_9, T'_{10})$, where %$T'_{10}$ is
$$
%T'_{10} = -(s_{13}+s_{23}+1) s -3 (s_{23}+s_{33}+1) -T_{10}.
T'_{10}=6s^2 + (4s_{12} - 18)s - 6s_{12} - 2s_{21} + 2s_{22} - 4s_{23} + 8 - T_{10}.
$$
\item Coordinate change $x\to1-x$ of $H_6$:
     $$H_6({\bm e}_1,{\bm e}_4,{\bm e}_7,T_{10})|_{x\to1-x}=H_6({\bm e}_4,{\bm e}_1,{\bm e}_7,T'_{10}),$$
   where ${\bm e}_1=\{e_1,e_2,e_3\},{\bm e}_4=\{e_4,e_5,e_6\},{\bm e}_7=\{e_7,e_8,e_9\}$,
   $$T'_{10}=3s^2 + (s_{11} + s_{12} - s_{23} + 2)s + 3s_{11} + 3s_{12} - 3s_{23} - 3s_{33} - 21-T_{10}.$$
 \item Coordinate change $x\to1/x$ of $H_6$:   $$x^{-s-3}\circ H_6({\bm e}_1,{\bm e}_4,{\bm e}_7,T_{10})|_{x\to1/x}\circ x^s=H_6({\bm e}_7-s{\bm1},{\bm e}_4,{\bm e}_1+s{\bm1},T'_{10}),$$
 %  $$x^{r-3}G_6({\bm e},a)|_{x\to1/x}\circ x^{-r}=G_6({\bm e}_7-s{\bm1},{\bm e}_4,{\bm e}_1+s{\bm1},-a_0,-a_3,-a_2,-a_1,-a_6,-a_5,-a_4),$$
   where
   $$\begin{array}{ll}T'_{10}&=4s^3 + (3s_{11} + 9)s^2 + (6s_{11} - s_{12} + 2s_{21} + s_{23} + 8)s + s_{33} + 6s_{12} + 3s_{21} \\
     &- 3s_{22}+ 3s_{23} + s_{31} + s_{32} - 3+T_{10}.\end{array}$$
   \end{itemize}
Here $H_6|_{x\to1-x}$ and $H_6|_{x\to1/x}$ are $H_6$ after the coordinate changes $x\to1-x$ and $x\to1/x$, respectively.

   \subsection{$(\theta,\partial)$-form and $(x,\theta,\partial)$-form}\label{Genxz}
Given a differential operator $P=a_n(x)\partial^n+\cdots\in\mathbb{C}[x][\partial]$
of order $n$ in $(x,\partial)$-form.
Rewrite each term as
  $$x^i\partial^j=x^{i-j}(x^j\partial^j),\quad i\ge j,\qquad
x^i\partial^j=(x^i\partial^i)\partial^{j-i},\quad i\le j,
    $$
and substitute 
  $$x^i\partial^i=\theta(\theta-1)\cdots(\theta-i+1),\quad i\ge1,\quad \theta=x\partial.$$
Then we have
\begin{prp}\label{xzdxexpression}
    Any differential operator $P=a_n(x)\partial^n + \cdots\in\mathbb{C}[x][\partial]$  %%# added linebreak
  of order $n$ can be written uniquely as
  $$P=x^qP_{-q}(\theta)+\cdots+xP_{-1}(\theta)+P_0(\theta)+P_1(\theta)\partial+\cdots+P_p(\theta)\partial^p,\ \ p\le n,\ q\ge0,$$
  where $P_*$ is a polynomial in $\theta$ of degree as follows:
  $$\deg (P_{-q})\le n, \dots, \deg (P_0)\le n,\quad \deg(P_1)
    \le n-1,\dots, \deg (P_p)\le n-p.$$
This expression is called the $(x,\theta,\partial)$-{\rm form} of $P$.
\end{prp}
When $q=0$, the equation has a $(\theta,\partial)$-form.
% Referring to the table of equations in \S \ref{TheTenTable}, we see that
\[ \def\arraystretch{1.2}
   \begin{array}{ccccccc}
     {\rm equation}&& H_{6}&\ H_{5}&\ H_{4}&\ H_{3}&\qquad E_2 \\
     p&&3&2&2&1&\qquad1\\
     q&&0&1&0&1&\qquad0
     \end{array}\]
             
\subsubsection{Local exponents at 0 and $\infty$}\label{GenxzLoc}
Given an operator $P=x^qP_{-q}(\theta)+\cdots+P_p(\theta)\partial^p$
of $(x,\theta,\partial)$-form. Assume
$$p,\ q\ge0, \qquad P_{-q},\ P_p\not=0.$$
Applying $P$ to a local solution
around $x=0$: $u=x^\rho(1+\cdots)$,
we see only the last term is effective to compute local exponents:
$$P_p(\theta)\,\partial^p\, u=\rho(\rho-1)\cdots(\rho-p+1)P_p(\rho-p)x^{\rho-p}(1+\cdots).$$

\begin{prp}\label{localat0}The local exponents at $x=0$ are
      $0$, $1$, $\dots$, $p-1$ and the roots of $P_p(\rho-p)$.
\end{prp}
At $x=\infty$, perform the change $x=1/y, w=y\partial_{y}$,
and use the formulae
$$\partial=-yw,\quad \partial^2=y^2w(w+1),
\quad \partial^3=-y^3w(w+1)(w+2).\dots
$$
Then $P$ changes into
$$y^{-q}P_{-q}(-w)+\cdots+P_0(-w)-P_1(-w)yw+P_2(-w)y^2w(w+1)+\cdots.
$$
Applying this to a local solution around $y=0$: $v=y^\rho(1+\cdots)$,
we see only the first term is effective:
$$y^{-q}P_{-q}(-w)\ v=y^{-q}P_{-q}(-\rho)y^{\rho}(1+\cdots).
$$

\begin{prp}\label{localat1}The local exponents at $x=\infty$
  are the roots of $P_{-q}(-\rho)$.
\end{prp}
This means that the first and the last terms of the expression
$P=x^qP_{-q}+\cdots+P_p\partial^p$ are determined,
up to multiplicative constants,
by the local exponents at $\infty$ and $x=0$, respectively.

For example, for $H_6$, the first term is
$$(\theta+s+2)(\theta+s+1)(\theta+s)(\theta+e_7)(\theta+e_8)(\theta+e_9),$$
and the last term is
$$-(\theta+3-e_1)(\theta+3-e_2)(\theta+3-e_3).$$
\subsection{Spectral type and the number of accessory parameters}\label{GenLoc}
In this section, the spectral type of an equation at a singular point, which characterizes local behavior of solutions at the singular point, is introduced. The set of spectral types of a Fuchsian differential equation determines the number of accessory parameters.
\begin{dfn}Consider a Fuchsian differential equation $P$ of order $n$.
  Suppose at a singular point,  the local exponents are given
  as $\{s,\  s+1,\  \dots,\  s+r-1,\  e_1,\  \dots,\ e_{n-r}\}$, %%# linebreak
  where $s, e_1,\dots,e_{n-r}$ are {\it generic}, %(up to the Fuchs relation
and the local solutions do not have logarithmic   terms ({\it i.e.}, local monodromy is semi-simple).
  In this case, we say the singular point has
  the {\it spectral type} $r1\dots1$.
For the spectral type in a more general situation, see \cite{Osh,Hara}.\end{dfn}

For example, the equations $H_6$ and the Gauss equation $E_2$ have spectral types 3111 and 11 at the three singular points, respectively. They are written as
  $$(3111,\ 3111,\ 3111)\quad{\rm and}\quad (11,\ 11,\ 11),$$
respectively.
The following is well-known (e.g. \cite{Heffter}, Satz II): 
\begin{prp}\label{nologcond}
  Let $P$ be a differential operator which is  regular singular at $x=0$:
  $$P=x^n\partial^{n}+x^{n-1}p_{n-1}\partial^{n-1}+\cdots+xp_1\partial+p_0,$$
  where $p_j$ are holomorphic at $x=0$. If the local exponents at $x=0$
  are $\{0,1,\dots,r-1,$ $e_1,\dots,e_{n-r}\}$ $(r=1,\ldots, n)$ 
%and $\{e_1,\dots,e_{n-r}\}$  are generic, 
then
$$p_j(0)=0,\quad j=0,\dots,r-1.$$
  Moreover, if the local solutions do not have logarithmic terms, {\it i.e.},
  if the spectral type at $x=0$ is $r1\dots1$, then
  $$x^2|p_{r-2},\ \dots,\ x^{r-1}|p_1,\ \ x^{r}|p_0.$$
Note that $p_{r-1}(0)=0$ implies $x|p_{r-1}$.

\end{prp}
\comment{
\begin{proof} % \noindent {\sl Proof.}
For notational simplicity we let $n=6$ and $r=2$. Set
  $$p_j=p_{j0}+p_{j1}x+p_{j2}x^2+\cdots,\quad u=u_0+u_1x+u_2x^2+\cdots$$
Then 
\[ \def\arraystretch{1.1}
\begin{array}{c}
Pu=(p_{00}+p_{01}x+p_{02}x^2+p_{03}x^3+\cdots)(u_0+u_1x+u_2x^2+u_3x^3+\cdots)\\
+x(p_{10}+p_{11}x+p_{12}x^2+\cdots)(u_1+2u_2x+3u_3x^2+\cdots)\\
+x^2(p_{20}+p_{21}x+\cdots)(2u_2+3!u_3x+\cdots)\\
+x^3(p_{30}+\cdots)(3!u_3+\cdots)+\cdots,
\end{array}\]
where
  $$p_{00}=p_{10}=p_{20}=0.$$
The solutions have no logarithmic term  if and only if for arbitrary $u_0,u_1$
and $u_2$, the coefficients $u_3,u_4,\dots$ are uniquely determined
by $Pu=0$.
The coefficient of $x$ is $p_{01}u_0$,
that of $x^2$ is $p_{01}u_1+p_{02}u_0+p_{11}u_1$. Thus we have
  $$p_{01}=p_{02}=p_{11}=0.$$
Since the genericity of the local exponents asserts $p_{30}\not=0$, the vanishing of the coefficient of $x^3$ determines $u_3$
as a linear form of $\{u_0,u_1,u_2\}$, and so on.
\end{proof} % \hfill$\square$ % \bigbreak 
\begin{cor}\label{nologcondcor}
  Let $P$ be as above. Set
  $$p_r=xq_r,\  p_{r-1}=x^2q_{r-1},\ \dots,\ p_1=x^rq_1,\ p_0=x^{r+1}q_0,$$
  where $q_0,\dots,q_r$ are holomorphic at $x=0$. Then $P$ has
  the following expression:
  $$x^{-r-1}P=x^{n-r-1}\partial^n+x^{n-r-2}p_{n-1}\partial^{n-1}+\cdots+p_{r+1}\partial^{r+1}+q_r\partial^r+\cdots+q_1\partial+q_0.  $$
\end{cor}}
  In particular when $n=6$ and $r=3$, $(${\it i.e.},
  spectral type is $3111$ $)$
  $$x^{-3}P=x^3\partial^6+x^2p_5\partial^5+xp_4\partial^4+q_3\partial^3+q_2\partial^2+q_1\partial+q_0,\quad p_i,q_j\in\mathbb{C}[x].
  $$
  
Recall that for an equation $P=\sum_j\sum_ip_{ij}x^i\partial^j \in\mathbb{C}[x][\partial]$, a subset $ap$ of coefficients $\{p_{ij}\}$ is called a set of {\it accessory parameters}, if all other coefficients are uniquely written in terms of $ap$ and the local exponents. The cardinality of $ap$ is called the {\it number of accessory parameters}.

\begin{prp}\label{numberofAP}{\rm(cf. \cite{Osh,Hara})}
  The number of accessory parameters of a Fuchsian equation of order $n$
  with $m$ singular points, with semi-simple local monodromy, is given by
  \[\frac12\left\{(m-2)n^2-\sum_{{\rm singular\ points}}({\rm multiplicity\ of\ local\ exponents}\ mod\ 1)^2+2\right\}.
  \]
\end{prp}
For $H_j$, $m=3$. The equation $H_6$ has Riemann scheme $R_6$ (Introduction), its spectral type is $(3111,3111,3111)$; since $\{6^2-3(3^2+3\cdot1^2)+2\}/2=1$, it has one accessory parameter. The others are computed as
  \[\def\arraystretch{1.1}
  \begin{array}{lll}
   {\rm equation}&{\rm spectral\ type}&\\ 
 \quad H_{6}\quad& (3111,3111,3111)&:\quad \{6^2-3(3^2+3\cdot1^2)+2\}/2=1,\\ 
 \quad H_{5}\quad& (2111,2111,311)&:\quad \{5^2-2(2^2+3\cdot1^2)-(3^2+2\cdot1^2)+2\}/2=1,\\ 
 \quad  H_{4}\quad& (211,211,1111)&:\quad\{4^2-2(2^2+2\cdot1^2)-(4\cdot1^2)+2\}/2=1,\\ 
 \quad H_{3}\quad& (111,111,111)&:\quad \{3^2-3(3\cdot1^2)+2\}/2=1,\\ 
\quad  E_2\quad& (11,11,11)&:\quad \{2^2-3(2\cdot1^2)+2\}/2=0.
  \end{array}\]
 The Gauss equation $E_2$ has no accessory parameter. The others have one.
  
   \subsection{Adjoint equations}\label{GenAdj}
Adjoint equation of a linear differential equation should be discussed under the frame work of projective differential geometry, as we sketch below. In this article, however, we make the following practical definition for {\it operators}. 

\begin{dfn}The {\sl adjoint} $P^*$ of $P=\sum p_j(x)\partial^j$ is defined as
  $$P^*=\sum (-)^j\partial^j\circ p_j(x).$$
\end{dfn}
When we are working on differential operators and their adjoints,
we {\it always assume} that the coefficients are polynomials in $x$
free of common factor. Otherwise we can not speak of adjoint symmetry:

\begin{remark}
  As we see in \S \ref{adjE2proof}, the adjoint of the Gauss
  operator $E=E(a,b,c)$ is again the Gauss operator $E^*=E(1-a,1-b,2-c)$.
  However, if we apply the above formula for
  $$P=\dfrac{1}{x(x-1)}\ E=\partial^2+\dfrac{(a+b+1)x-c}{x(x-1)}\ \partial+\dfrac{ab}{x(x-1)},
  $$
  then the adjoint $P^*$ is an operator with the Riemann scheme
  $$\left(\begin{array}{lcc}x=0:&1&c\\ x=1:&1&a+b-c+1\\ x=\infty:&-a-1&-b-1\end{array}\right),
  $$
  which is not Gauss, but $P^*\circ x(x-1)=E^*.$    
\end{remark}
\noindent

\subsubsection{Adjoints of the operators}\label{GenAdjAdj}
The adjoint operator of $H_j$ is the same operator with a simple change of the local exponents, and the accessory parameter. Once the operator is expressed in the $(x,\theta,\partial)$-form, this is easily checked by using the following formulae:

$$(PQ)^*=Q^*P^*,\quad \partial^*=-\partial,\quad \theta^*=-\partial\cdot x=-(\theta+1),$$
$$(\theta^i(\partial)^j)^*=(-\partial)^j(-\theta-1)^i=(-\theta-1-j)^i(-\partial)^j,\quad \partial\theta=(\theta+1)\partial.$$
For example, the adjoint of $H_6$ is computed as
$$\begin{array}{ll}
  T_0^*&=(-\theta+s+1)(-\theta+s)(-\theta+s-1)(-\theta-1-e_7)\cdots(-\theta-1-e_9),\\
  (T_1\partial )^*&=\partial^*(-\theta+1+s)(-\theta+s)B_1(-\theta-1)\\&=(-\theta+s)(-\theta+s-1)B_1(-\theta-2)\cdot(-\partial ),\\
(T_2\partial ^2)^*&=(-\theta+s-1)B_2(-\theta-3)\cdot(-\partial )^2,\\
  (T_3\partial ^3)^*&=(\theta+1+e_1) (\theta+1+e_2) (\theta+1+e_3)(-\partial )^3.\end{array}$$
\blue{The accessory parameter $T_{10}$ changes as in \S \ref{cochH6}.}%into  $$4*s_{12}*s+6*s^2-6*s_{12}-2*s_{21}+2*s_{22}-4*s_{23}-18*s+8-T_{10}.$$
\par\medskip
  Change of the Riemann schemes is given as follows:%will jump to the eyes:
\begin{itemize}
\item  $H_6$:
  \[
  \begin{array}{l}
    \left(\begin{array}{ccccccccc}
x=0:&0&1&2&e_1&e_2&e_3\\
x=1:&0&1&2&e_4&e_5&e_6\\
x=\infty:&s&\ s+1\ &s+2&e_7&e_8&e_9\end{array}\right)
\\
\hskip60pt
\to\ \left(\begin{array}{cccccccc}
0&1&2&2-e_1&2-e_2&2-e_3\\
0&1&2&2-e_4&2-e_5&2-e_6\\
-1-s&\ -s\ &1-s&1-e_7&1-e_8&1-e_9\end{array}\right),
  \end{array}
  \]
  
%$9+3s+\sum e=6\cdot5/2=15.$
  %\end{example}
  \comment{
  $S\!E_6$:
  $$S\!E_6^*(a,b,c,g,p,q,r)=S\!E_6( -a - 1, -b - 1, -c - 1, -g, -p - 1, -q - 1, 1 - r).$$
%\end{example}
  }

\item  $H_5$:
$$ 
  \begin{array}{l}
    \left(\begin{array}{cccccc}
x=0:&0&1&e_1-1&e_2-1&e_3-1\\
x=1:&0&1&e_4-1&e_5-1&e_6-1\\
x=\infty:&1+s&\ 2+s\ &3+s&e_7+1&e_8+1\end{array}\right)
\\
\hskip100pt
\to\ \left(\begin{array}{cccccc}
0&1&2-e_1&2-e_2&2-e_3\\
0&1&2-e_4&2-e_5&2-e_6\\
-s-1&\ -s\ &-s+1&1-e_7&1-e_8\end{array}\right),\end{array}$$
%$5+3s+\sum e=5\cdot4/2=10.$
%\end{example}
\comment{
\item  $S\!E_5$:
  $$S\!E_5^*(a,b,c,g,p,q)=S\!E_5( -a - 1, -b - 1, -c - 1, -g, -p - 1, -q - 1).$$
%\end{example}
}
\item $H_4$:
\[\left(\begin{array}{ccccc}
x=0:&0&1&e_1&e_2\\
x=1:&0&1&e_3&e_4\\
x=\infty:&e_5&e_6&e_7&e_8\end{array}\right)\to
\left(\begin{array}{cccc}
0&1&1-e_1&1-e_2\\
0&1&1-e_3&1-e_4\\
1-e_5&1-e_6&1-e_7&1-e_8\end{array}\right),
\]
% $3+2s+\sum e=4\cdot3/2=6.$
%\end{example}
\comment{ $S\!E_4$:
$$S\!E_4^*(a,b,c,g,q)=S\!E_4(-a-1,-b-1,-c-1,-g,-q-1.)$$
%\end{example}
 $Z_4$:
$$ Z_4^*(A,\!k)=Z_4(A,-k).$$
%\end{example}
 ${\rm ST}_4$:
\[\left(\begin{array}{ccccc}
x=0:&0&1&e_1&e_2\\
x=1:&0&1&e_3&e_4\\
x=\infty:&s&s+1&e_5&e_6\end{array}\right)\to
\left(\begin{array}{cccc}
0&1&1-e_1&1-e_2\\
0&1&1-e_3&1-e_4\\
-s&1-s&1-e_5&1-e_6\end{array}\right),
\]
% $3+2s+\sum e=4\cdot3/2=6.$
%\end{example}
  $_4H_3$:
\[\left(\begin{array}{ccccc}
x=0:&0&e_1&e_2&e_3\\x=1:&0&1&2&e_4\\x=\infty:&e_5&e_6&e_7&e_8\end{array}\right)\to
\left(\begin{array}{cccc}
0&-e_1&-e_2&-e_3\\0&1&2&2-e_3\\1-e_5&1-e_6&1-e_7&1-e_8\end{array}\right),\]
$$_4E^*_3(a_0,a_1,a_2,a_3;b_1,b_2,b_3)= {}_4E_3(1-a_0,1-a_1,1-a_2,1-a_3;2-b_1,2-b_2,2-b_3).$$
%\end{example}
}
\item  $H_3$:
\[
\left(\begin{array}{cccc}
x=0:&0&e_1&e_2\\
x=1:&0&e_3&e_4\\
x=\infty:&e_5&e_6&e_7\end{array}\right)\to
\left(\begin{array}{ccc}
0&-e_1&-e_2\\
0&-e_3&-e_4\\
2-e_5&2-e_6&2-e_7\end{array}\right),
\]    %$\sum e=3\cdot2/2=3.$
%\end{example}
\comment{
  $S\!E_3$:
  $$S\!E_3^*(a,b,c,g)=S\!E_3( -a - 1, -b - 1, -c - 1, -g).$$
%\end{example}
\item $Z_3$:
$$Z_3^*(A_0,A_1,A_2,A_3)=Z_3(A_0,A_1,-A_2,A_3).$$
%\end{example}
  $_3E_2$:
\[\left(\begin{array}{cccc}
  x=0:&0&e_1&e_2\\
  x=1:&0&1&e_3\\
  x=\infty:&e_4&e_5&e_6
\end{array}\right)\to
\left(\begin{array}{ccc}
  0&-e_1&-e_2\\
  0&1&1-e_3\\
  1-e_4&1-e_5&1-e_6
\end{array}\right),\]
$$_3E^*_2(a_0,a_1,a_2;b_1,b_2)= {}_3E_2(1-a_0,1-a_1,1-a_2;2-b_1,2-b_2).$$
%\end{example}
}
\comment{\item\label{adjE2} $E_2$:
  \[ \begin{array}{l}       %%# changed to array to make linebreak
    R_2(a,b,c)=\left( \begin{array}{cc}
    0&1-c\\
    0&c-a-b\\
    a&b\end{array}\right) \\
\hskip60pt \to
\left( \begin{array}{cc}
    0&-(1-c)\\
    0&-(c-a-b)\\
    1-a&1-b\end{array}\right)=R_2(1-a,1-b,2-c).
  \end{array}
  \]%\end{example}
}
\item  $E_2$:
\[
\left(\begin{array}{ccc}
x=0:&0&e_1\\
x=1:&0&e_2\\
x=\infty:&e_3&e_4\end{array}\right)\to
\left(\begin{array}{ccc}
0&-e_1\\
0&-e_2\\
1-e_3&1-e_4\end{array}\right).
\]    %$\sum e=1.$
\end{itemize}

\begin{remark}\label{adjisminus}(See the end \S \ref{GenAdjPrj}.)
Let $\{e_1,e_2,\dots\}$ be the set of local exponents of an equation at a singular point. The set of local exponents of the adjoint equation at the point is   $\{p-e_1,p-e_2,\dots\}$ for some integer $p$.
\end{remark}

\subsubsection{Self-adjoint equations}
For the equation $H_j$, the self-adjoint one is a special $H_j$ with the following Riemann scheme, and it turns out to be a special $E_j$. In this subsection, the equations may have logarithmic singularities.
\begin{lemma}
If $\{x_1,x_2,\dots,x_r\}=\{-x_1,-x_2,\dots,-x_r\}$, by changing the indices, we have
\begin{itemize}
\item $x_1=-x_2,x_3=-x_4,\dots, x_{r-1}=-x_r$ when $r$ is even, 
\item $x_1=-x_2,x_3=-x_4,\dots, x_{r-2}=-x_{r-1},x_r=0$ when $r$ is odd.
\end{itemize}
\end{lemma}
 For example, %by this Lemma,
if $e=\{e_1,e_2,e_3\}=\{2-e_1,2-e_2,2-e_3\}$, then
we may assume $e=\{e_1,2-e_1,1\}$ by putting $x_i=e_i -1$;
if $e=\{e_5,e_6,e_7,e_8\}=\{1-e_5,1-e_6,1-e_7,1-e_8\}$, then
$e=\{e_5,1-e_5,e_7,1-e_7\}$ by putting $x_i = e_i - 1/2$.

\begin {itemize}

\item The self-adjoint $H_6$ has local exponents as 
\[\left(\begin{array}{ccccccc}
 x=0:&0&1&2&e_1&2-e_1&1\\ x=1:& 0&1&2&e_4&2-e_4&1\\ x=\infty:&-1/2&\ 1/2\ &3/2&\ e_7\ &1-e_7&1/2\end{array}\right),\]
it is irreducible for generic $\{e_1,e_4,e_7\}$, 
\footnote{Suppose $H_6=P_1\circ P_2$. If ${\rm order}(P_2)=1,$ (resp. 3) choose one (resp. three) element(s) from the set of local exponents of each singular point, the sum is not an integer.
If ${\rm order}(P_1)=2,$ choose two elements and do the same, if the sum is an integer then it is  $\ge2+k$, where $k$ is the number of apparent singular points of $P_2$. On the other hand,
Fuchs relation of $P_2$ says the sum is $\le k+1$.}
and has the accessory parameter as
$$T_{10} = e_1^2 - e_4^2 + 2 e_7^2 - 2 e_1 + 2 e_4 - 2 e_7 - 15/4.$$

\item The self-adjoint $H_5$ has the Riemann scheme 
\[\left(\begin{array}{cccccc}
 x=0:&0&1&\ e_1-1\ &2-e_1&1/2\\ x=1:& 0&1&e_4-1&2-e_4&1/2\\  x=\infty:&0&1&2&e_7+1&\ 1-e_7\ \end{array}\right).\]
It is reducible and is equal to $\partial\circ X\circ\partial,$ where 
$$X=A^3dx^3+\cdots, \quad A=x(x-1)$$ is essentially a self-adjoint $H_3$ defined below, that is 
 $A\circ X\circ A^{-1}$ is the self-adjoint $H_3$ with the Riemann scheme
\[\left(\begin{array}{cccc}
 x=0:&0&\ e_1-3/2\ &3/2-e_1\\ x=1:& 0&\ e_4-3/2\ &3/2-e_4 \\  x=\infty:&e_7+1&\ 1-e_7\ &1\end{array}\right).\]

\item  The self-adjoint $H_4$ has the Riemann scheme 
\[\left(\begin{array}{ccccc}
 x=0:&0&1&e_1&1-e_1\\  x=1:&0&1&e_3&1-e_3\\  x=\infty:&e_5&\ 1-e_5\ &e_7&1-e_7\end{array}\right), \]
it is irreducible for generic $\{e_1,e_3,e_5,e_7\}$, and has the accessory parameter as
$$\mathcal T_{10}=e_1^2 - e_3^2 + e_5^2 + e_7^2 - e_1 + e_3 - e_5 - e_7 - 2.$$
\comment{ T0:=mult((z+c5),(z+c6),(z+c7),(z+c8),[dx,x]):
     T1:=-2*mult(z,z,z,[dx,x])+ T12*mult(z,z,[dx,x])+ T11* z+ T10:
     T12:=c1+c2-c5-c6-c7-c8-5:
     T11:=3 (c1+c2)-c1 c2+c3 c4-c5 c6-c5 c7-c5 c8-c6 c7-c6 c8-c7 c8-8:
     T2:=mult((z-c1+2),(z-c2+2),[dx,x]):}

\item  The self-adjoint $H_3$ has the Riemann scheme 
\[\left(\begin{array}{cccc}
 x=0:&0&e_1&-e_1\\ x=1:& 0&e_3&-e_3\\  x=\infty:&e_5&\ 2-e_5\ &1\end{array}\right),\]
it is irreducible for generic $\{e_1,e_3,e_5\}$, and has the accessory parameter as
$$a_{00}=(e_1^2 - e_3^2 + e_5^2)/2 - e_5.$$

\item The self-adjoint $E_2$ has the Riemann scheme 
\[\left(\begin{array}{ccc}
 x=0:&0&0\\  x=1:&0&0\\  x=\infty:&e_3&1-e_3\end{array}\right),\]
it is irreducible for generic $e_3$.
\end{itemize}

\subsubsection{Adjoint equation in projective differential geometry}\label{GenAdjPrj}
In general, two linear homogeneous ordinary differential equations are said to be projectively equivalent if one changes into the other by multiplying a function to the equation, multiplying a function to the unknown, and by changing the independent variable. 
We give a short discussion on the notion of adjoint defined
projectively invariant way as follows (cf. \cite{Wil}).
For notational simplicity, we consider a third-order equation
  \[E: \quad u'''+p_1u''+p_2u'+p_3u=0,\]
and its Schwarz map: $x\mapsto u(x)=(u^1(x),u^2(x),u^3(x))$,
where $u^i$ are independent solutions. It is seen as a curve in the 3-space
or on the projective plane relative to the homogeneous coordinates.
Define its dual curve by the map:
$x\mapsto \xi(x)= u(x)\wedge u(x)' \in \wedge^2 V$, that is,
$\xi(x)=(\xi_1(x),\xi_2(x),\xi_3(x))$, where
\[ \xi_1=
  \left|\begin{array}{cc}u^2&u^3\\(u^2)'&(u^3)'\end{array}\right|, \quad
   \xi_2=\left|\begin{array}{cc}u^3&u^1\\(u^3)'&(u^1)'\end{array}\right|, \quad
   \xi_3=\left|\begin{array}{cc}u^1&u^2\\(u^1)'&(u^2)'\end{array}\right|.\]
By computation, we see $\xi_1,\xi_2$ and $\xi_3$ satisfy
  \[\xi'''+2p_1\xi''+(p_1'+p_1^2+p_2)\xi'+(p_2'+p_1p_2-p_3)\xi =0,\]
while the adjoint equation $E^*$  of $E$ is given as
  \[E^*:  \quad v'''-(p_1v)''+(p_2v)'-p_3v=0.\]
These two equations look different, but both are equivalent
projectively (change $\xi$ to $\lambda^{-2}\xi$ and
$v$ to $\lambda v$ where
$\lambda=\exp(\int\frac{1}{3}p_1\, dx)$) to the equation
  \[{\rm adj}E:\quad w'''+P_2w'+(P_2'-P_3)w=0,\] %prE:\quad
where
  \[P_2=p_2-p_1'-\frac{1}{3}p_1^2, \quad
    P_3=p_3-\frac{1}{3}p_1''+\frac{2}{27}p_1^3-\frac{1}{3}p_1p_2.\]
Namely, the equation $E^*$ is equivalent to the equation satisfied
by $\xi$; this equation of $\xi$ is sometimes called the
Wronskian equation. By the way, the equation $E$ itself is %known to be
equivalent projectively (change of coordinate) to
  \[   u'''+P_2u'+P_3u=0.\] %PE:\quad
Though $P_2$ and $P_3$ are not projectively invariant, the cubic form
  \[ Rdx^3,\quad{\rm where}\quad R=P_3-\frac12 P_2'\]
is invariant (the Laguerre-Forsyth invariant).
Writing this invariant $R^*$ for $adjE$, we see that 
  \[R^*=-R.\]
This identity of invariants shows a relation of a differential equation
and its adjoint equation. In general for an equation of order $n$,
invariants $R_3,\dots,R_n$ are defined,
and they are related to the invariants $R^*_3,\dots,R^*_n$
of the adjoint equation as $R_j^*=(-)^jR_j$ (cf. \cite{Wil}).
%\par\medskip

Now we apply the above general theory to the Fuchsian differential
equation $E$. The local exponents of the adjoint equation are given as follows.
Let $e_1$, $e_2$, $e_3$ be local exponents of $E$ at $x=0$: assume that
$u^i$ are chosen as
  \[ u^1=x^{e_1}f_1,\quad u^2=x^{e_2}f_2,\quad u^3=x^{e_3}f_3,\]
where $f_i$ are holomorphic at $x=0$ (and $f_i(0)=1$ for simplicity).
%Within the projective consideration, the differences $e_2-e_1$ and $e_3-e_1$ make sense.
It is easy to see that
  \[u\wedge u' = ( x^{e_2+e_3-1}g_1, x^{e_1+e_3-1}g_2,x^{e_1+e_2-1}g_3),\]
where $g_1=(e_3-e_2)f_2f_3+ xh_1$, $h_1$ being holomorphic at $x=0$, and so on. 
Within the projective consideration, the differences of the exponents make sense.
%On the other hand, we have
%$$p_1 = (3-e_1-e_2-e_3)/x + h_1,\quad \lambda^3 =x^{3-e_1-e_2-e_3} \cdot h_2,$$
%where $h_1$ and $h_2$ are holomorphic at $x=0$. 
These explain why $\{p-e_1,\, p-e_2,\, p-e_3\}$ ($p\in \mathbb{Z}$)
  (cf. Remark \ref{adjisminus}) appears as a set of local exponents of the adjoint equation $E^*$ at $x=0$.

\section{Addition and middle convolution}\label{GenAdd}\secttoc
In this section, addition and middle convolution are introduced.
We consider the Weyl algebra $W[x]=\mathbb{C}[x][\partial]$,
and put
$$
 W(x)=\mathbb{C}(x)\otimes_{\mathbb{C}[x]}W[x]. 
$$
We regard a differential equation (a differential operator) as an element
of $W[x]$ or $W(x)$.

\subsection{Definition of addition and middle convolution}
\begin{dfn}\label{GenAddAdd}
 For $P\in W(x)$ and a function $f$ in $x$, {\it addition} by $f$ is defined as
 $${\rm Ad}(f)P:=f\circ P\circ f^{-1}.$$
\end{dfn}

Katz (\cite{katz}) introduced the middle convolution as an operation for local systems on a Riemann surface, and Dettweiler and Reiter (\cite{DR}) made an additive analogue for Fuchsian systems of ordinary differential equations.
Oshima (\cite{Osh,Oshi2}) interpreted the middle convolution for Fuchsian systems as an operation on the Weyl algebra $W[x]$.

\begin{dfn}\label{defmc}\label{GenAddMid}
i)
For $P\in W(x)\setminus\{0\}$, we choose an element in
$(\mathbb{C}(x)\setminus\{0\})P\cap W[x]$ with the minimal degree,
and denote it by ${\rm R}(P)$.
For $P=0$, we put ${\rm R}(P)=0$.
Note that ${\rm R}(P)$ is determined up to multiplication by non-zero elements of $\mathbb{C}$.

\par\noindent
ii)
We define an automorphism ${\rm L}$ of $W[x]$ by
$$
 {\rm L}(\partial)=x,\  {\rm L}(x)=-\partial,
$$
which is called the Laplace transformation.

\par\noindent
iii)
We define the {\it middle convolution} $mc_{\mu}$ with parameter $\mu\in\mathbb{C}$ by
$$
 mc_{\mu}={\rm L}^{-1}\circ{\rm R}\circ{\rm Ad}(x^{-\mu})\circ{\rm L}\circ{\rm R}.
$$
Owing to the ambiguity of ${\rm R}$, for $P\in W[x]$, $mc_{\mu}(P)\in W[x]$ is determined
up to multiplication by non-zero elements of $\mathbb{C}$.
\end{dfn}

\begin{dfn}\label{GenAddRL}
  For a function $u(x)$, {\it Riemann-Liouville transformation} of $u$
  with parameter $\mu$ is defined as the function in $x$:
 $$I^{\mu}_{0}(u)(x)=\frac1{\Gamma(\mu)}\int_0^x u(t)(x-t)^{\mu-1}\, dt.$$
% where $\gamma$ is a cycle.\footnote{\blue{$\gamma$ is topologically closed and the values of the integrand at the starting point and the ending point agree.}}
\end{dfn}

Analytically, the middle convolution $mc_{\mu}$ is realized by Riemann-Liouville transformation.
Namely,
if $u$ is a solution of a linear differential equation $P$,
the function $I^{\mu}_{0}(u)$ is a solution of
the differential equation $mc_{\mu}(P)$ under some condition (\cite[Theorem 4.2]{Oshi2}).

\medskip   
Algorithmically,
the equation $mc_{\mu}(P)$ is obtained as follows (\cite{Osh,Oshi2,Hara}):
Write $P\in  W(x)$ as in the form $\sum_{j=0}^na_j(x)\partial^j$ with $a_j(x)\in\mathbb{C}(x)$,
and remove poles and the common factor of $a_n(x),a_{n-1}(x),\dots,\allowbreak a_0(x)$ (the operation ${\rm R}$).
We use the same letter $P$ for the result.
 Multiply $P$ by $\partial^k$ with sufficiently large positive integer $k$
 from the left so that $\partial^kP$ can be written
 as a linear combination of $\theta^i\circ\partial^j$ over $\mathbb{C}$, where $\theta=x\partial$.
Then replace $\theta$ by $\theta-\mu$, and divide the result by $\partial$
from the left as many times as possible.
 (The result is independent of $k$.) 
\subsubsection{Some properties of middle convolution}
We put $W[x]^0=W[x]\setminus\partial W[x]$.
On the set $W[x]^0$, the middle convolution has the additive property
$$
 mc_0={\rm id.},\quad
 mc_{\mu}\circ mc_{\mu'}=mc_{\mu+\mu'},
$$
and so $mc_{\mu}$ is invertible:
$$
(mc_{\mu})^{-1}=mc_{-\mu}.
$$

\noindent
For an operator $P\in W[x]^0$ with singular points $0$, $1$, $\infty$, set
$$\begin{array}{lcl}
  d&=&({\rm multiplicity\ of\ }0{\ \rm in\ the\ exponents\ at\ }x=0)\\
  & +&({\rm multiplicity\ of\ }0{\ \rm in\ the\ exponents\ at\ }x=1)\\
  & +&({\rm multiplicity\ of\ }\mu{\ \rm in\ the\ exponents\ at\ }x=\infty)-{\rm order}(P),\end{array}
$$     %%# added \\
where the exponents are regarded mod 1.
%Here and in the following, ${\rm order}(P)$ denotes the order of the operator $P$. 
Then we have
  $${\rm order}(mc_{\mu}(P))={\rm order}(P)-d.$$
The change of the Riemann scheme and the spectral type of $P$ by the middle convolution $mc_{\mu}$
is described in \cite[Theorem 5.2]{Osh}. 
It is known  that middle convolutions do not change the number of accessory parameters.

\subsubsection{Simple  examples}
  If $P=E_2$, then $d=1+1+(0{\ \rm or\ }1)-2=0{\ \rm or\ }1$.
  Thus any middle convolution of $E_2$ is again a Gauss operator
  or a 1st order operator. But if we perform an addition first
  to change the local exponent 0 of $x=0$ or/and $x=1$ non-zero,
  then $d=-2,-1{\ \rm or\ }0$. %See \S \ref{E4MC}
  So
  ${\rm order}(mc_{\mu}(E_2))$ can be $4$ or $3$ or $2$. In the following we see how the Gauss equation $E_2$ is transformed to the generalized hypergeometric equation ${}_3E_2$:

  \begin{itemize}[leftmargin=36pt] %[leftmargin=*]
  \item $E_2\longrightarrow{}_3E_2$: For a solution $u$ of
  the Gauss equation $E_2(e)$, perform a multiplication (called an addition)
  $u(x)\rightarrow x^\nu u(x)$ with $\nu\in\mathbb{C}$
 and then make a middle convolution
  with parameter $\mu$. 
The Riemann scheme changes as
  \[\begin{array}{lcl}   %%# added array and displaced 
    \left(\begin{array}{lll}x=0:&0&e_1\\x=1:&0&e_2\\x=\infty:&e_3&e_4  \end{array}\right) & \underset{x^\nu}\to &
    \left(\begin{array}{ll}\nu&e_1+\nu\\0&e_2\\ e_3-\nu&e_4-\nu \end{array}\right)\\
&  \underset{mc_{\mu}}\to &
\left(\begin{array}{lll}0&\nu+\mu&e_1+\nu+\mu\\ 0&1&e_2+\mu\\ 1-\mu&e_3-\nu-\mu&e_4-\nu-\mu  \end{array}\right),
\end{array}
\] 
where $e_1+\cdots+e_4=1$.
Thanks to \cite[Theorem 5.2]{Osh}, the spectral type of the last one is $(111,21,111)$. 
Thus the last one is the Riemann scheme of
a generalized hypergeometric equation ${}_3E_2$.\medskip

\item $E_2\longleftarrow{}_3E_2$: For the operator
  $${}_3E_2= (\theta+a_0)(\theta+a_1)(\theta+a_2)-(\theta+b_1)(\theta+b_2)\partial,$$
  replace $\theta$ by $\theta-a_2+1$, and we get
  \[ \def\arraystretch{1.1}
  \begin{array}{l}(\theta+a_0-a_2+1)(\theta+a_1-a_2+1)(\theta+1)%\\
%    \qquad \qquad 
-(\theta+b_1-a_2+1)(\theta+b_2-a_2+1)\partial\\
    \quad =\partial\ [x(\theta+a_0-a_2+1)(\theta+a_1-a_2+1)%\\
 %     \quad \qquad \qquad  
-(\theta+b_1-a_2)(\theta+b_2-a_2)].
  \end{array}\]
  Dividing by $\partial$ from the left we have a second-order equation. Multiplying a certain power of $x$, we get a Gauss equation $E_2$.
  \end{itemize}

\subsection{From $H_3$ to $H_6,H_5$,  and $ H_4$}\label{fromH3toH654bymc}
In this section and \S \ref{fromH654toH3bymc}, statements for $H_j$ are valid also for $G_j$ and $E_j$.
\subsubsection{From $H_3$ to $H_6$}\label{fromH3toH6}
We repeat the statement in the Introduction. Perform an addition to $H_3=x^2(x-1)^2\partial^3+\cdots$: 
$$L:=x(x-1){\rm Ad}(x^{g_0}(x-1)^{g_1})H_3=x^3(x-1)^3\partial^3+\cdots.$$
Then the Riemann scheme changes as 
$$
R_3:  \begin{pmatrix}
  x=0:&0&b_1&b_2\\
  x=1:&0&b_3&b_4\\
  x=\infty:&b_7&b_5&b_6
 \end{pmatrix}
 \to
R(L): \begin{pmatrix}
  g_0&b_1+g_0&b_2+g_0\\
  g_1&b_3+g_1&b_4+g_1\\
 b_7-g_0-g_1&b_5-g_0-g_1&b_6-g_0-g_1
 \end{pmatrix}.$$
Note $b_1+\cdots+b_7=3.$ 
  Since $\partial^3\circ L$ has a $(\theta,\partial)$-form,
  we perform a middle convolution (replace $\theta$ by $\theta-u$), and we get
  $$\begin{pmatrix}
x=0:&0&1&2&g_0+u&b_1+g_0+u&b_2+g_0+u\\
x=1:&0&1&2&g_1+u&b_3+g_1+u&b_4+g_1+u\\
x=\infty:&-u+1&-u+2&-u+3&b_5-g_0-g_1-u&b_6-g_0-g_1-u&b_7-g_0-g_1-u
  \end{pmatrix}.$$
  Finally we change the names of the exponents as
 $$\begin{pmatrix}
   x=0:&0&1&2&e_1&e_2&e_3\\
   x=1:&0&1&2&e_4&e_5&e_6\\
   x=\infty:&s&s+1&s+2&e_7&e_8&e_9
  \end{pmatrix}$$
  and regard $e_1,\dots,e_9$ are free and $s$ is determined
  by the Fuchs relation. 
Thanks to \cite[Theorem 5.2]{Osh}, 
the spectral type is $(3111,3111,3111)$. Thus we find that this is the Riemann scheme of $H_6(e)$.
%2023/01/23 re-checked (Part3/G3/G3toG6changeofT10.mw.)
  
\subsubsection{From $H_3$ to $H_5$}\label{fromH3toH5}
Perform an addition:
$$(x-1){\rm Ad}((x-1)^{g_1})H_3=x^2(x-1)^3\partial^3+\cdots,$$
and multiply $\partial^2$ from the left. This admits a $(\theta,\partial)$-form. Replace $\theta$ by $\theta-u$.
The resulting equation has the Riemann scheme
$$\begin{pmatrix}
0& 1& 2& b_2 + u& b_1 + u\\
0& 1& g_1 + u&  b_4+ g_1  + u& b_3 + g_1 + u\\
-u + 1& 2 - u& b_6 - g_1 - u& b_5 - g_1 - u&\quad -\sum_{i=1}^6b_i - g_1 - u + 3
\end{pmatrix}.$$
Exchange the singularities $x=0$ and $x=\infty$, perform an addition to make the local exponents at $x=0$ as $\{0,1,*,*,*\}$, and rename the local exponents to find the result is the Riemann scheme of $H_5$.

  \subsubsection{From $H_3$ to $H_4$}\label{fromH3toH4}
  Without performing an addition to $H_3$, multiply $\partial$ from the left and get a $(\theta,\partial)$-form. Replace $\theta$ by $\theta-u$, and do the same as above to get the Riemann scheme of $H_4$.

  \subsection{From $H_6$, $H_5$,  and  $H_4$ to $H_3$}\label{fromH654toH3bymc}
\subsubsection{From $H_6$ to $H_3$}\label{fromH6toH3}
Recall the $(\theta,\partial)$-form of $H_6$ given in Proposition \ref{eqwithR6}, and  the formulae
\[(\theta+3)(\theta+2)(\theta+1)=\partial^3x^3,       %%# spacing
   \ (\theta+3)(\theta+2)\partial=\partial^3x^2,
   \ (\theta+3)\partial^2=\partial^3x,
   \ \theta\partial=\partial(\theta-1).\]
We see that  by replacing
$\theta$ by $\theta-t$ (middle convolution with parameter $t$),
where
  \[t:=s-1,\quad s=2-\frac13\sum_{i=1}^9e_i,\]
the expression of $H_6(\theta=\theta-t)$ is divisible by $\partial^3$ from the left. If we write the quotient by $mcH=x^3(x-1)^3\partial^3+\cdots$,
then its Riemann scheme is
$$R(mcH):\ \left(\begin{array}{ccc}
  e_1+t&e_2+t&e_3+t\\ e_4+t&e_5+t&e_6+t\\ e_7-t&e_8-t&e_9-t\end{array}\right).
$$
We next transform it into $x^{-(t+e_1)-1}(x-1)^{-(t+e_4)-1}mcH\circ x^{t+e_1}(x-1)^{t+e_4}$. Then the equation can be expressed as $x^2(x-1)^2\partial^3+\cdots$, and 
the Riemann scheme changes into
$$ \left(\begin{array}{ccc}
  0&e_2-e_1&e_3-e_1\\ 0&e_5-e_4&e_6-e_4\\ e_7+e_1+e_4+t&e_8+e_1+e_4+t&e_9+e_1+e_4+t
\end{array}\right).
$$

\noindent
Introduce parameters $\epsilon_1,\dots,\epsilon_7$ by
$$\begin{array}{l}
e_2 - e_1 = \epsilon_1,\  e_3 - e_1 = \epsilon_2,\  e_5 - e_4 = \epsilon_3,\  e_6 - e_4 = \epsilon_4,\\
e_1 + e_4 + e_7+t = \epsilon_5,\  e_1 + e_4 + e_8+t =\epsilon_6,\  e_1 + e_4 + e_9+t = \epsilon_7,
\end{array}$$
$ \epsilon_1+\cdots+\epsilon_7=3$. The equation is $H_3(\epsilon)$, that is, $H_3(e)$ replaced $e$ by $\epsilon$.

\begin{remark}\label{fromH6toH5}(From $H_6$ to $H_5$) 
On the other hand, replace $\theta$ by $\theta-e_9+1$ in $H_6$ and divide by $\partial$ from the left. The Riemann scheme turns out to be
$$\begin{pmatrix}
  0& 1& e_1 + e_9 - 1& e_2 + e_9 - 1& e_3 + e_9 - 1\\
  0& 1& e_4 + e_9 - 1& e_5 + e_9 - 1& e_6 + e_9 - 1\\
  s+1-e_9&s+2-e_9&s+3-e_9&e_7-e_9+1&e_8-e_9+1
\end{pmatrix}.$$
Put
  $e_i+e_9=\epsilon_i\ (i=1,\dots,6)$, and $e_j-e_9=\epsilon_j\ (j=7,8)$, and replace $s-e_9$ by $s$. Then it is equal to $H_5(\epsilon)$.
\end{remark}
 
 \subsubsection{From $H_5$ to $H_3$}\label{fromH5toH3}
 Recall the $(x,\theta,\partial)$-form of
$H_5=H_6(e_9=0)/\partial=x\overline{T}_0+\overline{T}_1+\cdots=x^3(x-1)^3\partial^5+\cdots$:
Perform a middle convolution:
multiply $\partial$ to $H_5$ from the left
and get a $(\theta,\partial)$-form,
then  replace $\theta$ this time by $\theta-s$ ($s=2-\frac13\sum_{i=1}^8e_i$),
and divide it from the left by $\partial^3$, and multiply powers of $x$ and $x-1$ to make one of the local exponents at $x=0$ and $x=1$ to be 0. Then we get $H_3$.
The procedure is quite analogous to that
of getting $H_3$ from $H_6$ shown above.
\subsubsection{From $H_4$ to $H_3$}\label{fromH4toH3}
Recall the $(\theta,\partial)$-form of
$H_4=\mathcal T_0+\mathcal T_1\partial+\mathcal T_2\partial^2=x^2(x-1)^2\partial^4+\cdots$.
Perform a middle convolution: Replace $\theta$  by $\theta-c_8$,  
and divide it from the left by $\partial$, and multiply powers of $x$ and $x-1$ to make one of the local exponents at $x=0$ and $x=1$ to be 0. Then we get $H_3$.

\section{Shifts, shift  operators, shift relations and S-values}\label{GenShift}\secttoc
For the hypergeometric series $F$ (\S 5), the following identities 
$$P_{a+} F(a,b,c;x)=aF(a+1,b,c;x),\ P_{a-} F(a,b,c;x)=(a-c)F(a-1,b,c;x),$$
where
$$P_{a+}=x\partial+a,\quad P_{a-}=x(x-1)\partial+bx+a-c$$
are well-known. They can be para-phrased by using the hypergeometric operator $E$ (\S 4.2) as
$$(EPQE):\ E(a+1,b,c)P_{a+}=Q_{a+}E(a,b,c),\quad E(a-1,b,c)P_{a-}=Q_{a-}E(a,b,c),$$
where 
$$Q_{a+}=x\partial +a+1,\quad Q_{a-}=x(x-1)\partial +(b+1)x+a-c-1.$$
The operators $P_{a\pm}$ are called by various names such as ladder operator, step-up/down operator, contiguity operator,... In this paper and in the paper \cite{HOSY2}, we call them 
\begin{center}{\it shift operators} for the {\it shifts} $a\to a\pm1$.
\end{center}
Another example: Since $$\partial F(a,b,c;x)=\frac{ab}cF(a+1,b+1,c+1;x),\quad E(a+1,b+1,c+1)\partial=\partial E(a,b,c),$$
$\partial$ is the shift operator for the shift $(a,b,c)\to(a+1,b+1,c+1)$.\par\smallskip
The relations $(EPQE)$ are also well-known (e.g. \cite{vPS} Proposition 1.13). We call them
\begin{center}{\it shift relations} for the {\it shifts}  $a\to a\pm1$.\end{center}
The computations such as
$$P_{a+}(a-1)\circ P_{a-}-xE=(a-1)(a-c),\quad P_{a-}(a+1)\circ P_{a+} -xE=a(a-c+1)$$ 
are also popular. We call these constants 
\begin{center}{\it S-values} for the {\it shifts} $a\to a\pm1$,
\end{center}
and write as
$$Sv_{a-}= (a-1)(a-c),\quad Sv_{a+}=a(a-c+1).$$
In this section we define these in a general setting. 
\subsection{The ring of differential operators, left ideals and reducibility}\label{GenShiftRing}
Let $D=\mathbb{C}(x)[\partial]$ be the ring of ordinary differential operators with coefficients in rational functions of $x$. \blue{We call the degree of the differential operator $P$ relative to $\partial$ the {\it order} of $P$ and denote it as ${\rm order}(P)$.}

\begin{itemize}[leftmargin=36pt] %[leftmargin=*]
\item Every left ideal of $D$ is principal, because $D$ admits Euclidean algorithm.

\item An operator $E\in D$ is said to be {\it reducible} if it can be written as the product of two operators of positive order. When $E$ is Fuchsian, it is reducible if and only if its solution space has a monodromy invariant proper non-trivial subspace. $E$ is said to be {\it irreducible} if it is not reducible.

\item If $E$ is irreducible, the left ideal $DE$ generated by $E$ is maximal, because, if not, there is a left ideal $L$ such that $D\supsetneq L\supsetneq E$, since $L$ is generated by an element $F\in D$, $E$ is divisible by $F$.
\end{itemize}

\begin{lemma}\label{inverseofP}
  Consider two operators $P,E\in D$ such that $0<{\rm order}(P)<{\rm order}(E).$
  If $E$ is irreducible, then $P$ has its (left) inverse in $D$ modulo $E$.
\end{lemma}
\begin{proof} %n\noindent {\sl Proof.}
  Since $DE$ is maximal and $P\not\in DE$, we have $D=DP+DE,$ that is, there exist $R,Q\in D$ satisfying $1=QP+RE$.
\end{proof} %   \hfill$\square$ % \bigbreak 

\begin{dfn}A singular point of an equation is said to be {\it apparent} if every solution at this point is holomorphic. 
\end{dfn}

\begin{prp}\label{generic_is_irred}$H_j$ $(j=2,\dots,6)$ are irreducible if the local exponents $e$  are generic.
\end{prp}

\begin{proof} Suppose a differential operator $E$ is reducible and is written as $F_1\circ F_2$, where ${\rm order}(F_1)\neq 0$ and ${\rm order}(F_2)\neq 0$.
At each of the singular points of $E$, % $\{0,\, 1,\, \infty\}$,
the set of local exponents of $F_2$ is a subset of that of $E$. The singular points of $F_2$ other than the  singular points of $E$ are apparent, so the local exponents at such points are non-negative integers. The Fuchs relation (\ref{Fuchsrelation}) for $F_2$ says that the sum of all the local exponents is an integer.  When $E=H_j$, \blue{the sum of a proper subset of the local exponents $e_1,e_2,\dots$ can not be an integer when the local exponents are generic.}
\end{proof} %\hfill$\square$ \bigbreak 

\begin{dfn}\label{solutionspace}
  For a given $E\in D$ with the set of singular points $S$, choose any point $x_0\in \mathbb{C}-S$. Let ${\rm Sol}(E)(x_0)$ be the solution space of $E$ at $x_0$.  For a loop $\rho\in \pi_1(\mathbb{C}-S,x_0)$ with base point $x_0$, we can analytically continue a solution at $x_0$ to get another solution at $x_0$. In this sense, ${\rm Sol}(E)(x_0)$ is a $\pi_1(\mathbb{C}-S,x_0)$-module. Since $x_0$ does not matter in the following arguments, from now on we drop $x_0$, and call this space simply {\it the solution space} and write as ${\rm Sol}(E)$, which is a  $\pi_1(\mathbb{C}-S)$-module.
\end{dfn}

\begin{lemma}\label{monodinv} $E\in D$ is reducible if and only if the solution space ${\rm Sol}(E)$ has a non-zero proper  $\pi_1(\mathbb{C}-S)$-submodule, which is often called a monodromy invariant subspace.
\end{lemma}
\begin{proof}
If $E$ factors as $F_1\circ F_2$ $(F_1,F_2\in D)$, then ${\rm Sol}(F_2)$
gives a $\pi_1(\mathbb{C}-S)$-submodule of  ${\rm Sol}(E)$.
\end{proof} % \hfill$\square$ 

\subsection{Shift operators and shift relations}\label{GenShiftShift}
In this and the next subsections, we study shift operators for  differential equations with an accessory parameter $ap$. When $ap$ is specified as a function of the local exponents, or the differential equation is rigid, just forget $ap$.
\begin{dfn}\label{DefShift}In general,
  let $H(e,ap)$ be an operator of order $n$
  with the local exponents $e=(e_1,\dots)$ and a  parameter $ap$, and ${\rm Sol}(H(e,ap))$ its solution space. For a shift 
$$sh_+:e\to e_+,\quad (e_+)_i=e_i+n_i,\quad n_i\in\mathbb{Z},$$
  a non-zero operator $P\in D$ of order lower than $n$   sending 
$${\rm Sol}(H(e,ap))\quad \ {\rm to} \quad {\rm Sol}(H(e_+,ap_+)),$$
for some $ap_+$, is called
  a {\it shift operator} for the shift $sh_+$ 
  and is denoted by $P_{+}$. A shift operator for the shift 
$sh_-:e\to e_-,\ (e_-)_i=e_i-n_i$   is denoted by $P_{-}$.
\end{dfn}

Here we make an important assumption:
\par\smallskip\noindent
{\bf Assumption:} $ap_+=ap-\alpha(e)$, where $\alpha$ is a polynomial in $e$.\par\smallskip
Without this, we can not go further; we can not define S-values, which play an important role in studying reducibility of the equations.
For every shift operator, we can assume that the  coefficients are polynomials of $(e,ap)$ free of common factors.
\begin{remark}
   When a differential {\it equation} in question is $Hu=0$, by multiplying a non-zero polynomial to the {\it operator} $H$, we can assume that the coefficients of $H$ has no poles. However, shift operators may have poles as functions of $x$.
  \end{remark}
Since $P_\pm\in D$, we have

\begin{lemma}\label{pi1morphsim}The shift operators
  are $\pi_1(\mathbb{C}-S)$-morphisms, {\it i.e.},
  they commute with the $\pi_1(\mathbb{C}-S)$-action.
\end{lemma}
Suppose a shift operator $P_{+}\in D$ for a shift $sh_+$ exists.
Since $H(e_+,ap_+)\circ P_{+}$ is divisible from right by $H(e,ap)$,
there is an operator $Q_{+}\in D$ satisfying the {\it shift relation}:
  $$(EPQE):\quad H(e_+,ap_+)\circ P_{+}=Q_{+}\circ H(e,ap).$$
Conversely, if there is a pair of non-zero operators $(P_{+},Q_{+})\in D^2$
of order smaller than $n$ satisfying this relation,
then $P_{+}$ is a shift operator for the shift $sh_+$.
We often call also the pair $(P_{+},Q_{+})$
the shift operator for $sh_+$.
  Lemma \ref{inverseofP} implies
\begin{prp}\label{inverseofshiftop}If $H(e,ap)$ is irreducible and $P_{+}$
%  $($resp. $P_{-}\,)$ 
exists then the inverse operator $P_{-}$
%  $($resp. $P_{+}\,)$ 
exists. More precisely, 
$$\begin{array}{ll}
P_+(e):& {\rm Sol}(H_6(e,ap))\rightarrow {\rm Sol}(H_6(e_+,ap_+)),\quad \ ap_+=ap-\alpha(e),\\%={\rm Sol}(H_6(e+n,ap-\alpha(e))),\\
P_-(e):& {\rm Sol}(H_6(e,ap))\rightarrow {\rm Sol}(H_6(e_-,ap_-)),\quad \ ap_-=ap+\alpha(e-n),\end{array}
$$
where 
%$S(e,ap)={\rm Sol}(H_6(e,ap))$, 
$e_\pm=e\pm n$. Same for $P_-$ and $P_+$.
\end{prp}

\subsection{S-values}\label{GenShiftS}
Consider  compositions of the two  shift operators in the previous subsection:
$$P_+(e_-,ap_-)\circ P_-(e,ap):{\rm Sol}(H(e,ap))\to {\rm Sol}(H(e_-,ap_-))\to {\rm Sol}(H(e,ap)),$$
and
$$P_-(e_+,ap_+)\circ P_+(e,ap):{\rm Sol}(H(e,ap))\to {\rm Sol}(H(e_+,ap_+))\to {\rm Sol}(H(e,ap)),$$
and assume that these maps are constants (times the identity)  independent of $ap$.
\begin{dfn}These constants will be called the {\it S-values} for $sh_\mp$, and are denoted as 
$$Sv_{sh_-}=P_+(e_-,ap_-)\circ P_-(e,ap) {\rm\quad mod\quad} H(e,ap)$$
and
$$Sv_{sh+}=P_-(e_+,ap_+)\circ P_+(e,ap) {\rm\quad mod\quad} H(e,ap).$$
\end{dfn}
\begin{prp}\label{twoSvalues} The two S-values are related as
  $$Sv_{sh_-}(e)=Sv_{sh_+}(e_-).$$
\end{prp}
\begin{proof}  
Consider the product of three operators:
\[ \begin{array}{l}  
  P_{+}(e_-,ap_-)\circ P_{-}(e,ap)\circ P_{+}(e_-,ap_-) : \\[2mm]
 \hskip60pt
{\rm Sol}(H(e_-,ap_-)) \to {\rm Sol}(H(e,ap))\to {\rm Sol}(H(e_-,ap_-)) \to {\rm Sol}(H(e,ap)).
\end{array}\]
The product of the left two is a constant $Sv_{sh_-}(e)$,
and that of the right two is a constant $Sv_{sh_+}(e_-)$.
\end{proof} %\hfill $\square$

\begin{prp}\label{Sred}   
If for some $e=\epsilon$, $Sv_{sh_+}(\epsilon)=0$ $($resp. $Sv_{sh_-}(\epsilon)=0)$, then $H(\epsilon,ap)$ and $H(\epsilon_+,ap_+)$ $($resp. $H(\epsilon_-,ap_-))$ are reducible. If  $Sv_{sh_+}(\epsilon)\not=0$ $($resp. $Sv_{sh_-}(\epsilon)\not=0)$, then $P_{sh_+}$ $($resp. $P_{sh_-})$ gives an isomorphism: ${\rm Sol}(H(\epsilon,ap))\to {\rm Sol}(H(\epsilon_+,ap_+))$ $($resp. ${\rm Sol}(H(\epsilon,ap))\to {\rm Sol}(H(\epsilon_-,ap_-)))$ as $\pi_1(\mathbb{C}-S)$-modules.
\end{prp}
\begin{proof} % \noindent {\sl Proof.}
  Shift operators are, by definition, non-zero; this leads to the first statement. Lemma \ref{pi1morphsim}  implies the second statement.
  \end{proof} % \hfill$\square$ 

\subsection{When $ap$ is a function of $e$}
For a given differential equation $H(e,ap)$, suppose the accessory parameters $ap$ are functions $ap(e)$ of the local exponents $e$; put $G(e)=H(e,ap(e))$. \blue{We can now discuss shift operators without worrying about the change of accessory parameters.}

\subsubsection{Uniqueness of shift operators}Paraphrasing \cite[Proposition 2.13]{vPS}, we have
\begin{prp}\label{uniquenessofshiftop} If $G(e)$ is irreducible and if a shift operator $P$ exists for a shift $sh:e\to e'$, then it is unique up to multiplicative constant. % including parameters.
\end{prp}
\begin{proof}
 Suppose there are two shift operators $P_1$ and $P_2$
that map ${\rm Sol}(G(e))$ to ${\rm Sol}(G(e'))$.
Let $R_1$ denote the inverse operator of $P_1$, as given in Proposition 4.9.
Then the composition $R_1 P_2$ % : {\rm Sol}(G(e)) \to {\rm Sol}(G(e))$ 
is a linear operator on ${\rm Sol}(G(e))$,
and hence admits an eigenvalue $c \in \mathbb{C}$ with corresponding eigenvector $0 \neq u \in {\rm Sol}(G(e))$.
That is, $R_1 P_2 u = c u$.
Applying $P_1$ to both sides, % and using that
% $P_1$ is the inverse of $R_1$,
we have $P_2 u = P_1 R_1 P_2 u = P_1 c u = c P_1 u$.
\red{This yields $(P_2- c P_1) u=0$.
Therefore $D G(e) \subset D G(e) + D(P_2-c P_1) \subsetneq D$.
Since $G(e)$ is irreducible, this implies
$D G(e) = D G(e) + D(P_2-c P_1)$,
and hence $P_2 - c P_1 \in D G(e)$.}
\end{proof}
\comment{
\begin{proof} 
  Suppose there are two shift operators $P_1$ and $P_2$ sending ${\rm Sol}(G(e))$ to ${\rm Sol}(G(e')$. Let $R_2$ be the inverse operator of $P_2$ modulo $G(e)$. The operator $R_2P_1:{\rm Sol}(G(e))\to{\rm Sol}(G(e))$ is equivalent modulo $G(e)$ to an operator of order lower than the order of $G(e)$. Since it is not zero, by Schur's lemma, it is an isomorphism. Choosing an eigenvector $f$ with eigenvalue $c$, we have $R_2P_1f=cf$, that is, $(R_2P_1-c)f=0$. This implies $R_2P_1=c$, that is, $P_1=cP_2$.
  \end{proof}
}

\subsubsection{Composition of shift operators}
\begin{lemma}\label{composition} Let $G$ be a differential operator with local exponents $e$.
For given shift operators and shift relations \blue{for two shifts $e_1\to e_2$ and $e_2\to e_3$ as}
\[\begin{array}{l}
 G(e_2)\circ P(e_1\rightarrow e_2) = Q(e_1\rightarrow e_2)\circ G(e_1), \\[1mm]
 G(e_3)\circ P(e_2\rightarrow e_3) = Q(e_2\rightarrow e_3)\circ G(e_2), 
\end{array}
\]
define the composed operators
\[\begin{array}{l}
P(e_1\rightarrow e_3):= P(e_2\rightarrow e_3)\circ P(e_1\rightarrow e_2), \\[1mm]
Q(e_1\rightarrow e_3):= Q(e_2\rightarrow e_3)\circ Q(e_1\rightarrow e_2). 
\end{array}
\]
Then they satisfy
\[G(e_3)P(e_1\rightarrow e_3) = Q(e_1\rightarrow e_3)G(e_1),\]
\blue{for the composed shift $e_1\to e_3$.}
\end{lemma}
In view of this lemma, we may consider %the product of two $P$'s 
the composition of the maps
$P(e_1\rightarrow e_2):{\rm Sol}(G(e_1))\longrightarrow {\rm Sol}(G(e_2))$
and 
$P(e_2\rightarrow e_3):{\rm Sol}(G(e_2))\longrightarrow {\rm Sol}(G(e_3))$
\blue{modulo $G(e_1)$, denoted by $P$, on the space ${\rm Sol}(G(e_1))$.
We solve the equation $G(e_3)P=QG(e_1)$ to get the corresponding operator $Q$.}

\subsubsection{Remote S-values}\label{GenShiftRemote}
 We consider generally a differential operator $G(e)$ with local exponents $e$
and let $P_+(e)$ and $P_-(e)$ be shift operators for the shifts $sh_\pm:e\to e_\pm$:
   $$P_+(e): {\rm Sol}(G(e))\to {\rm Sol}(G(e_+)),\quad P_-(e): {\rm Sol}(G(e))\to {\rm Sol}(G(e_-))$$
satisfying the shift relations
   $$G(e_-)\circ P_-(e)=Q_-\circ G(e),\quad G(e_+)\circ P_+(e)=Q_+\circ G(e),$$
for some $Q_-$ and $Q_+$.
We have seen that we get constant $S(e,-1):=S_{sh_-}$ independent of $x$ such that 
  $$P_+(e_-)\circ P_-(e)=S(e,-1)+R\circ G(e)$$
for some \blue{operator} $R$. Composing these kind of identities,
we get a constant $S(e,-2)$, called a {\it remote S-value}: 
  $$P_+(e_-)\circ P_+(e_{-2})\circ P_-(e_-)\circ P_-(e)=S(e,-2)+R\circ G(e)$$
for some $R$, where $e_{-2}:=(sh_-)^2(e)$.  Comparing this identity with the identity 
  $$P_+(e_{-2})\circ P_-(e_-)=S(e_-,-1)+R\circ G(e_-)$$
for some $R$, 
multiplied by $P_+(e_-)$ on the left and $P_-(e)$ on the right, we get
 $$S(e,-2)=S(e_-,-1)S(e,-1).$$
Continuing this process, we have
\begin{prp}\label{remoteS} In general, define the \text{remote S-value}
  $S(e,-k)$   by
  $$P_+(e_-)\cdots P_+(e_{-(k+1)})P_-(e_{-k})\cdots P_-(e)=S(e,-k)+R\circ G(e)$$
  for some $R$, where $e_{-k}:=(sh_-)^k(e)$.  Then, it is the product of S-values:
 $$S(e,-k)=S(e_{-k+1},-1)\cdots S(e_-,-1)S(e,-1),\quad k=2,\, 3,\, \dots.$$
  Similarly, define  the \text{remote S-value} $S(e,k)$ by
  $$P_-(e_+)\cdots P_-(e_{k})P_+(e_{k-1})\cdots P_+(e)=S(e,k)+R\circ G(e)$$
  for some $R$, where $e_{k}:=(sh_+)^k(e)$.   Then, it is the product of S-values:
  $$S(e,k)=S(e_{k-1},1)\cdots S(e_+,1)S(e,1),\quad k=2,\, 3,\, \dots.$$
\end{prp}

\subsubsection{Relation between $P$ and $Q$}\label{RelationPQ}
Assume an operator $E=E(e)$  has adjoint symmetry: $E(e)^*=E({\rm adj}(e))$ for a linear transformation ${\rm adj}$ \blue{on the space of local exponents}, assume also $E$ admits a shift relation
$$E(\sigma(e))\circ P=Q\circ E(e)$$
for a shift $\sigma$. Taking adjoint, we have
  $$E(e)^*\circ Q^*=P^*\circ E(\sigma(e)^*),\quad{\rm that\ is,}\quad E({\rm adj}(e))\circ Q^*=P^*\circ E({\rm adj}\circ \sigma(e)).$$
Since ${\rm adj}(e)=\sigma\circ {\rm adj}\circ\sigma(e),$ (recall Remark \ref{adjisminus}: $({\rm adj}(e)_j={\rm constant}-e_j$)
we have
  $$Q^*=(-)^\nu P({\rm adj}\circ \sigma(e)),\quad \nu={\rm order}(P)$$
and so we have

\begin{prp}\label{expofQ} If an operator $E(e)$ with the adjoint symmetry
  $E(e)^*=\\E ({\rm adj}(e))$    %%# added  'the' and linebreak
  admits a shift relation $E(\sigma(e))\circ P=Q\circ E(e)$, then 
  $$Q=(-)^\nu P({\rm adj}\circ \sigma(e))^*,\quad \nu={\rm order}(P).$$
\end{prp}

\subsection{Reducibility type and shift operators}\label{GenRed}
We discuss factorization of Fuchsian operators in $D=\mathbb{C}(x)[\partial]$.

\begin{dfn}When $H\in D$ is reducible and factorizes as
  $$H=F_1\circ \cdots\circ F_r,\quad F_j\in D,\quad 0<{\rm order}(F_j)=n_j,\ (j=1, \dots, r),$$
  we say $H$ is {\it reducible of type} $[n_1, \dots, n_r]$;
  we sometimes call $[n_1, \dots, n_r]$ the {\it type of factors}.
We often forget commas, for example, we write [23] in place of [2, 3].
When only a set of factors matters,
we say $H$ is {\it reducible of type} $\{n_1, \dots, n_r\}$. 
\end{dfn}
By repeated use of Lemma \ref{monodinv}, we have

\begin{prp}\label{invsubsp}
  $H$ admits a factorization  $F_1\circ \cdots\circ F_r$
  of type $[n_1, \dots, n_r]$ if and only if ${\rm Sol}(H)$
  has monodromy invariant subspaces
  $${\rm Sol}(H)=S_1\supset S_2\supset\cdots\supset S_r,$$
  with
  $$\dim S_1/S_2=n_1,\ \dim S_2/S_3=n_2,\dots,\  \dim S_r=n_r.$$
\end{prp}

%\textsc{Notation}: If $F_j\ (j=1,\dots,r)$ have no singular points other than the singular points of $H$, we write $$[n_1, \dots, n_r]A0.$$
%\par\smallskip

\blue{Note that} even if the equation $H$ has singularity only at $S=\{0, 1, \infty\}$,
the factors may have singularities out of $S$.

\begin{prp}\label{apparentsing} If $H$ has singularity only at $S$, then the singular points of $F_1$ and $F_r$ out of $S$ are apparent.
\end{prp}
\begin{proof} %\noindent {\sl Proof.}
  For the factor $F_r$, the claim is obvious.
  The claim for $F_1$ follows by taking adjoint. \end{proof} % \hfill$\square$ %\bigbreak 

\begin{remark} The way of factorization is far from unique: in fact,
an operator can have different types of factorization such as 
the shift relation $H'\circ P=Q\circ H$ and the factorizations
\begin{gather*}
  A\circ B = (A\circ f)\circ (f^{-1}\circ B),\ f\in\mathbb{C}(x),\ f\not=0,\\
  \partial^2=\left(\partial+\frac1{x-c}\right)\circ
  \left(\partial-\frac1{x-c}\right),\ c\in \mathbb{C}.
\end{gather*}
Therefore, when we discuss the singularity of the factors of a decomposition,
we usually choose the factors so that they have least number of singular points.
\end{remark}

\noindent
Proposition \ref{Sred} and Proposition \ref{invsubsp} lead to

\begin{prp}\label{FactorType}
  Suppose $H(e)$ and $H(e_\pm)$ are connected by shift relations. If $Sv_+(\epsilon)\not=0$ (resp. $Sv_-(\epsilon)\not=0$) for some $e=\epsilon$, then $H(\epsilon)$ and $H(\epsilon_+)$ (resp. $H(\epsilon_-)$) admit the factorization of the same type.
\end{prp}

\begin{thm}\label{red_atoap}
  Assume $H$ and $H'$ are connected by the shift relation $H'P=QH$.
  If $H$ is reducible, so is $H'$. If $H'$ is reducible, so is $H$.
\end{thm}
\begin{proof}Assume $H$ is reducible:
  \[H=F_1\circ F_2,\quad n_j={\rm order}(F_j),\quad j=1,\, 2,\]
and $F_2$ is irreducible. Then, considering 
the dimension of $P({\rm Sol}(F_2))$, we have three cases:

\smallskip
(1)\quad $\dim P({\rm Sol}(F_2))=n_2$,

\smallskip
(2) \quad $0<\dim P({\rm Sol}(F_2))<n_2$,

\smallskip
(3)\quad $P({\rm Sol}(F_2))=0$.
\smallskip

In the first case, 
$H'$ has an $n_2$-dimensional solution space $P({\rm Sol}(F_2))$, and, therefore, it is divisible by an irreducible operator of order $n_2$. Thus $H'$ is reducible.

The second case does not occur because
the kernel of $P$ is a nontrivial invariant subspace of ${\rm Sol}(F_2)$ and this contradicts to the irreducibility of $F_2$.

Assume the third case; we write $P$ as $P=P_1\circ F_2$
and divide both sides of $H'P=QH$ by $F_2$. Then, we have
  \[H'\circ P_1=Q\circ F_1.\]
Since ${\rm order} (P)<n=n_1+n_2$, we see that
${\rm order} (P_1) < n_1$ and that $P_1({\rm Sol}(F_1))\not=0$.
Thus ${\rm Sol}(H')$ admits a non-trivial invariant subspace,
which implies that $H'$ is reducible.
The latter statement is obtained by taking adjoint. \end{proof}
\begin{remark}
If $Sv_-(e)\ (=Sv_+(e_-))=0$, shift operators
$$P_-(e):\ {\rm Sol}(H(e))\longleftrightarrow{\rm Sol}(H(e_-))\ :P_+(e_-)$$
are not bijective. So the reducible types of $H(e)$ and $H(e_-)$ may be different. \footnote{In general, for a reducible operator $H$,  reducible type is not unique (typical example is $H:=E'\circ P=Q\circ E$). 
However for the operator $H_j(e)$ having generic exponents $e$ but with one reducibility condition, the reducible type is unique. So `different' makes sense.}
In many cases (all the equations of order greater than 2 in this paper) they are actually different, but not always (see e.g. \cite{EOY}).
\end{remark} 
\subsection{From $H_6$ to $H_5$ and $H_3$ by factorization}\label{fromH6toH53byfac}%\quad Papers/G6oG5twoways}
\blue{Recall that middle convolutions send $H_6$ to $H_5$ (Remark \ref{fromH6toH5}), and $H_6$ to $H_3$ (\S \ref{fromH6toH3}). In this section we show that $H_5$ and $H_3$ can be also obtained from $H_6$ by factorizations.}
\subsubsection{From $H_6$ to $H_5$ by factorization}\label{G6G5}%\quad Papers/G6oG5twoways}
Recall the $(\theta,\partial)$-form of $H_6:=H_6(e,a)=T_0+T_1\partial+T_2\partial^2+T_3\partial^3$. Since
$$ T_0=(\theta+s+2)(\theta+s+1)(\theta+s)B_0,\quad B_0=(\theta+e_7)(\theta+e_8)(\theta+e_9),$$
if $e_9=0$,  $H_6$ is divisible  by $\partial$ from the right. We get, as in \S 1.2,
$$H_5=H_5(e_1,\dots,e_8)=H_6(e_1,\dots,e_8,e_9=0)/\partial.$$

\subsubsection{From $H_6$ to $H_3$ by factorization}\label{fromG6toG3byFactor}\label{fromG6toG3}% \quad Papers/G6toG3byFactor}
When $s=1$,  the coefficients of $H_6$ change as
\[ \def\arraystretch{1.3} \begin{array}{rcl}       %%# added linebreak
  T_0 &=&(\theta+3)(\theta+2)(\theta+1)B_0\\% (\theta+e_7)(\theta+e_8)(\theta+e_9) \\ 
  &=& \partial^3 x^3B_0,\\%  (\theta+e_7)(\theta+e_8)(\theta+e_9),\\
  T_1 \partial &=& (\theta+3)(\theta+2) B_1(\theta,s=1) \partial
  = \partial (\theta+2)(\theta+1) B_1(\theta-1, s=1)\\
  &=& \partial^3 x^2 B_1(\theta-1,s=1),\\
  T_2 \partial^2 &=& (\theta+3) B_2(\theta,s=1) \partial^2
  = \partial^2 (\theta+1) B_2(\theta-2,s=1) \\
&=& \partial^3 x B_2(\theta-2,s=1),\\
  T_3 \partial^3 &=& \partial^3 T_3(\theta-3, s=1).
\end{array}\]
\blue{We have the factorization} $H_6=\partial^3\circ V$, where $V$ is a differential operator of order 3:
$$ V=x^3B_0+x^2B_1(\theta-1)+xB_2(\theta-2)+T_3(\theta-3),\quad e_9=3-e_1-\cdots-e_8.$$
In order to get a relation of $V$ with equation $H_3$, 
we multiply $x^{e_1}(x-1)^{e_4}$ from the right to $V$,
and rename the local exponents as follows. By following  these transformations
by the move of the Riemann scheme $R_V$ of $V$ as
$$\begin{array}{l}R_V=\left(\begin{array}{lll}
  e_1&e_2&e_3\\
  e_4&e_5&e_6\\
  *&e_7&e_8\end{array}\right)
  \to
  \left(\begin{array}{lll}
  0&e_2-e_1&e_3-e_1\\
  0&e_5-e_4&e_6-e_4\\
  *&e_7+e_1+e_4&e_8+e_1+e_4\end{array}\right)
=\left(\begin{array}{lll}
  0&b_1&b_2\\
  0&b_3&b_4\\
  b_7&b_5&b_6\end{array}\right)=R_3,\end{array}
  $$
we see that the transformed equation is $H_3$.

\subsection{Polynomial solutions}
The equation $H_6$ can have polynomial solutions (\S \ref{polynomSolH6}), more generally, we have
\begin{prp}\label{polynomSol}
 Let $H$ be an equation admitting a $(\theta,\partial)$-form. If $H$ can be written as
  $$H=({\rm a\ polynomial\ in\ }\theta)(\theta-m)+({\rm a\ polynomial\ in\ }\theta\ {\rm and}\ \partial)\ \partial$$
  for a non-negative integer $m$, then $H$ is divisible from the right by $\partial-f'/f$, where $f$ is a polynomial of $x$  of degree $\le m$.
\end{prp}
\begin{proof} % \noindent {\sl Proof.}
 $H$ maps the set of polynomials of $x$ of degree  $\le m$ to that of degree  $\le m-1$, %and maps that of degree $\le m$ bijectively to itself,
  so there is such $f$ killed by $H$.
  \end{proof} % \hfill$\square$ \bigbreak 

\par\noindent
A well-known example: the Gauss hypergeometric operator $(\theta+a)(\theta+b)-(\theta+c)\partial$  admits a polynomial solution when $a$ is a non-positive integer (see \S \ref{E2Fact}).
%\begin{remark}
  The zeros of the polynomial solution other than $\{0,1\}$ are
  apparent singular points; a special case of Proposition \ref{apparentsing}.
%\end{remark}

\newpage
\section{The Gauss hypergeometric equation $E_2$}\label{E2}
%\nostcrule
\secttoc
\blue{In order to make clear the story of this and the following papers,} we review some known facts about the Gauss hypergeometric  equation. We start with the hypergeometric operator in $(x,\partial)$-form
  $$E_2=E(a,b,c):=x(x-1)\partial^2+((a+b+1)x-c)\partial +ab,\quad \partial=d/dx.$$
It has singularities at $\{0, 1, \infty\}$, and is symmetric under the exchange $a\leftrightarrow b$. Its $(\theta,\partial)$-form is given as
  $$E(a,b,c)=E_0+E_1\partial,\quad E_0(\theta,a,b)=(\theta+a)(\theta+b),\ E_1(\theta,c)=-(\theta+c).$$
Historically, the hypergeometric series 
  $$F(a,b,c;x)=\sum\frac{(a)_n(b)_n}{(c)_n(1)_n}x^n$$
\blue{studied before the  hypergeometric equation was found. However our main objects $H_6,G_6,\dots$ have no simple expression of local solutions, so we started with the differential equation. } 

\subsection{Exponents at $x=0$ and $x=1$}\label{E2Ex}
To see the local exponents at $x=0$, we use the $(\theta,\partial)$-form.
Apply $E(a,b,c)$ to $u=x^\rho(1+\cdots)$.
Since $E_0$ keeps the local exponents $\rho$,
we neglect it, and see the effect of $E_1$:
  $$E_1\partial u=-(\theta+c)\rho x^{\rho-1}(1+\cdots)=(\rho-1+c)\rho x^{\rho-1}+O(x^{\rho}).$$
       {\sl The local exponents at $x=0$ are determined by the last term $E_1$,
         and are given as $\rho=0$, $1-c$.} (Special case of Proposition \ref{localat0})

% \par\medskip

   Apply the transformation $x\rightarrow 1-x$ in the $(x,\partial)$-form of $E(a,b,c)$. We find the resulting equation coincides with $E(a,b,a+b-c+1)$. Thus the local exponents at $x=1$ are $\{0,c-a-b\}$.

\subsection{Transformation $x\to1/x$ and the local exponents at $x=\infty$}\label{E20toinfty}
Put $x=1/y, w=y\partial_{y}(=-\theta), \partial_y=d/dy$ in the $(\theta,\partial)$-form:
\begin{equation}\label{Ey}E_y=(-w+a)(-w+b)-(-w+c)(-y)w.\end{equation}
Apply this to $u=y^\rho(1+\cdots)$. Since the second term increases the local exponent $\rho$, we neglect it, and see the effect of the first term:
$$(-w+a)(-w+b)y^\rho(1+\cdots)=(-\rho+a)(-\rho+b)y^\rho(1+\cdots).$$
{\sl The local exponents at $x=\infty$ are determined by the first term $E_0$, and are given as $\rho=a$, $b$.} (Special case of Proposition \ref{localat1}) 

Let us see that $E_y$ can be transformed to a Gauss operator. Compose $y^a$ ($a$: one of the local exponents at infinity) from the right 
\[\def\arraystretch{1.2}\setlength\arraycolsep{2pt} \begin{array}{rcl}
E_y y^a&=&y^a\left[\ \{a-(w+a)\}\{b-(w+a)\}-\{c-(w+a)\}(-y)(w+a)\ \right]\\
    &=&y^a\left[\ (-w)(-w+b-a)-(-w+c-a)(-y)(w+a)\ \right].\end{array}
\]
By multiplying $-y^{-a-1}$ to the expression of the last line, we see that 
\[
\begin{array}{l} -\{\ (-w-1)(-w+b-a-1)y^{-1}-(-w+c-a-1)(-)(w+a)\ \}\\
    \quad =(w+a)(w-c+a+1)-(w-b+a+1)(w+1)y^{-1}.\end{array}
\]
 In the last line, we exchanged the first and the second terms. 
Since $\partial_{y}=(w+1)y^{-1}$, the last operator is equal to 
$$E(a,1-c+a,1+a-b)= (w+a)(w-c+a+1)-(w-b+a+1)\partial_{y}.$$
The transformations above from $E(a,b,c)$ to
         $E(a,1-c+a,1+c-b)$ can be visualized by the Riemann schemes as
\[R_2(a,b,c):=
\left(\begin{array}{ccc}
x=0:& 0&1-c\\
x=1:& 0&c-a-b\\
x=\infty:& a&b\end{array}\right)
\to
\left(\begin{array}{cc}
a&b\\
0&c-a-b\\
0&1-c
\end{array}\right)
\to
\left(\begin{array}{cc}
0&b-a\\
0&c-a-b\\
a&1-c+a
\end{array}\right),
\]
\noindent
which is the transformation $R_2(a,b,c)\to R_2(a,1-c+a,1+a-b)$.
Summing up, we have
$$x^{-a-1}E(a,b,c)|_{x\to1/x}\circ x^a=-E(a,1-c+a,1+a-b),$$
where $E(a,b,c)|_{x\to1/x}$ denotes $E_y$ in \eqref{Ey}
with the change $y\to x, w\to \theta$.

\subsection{Adjoint operator of $E_2$}\label{adjE2proof}
The adjoint of $E(a,b,c)=E_0(\theta,a,b)+E_1(\theta,c)\partial$ is computed as
\[\def\arraystretch{1.2}\setlength\arraycolsep{3pt} \begin{array}{rcl}
   E_0(\theta,a,b)^*&=&(-\theta-1+b)(-\theta-1+a)=(\theta+1-a)(\theta+1-b)\\ 
   &=&E_0(\theta,1-a,1-b),\\
(E_1(\theta,c)\partial)^*&=&-\partial E_1^*=-\partial(-1)(-1-\theta+c)=-(\theta+2-c)\partial\\
  &=&E_1(\theta,2-c)\partial,\end{array}\]
and we have
\[E(a,b,c)^*=E(1-a,1-b,2-c).\]

\subsection{Differentiation}\label{E2Der}
The differentiation of any solution $u$ of the Gauss equation 
$E(a,b,c)$ is again a solution of another Gauss equation $E(a+1,b+1,c+1)$.
This is seen by differentiating the hypergeometric series or
by composing $\partial$ and the equation $E$ to see that $u'$ satisfies
the equation with parameter $(a+1,b+1,c+1)$:
Since $\partial\circ \theta=(\theta+1)\circ \partial$,
\[\def\arraystretch{1.2}\setlength\arraycolsep{2pt} \begin{array}{rcl}
\partial\circ E(a,b,c)&=&\partial\circ(E_0(\theta,a,b)+E_1(\theta,c)\partial)
  =(E_0(\theta+1,a,b)+E_1(\theta+1,c)\partial)\circ \partial\\
  &=& (E_0(\theta,a+1,b+1)+E_1(\theta,c+1)\partial)\circ \partial\\
  &=& E(a+1,b+1,c+1)\circ \partial.\end{array}
\]  
In terms of the Riemann scheme, this is expressed as
 \addtolength{\arraycolsep}{-1pt}  %%# temp change of \arraycolsep
\[R_2(a,b,c)=\left( \begin{array}{cc}
0&1-c\\
0&c-a-b\\
a&b\end{array}\right)
\underset{\partial}{\to}
\left(\begin{array}{cc}
0&1-c-1\\
0&c-a-b-1\\
a+1&b+1\end{array}\right)=R_2(a+1,b+1,c+1).\]
 \addtolength{\arraycolsep}{1pt}  %%# temp change of \arraycolsep
The inverse of $\partial$ is obtained as follows:
Write the Gauss equation as 
$$E(a,b,c)=E'\circ \partial-ab,\qquad E'=E'(a,b,c)=x(x1)\partial+(a+b+1)x-c,$$
The derivation of the Gauss series $F(a,b,c;x)$
is $\frac{ab}cF(a+1,b+1,c+1;x)$; hence, we have
$$\frac1cE'(a,b,c)F(a+1,b+1,c+1;x)=F(a,b,c;x),
$$
which means that the operator $\partial$ is read as the shift operator of the parameter shift $(a,b,c)\to(a+1,b+1,c+1)$ and
$E'$ that of the reverse  shift $(a+1,b+1,c+1)\to(a,b,c)$.

\subsection{Shift operators of $E_2$}\label{E2Shift}
The shift operator $P_{a+}$ for the parameter-ascending shift $a\to a+1$ is 
obtained by the following procedure (we write $R_{abc}$ for $R_2(a,b,c)$):
\[ \begin{array}{l}     %%# changed to array 
R_{abc}
%=\left( \begin{array}{cc}
%0&1-c\\
%0&c-a-b\\
%a&b\end{array}\right) % & 
\underset{x^a}{\to}
\left( \begin{array}{cc}
a&a+1-c\\
0&c-a-b\\
0&b-a\end{array}\right) 
\underset{\partial}{\to}
\left( \begin{array}{cc}
a-1&a-c\\
0&c-a-b-1\\
2&b-a+1\end{array}\right) %\\  \noalign{\medskip}%&
 \underset{x^{1-a}}{\to}
\left( \begin{array}{cc}
0&1-c\\
0&c-a-b-1\\
a+1&b\end{array}\right).
\end{array} \]
Thus, we have the operator
$$P_{a+}=x^{1-a}\circ \partial\circ x^a =x^{1-a}\circ (ax^{a-1}+x^a\circ \partial)=x\partial+a.$$
The descending operator $P_{-a}$ for $a\to a-1$ is obtained by 
\[ \begin{array}{ll}
  R_{abc}  &
\underset{X}{\to}
\left( \begin{array}{cc}
c-a&1-a\\
a+b-c&0\\
a-b&0\end{array}\right)\underset{\partial}{\to}
\left( \begin{array}{cc}
c-a-1&-a\\
a+b-c-1&0\\
a-b+1&2\end{array}\right)\\ \noalign{\medskip}&
\underset{X^{-1}x(x-1)}{\to}
\left( \begin{array}{cc}
0&1-c\\
0&c-a-b+1\\
a-1&b\end{array}\right),
\end{array} \]
where $X=x^{c-a}(x-1)^{a+b-c}$. Hence, we get the operator $-P_{a-}$,
where $P_{a-}=x(1-x)\partial+c-a-bx$,
which is a little more complicated than that for $a\to a+1$.
When $c\to c-1$, we see that
\[ \begin{array}{ll}
R_{abc} &\underset{x^{c-1}}{\to}
\left( \begin{array}{cc}
c-1&0\\
0&c-a-b\\
a-c+1&b-c+1\end{array}\right)\underset{\partial}{\to}
\left( \begin{array}{cc}
c-2&0\\
0&c-a-b-1\\
a-c+2&b-c+2\end{array}\right)\\  \noalign{\medskip}
&\underset{x^{2-c}}{\to}
\left( \begin{array}{cc}
0&2-c\\
0&c-a-b-1\\
a&b\end{array}\right)
\end{array} \]
and we get a descending operator $P_{c-}=x\partial+c-1$. 
For the ascending case $c\to c+1$, we see that
\[  \begin{array}{ll}
  R_{abc}&  \underset{(x-1)^{a+b-c}}{\to}
\left( \begin{array}{cc}
0&1-c\\
a+b-c&0\\
c-b&c-a\end{array}\right)\underset{\partial}{\to}
\left( \begin{array}{cc}
0&-c\\
a+b-c-1&0\\
c-b+1&c-a+1\end{array}\right)\\
&  \underset{(x-1)^{1+c-a-b}}{\to}
\left( \begin{array}{cc}
0&-c\\
0&1+c-a-b\\
a&b\end{array}\right);
\end{array}\]
thus we get an ascending operator $P_{c+}=(x-1)\partial+a+b-c$.

By changing the notation of parameters
from $(a,b,c)$ to $(e_1,e_2,e_3,s=1-e_1-e_2-e_3)$, \blue{we repeat the process above as follows: }
\[
\begin{array}{ll}
  R_2=\left( \begin{array}{ccc}
x=0:&0&e_1\\
x=1:&0&e_2\\
x=\infty:&s&e_3\end{array}\right)
&  
\underset{x^s}{\to}
\left( \begin{array}{cc}
s&e_1+s\\
0&e_2\\
0&e_3-s\end{array}\right)\underset{\partial}{\to}
\left( \begin{array}{cc}
s-1&e_1+s-1\\
0&e_2-1\\
2&e_3-s+1\end{array}\right)\\  \noalign{\medskip}&
\underset{x^{1-s}}{\to}
\left( \begin{array}{cc}
0&e_1\\
0&e_2-1\\
s+1&e_3\end{array}\right)
\end{array} \]
  and, therefore, we get the shift operator $P_{2-}:=x\partial+s$ for the shift $e_2\to e_2-1$. Since
  \[ \begin{array}{l}       %%# changed to array
R_2 %&
 \underset{X}{\to}
\left( \begin{array}{cc}
e_2+e_3&e_{123}\\
-e_2&0\\
s-e_3&0\end{array}\right)\underset{\partial}{\to}
\left( \begin{array}{cc}
e_2+e_3-1&e_{123}-1\\
-e_2-1&0\\
s-e_3+1&2\end{array}\right) %\\ \noalign{\medskip}&
 \underset{X^{-1}x(x-1)}{\to}
\left( \begin{array}{cc}
0&e_1\\
0&e_2+1\\
s-1&e_3\end{array}\right),
\end{array} \]
  where $e_{123}=e_1+e_2+e_3,X=x^{e_2+e_3}(x-1)^{-e_2}$, we have $-P_{2+}$,
  where $P_{2+}:=x(1-x)\partial+e_2+e_3-e_3x$ is the shift operator for the shift $e_2\to e_2+1$. Since
\[ \begin{array}{l}         %%# changed to array
R_2 % &
 \underset{x^{-e_1}}{\to}
\left( \begin{array}{cc}
-e_1&0\\
0&e_2\\
s+e_1&e_3+e_1\end{array}\right)\underset{\partial}{\to}
\left( \begin{array}{cc}
-e_1-1&0\\
0&e_2-1\\
s+e_1+1&e_3+e_1+1\end{array}\right) %\\ \noalign{\medskip}&
 \underset{x^{e_1+1}}{\to}
\left( \begin{array}{cc}
0&e_1+1\\
0&e_2-1\\
s&e_3\end{array}\right),
\end{array} \]
we have the shift operator $P_{1+2-}:=x\partial-e_1$  for the shift $(e_1,e_2)\to(e_1+1,e_2-1)$. Since
\[ \begin{array}{l}     
R_2  %&
\underset{(x-1)^{-e_2}}{\to}
\left( \begin{array}{cc}
0&e_1\\
-e_2&0\\
s+e_2&e_3+e_2\end{array}\right)\underset{\partial}{\to}
\left( \begin{array}{cc}
0&e_1-1\\
-e_2-1&0\\
s+e_2+1&e_3+e_2+1\end{array}\right)%\\ \noalign{\medskip}&
\underset{(x-1)^{e_2+1}}{\to}
\left( \begin{array}{cc}
0&e_1-1\\
0&e_2+1\\
s&e_3\end{array}\right),
\end{array}\]
we have the shift operator  $P_{1-2+}:=(x-1)\partial-e_2$  for the shift $(e_1,e_2)\to(e_1-1,e_2+1)$.

The shift operators relative to $\{a,b,c\}$ and $\{e_1, e_2,e_3\}$ are
related as
$$P_{2-}=P_{a+},\quad P_{2+}=P_{a-},\quad P_{1+2-}=P_{c-},\quad P_{1-2+}=P_{c+}.$$

\begin{remark}\label{E2ShiftGen}
The general shift operators for 
$$ {\rm Sol}(E(a,b,c))\to{\rm Sol}(E(a+p,b+q,c+r)),\quad p,q,r\in \mathbb{Z}$$
are given in \cite{Eb1, Eb2}. We thank H. Ando for his Maple program computing them.
\end{remark}

\subsubsection{Relation between $P$ and $Q$}\label{E2ShiftPQ}
Let us see Proposition \ref{expofQ} for $E_2$. By taking adjoint of the shift relation, for example,
 $$ E(a+1,b,c)\circ P_{a+}=Q_{a+}\circ E(a,b,c),\quad P_{a+}=x\partial+a,$$
we have
 $$E(1-a,1-b,2-c)Q_{a+}^*=P_{a+}^*E(-a,1-b,2-c),$$
since the adjoint of $E(a,b,c)$ is $E(1-a,1-b,2-c)$. Hence we have
 $$Q_{a+}^*=-P_{a+}(-a,1-b,2-c)=-(x\partial-a)\quad{\rm so}\quad Q_{a+}=x\partial+1+a.$$ 
In this way $Q_{a+}$ can be computed from $P_{a+}$. List of pairs of shift operators $(P,Q)$:
\[\def\arraystretch{1.2}
\begin{array}{llllll}
P_{a+}&=&x\partial+a,&Q_{a+}&=&x\partial+a+1,\\
P_{a-}&=&x(x-1)\partial+a+bx-c,\quad
&Q_{a-}&=&x(x-1)\partial+a+bx-c+x-1,\\  %=&\frac1x\circ P_{a-}\circ x,\\
P_{c+}&=&(x-1)\partial+a+b-c,&Q_{c+}&=&P_{c+},\\
P_{c-}&=&x\partial+c-1,&Q_{c-}&=&P_{c-}.
\end{array}
\]

\subsection{S-values and reducibility conditions of $E_2$}\label{E2S}
Since $P_{a+}=x\partial+a$, $P_{a-}=x(x-1)\partial+bx+a-c$, and
$E(a,b,c)=x(x-1)\partial^2+\cdots$,
the S-value $Sv_{a-}$ for the shift $a\to a-1\to a$ is computed as
  $$P_{a+}(a-1)\circ P_{a-}(a)-xE(a,b,c)=(a - 1)(a - c).$$
Similarly, we get
  $$Sv_{b-}=(b-1)(b-c),\quad Sv_{c-}=(b - c + 1)(a - c + 1).$$
Thus $E(a,b,c) $ is reducible if one of 
  $$a-1,\ a-c,\ b-c+1,\ a-c+1 $$
vanishes, and we get by Theorem \ref{red_atoap} 
the well known condition of reducibility
  $$a,\ b,\ c-a,\ c-b \in\ \mathbb{Z}.$$

\subsection{Reducibility conditions and the Euler integral representation}\label{E2Red}
The identity
    $$E(a,b,c)\varphi=-b\frac{\partial}{\partial s}\left(\frac{s(1-s)}{x-s}\varphi\right),\quad \varphi=s^{b-c}(1-s)^{c-a-1}(x-s)^{-b}$$
implies that the function defined by the integral
    $$F_\gamma(x)=\int_\gamma\varphi \, ds$$
along a closed path $\gamma$ \footnote{$\gamma$ is topologically closed and the values of $\varphi$ at the starting point and the ending point agree.}  gives a solution to $E(a,b,c)$.
The integrand has exponents $$b-c,\quad c-a-1,\quad -b, \quad a$$
at $0$, $1$, $x$, $\infty$, respectively.
If one of the exponents is a negative integer,
then we can choose as $C$ a small loop around this point,
and $F_C(x)\not=0$ generates an invariant subspace of the solution space,
which means the equation is reducible.    

\subsection{Reducible cases of $E_2$}\label{E2Fact}
When $E(a,b,c)$ is reducible, we see its factorization,
which gives examples of the discussion in \S \ref{GenRed}. % and \ref{RedOrd1}. 
Recall the first four solutions among
the Kummer's 24 solutions (cf. \cite{Er}):
   $$\begin{array}{ccl}
  {\rm I}&:&F(a,b,c;x),\\[3pt]
  {\rm II}&:&(1-x)^{c-a-b}F(c-a,c-b,c;x),\\[3pt]
  {\rm III}&:&x^{1-c}F(a-c+1,b-c+1,2-c;x),\\[3pt]
  {\rm IV}&:&x^{1-c}(1-x)^{c-a-b}F(1-a,1-b,2-c;x).
\end{array}$$
Note that the parameters of hypergeometric series in I and IV
as well as II and III are related; recall the adjoint relation:
$$  E^*(a,b,c)=E(1-a,1-b,2-c),\quad E^*(c-a,c-b,c)=E(a-c+1,b-c+1,2-c).$$

When the operator $E(a,b,c)$ is
reducible ($a$, $b$, $c-a$, or $c-b\in \mathbb{Z} $), 
$E$ factorizes into $F_1\circ F_2$,
$$F_2=\partial-\frac{G'}G, \quad G=x^\mu(x-1)^{\nu}g,$$
\blue{where
  $$(\mu,\nu)=(0,0),\ (0,c-a-b),\ (1-c,0),\ (1-c,c-a-b),$$
  according to the types ${\rm I},\dots,{\rm IV}$ of $G$, respectively},
and $g$ is a hypergeometric polynomial:
  $$\begin{array}{lcc}
 {\rm condition}& {\rm type\ of\ }G&{\rm degree\ of\ the\ polynomial\ }g\\[3pt]
    a=\cdots,-2,-1& {\rm I}& -a\\[2pt]
  a=0& {\rm I}&0\\[2pt]
  a=1& {\rm IV}&0\\[2pt]
  a=2,3,\cdots&{\rm IV}&a-1\\ [3pt]
  &&\\%\end{array}$$  $$\begin{array}{llc}
   % {\rm condition}& {\rm type\ of\ }G&g\\
  c-a=\cdots,-2,-1& {\rm II}&-(c-a)\\[2pt]
  c-a=0& {\rm II}&0\\[2pt]
  c-a=1& {\rm III}&0\\[2pt]
  c-a=2,3,\cdots&{\rm III}&c-a-1
 \end{array}$$
The zeros of $g$ are the apparent singular points of $F_2$, and so of $F_1$. 
Therefore, the apparent singularities are the zeros of
the hypergeometric series (cf. Proposition \ref{polynomSol}).

\newpage

\section{Shift operators of $H_6$}\label{shiftopH6}\secttoc
We use the following notation to denote blocks of local exponents as
\[ \def\arraystretch{1.1} \begin{array}{l}
  {\bm e}_1=(e_1,e_2,e_3),\ {\bm e}_4=(e_4,e_5,e_6),\ {\bm e}_7=(e_7,e_8,e_9),\ 
  e=({\bm e}_1,{\bm e}_4,{\bm e}_7),\\ {\bf 1}=(1,1,1), \ 
  {\bm e}_1\pm{\bf 1}=(e_1\pm1,e_2\pm1,e_3\pm1),\ \dots,
\end{array}\]
and call the shifts generated by 
$$ {\bm e}_1\to {\bm e}_1\pm {\bm 1},\quad{\bm e}_4\to {\bm e}_4\pm {\bm 1},\quad{\bm e}_7\to {\bm e}_7\pm {\bm 1}$$
the {\it block shifts.} In this section we find the shift operators of $H_6$ for the block shifts:
  $$%\begin{array}{ll}
  sh_1:{\bm e}_1\to{\bm e}_1-{\bf1},\quad
  sh_2:{\bm e}_4\to{\bm e}_4-{\bf1},\quad
  sh_3:e\to ({\bm e}_1-{\bm 1},{\bm e}_4-{\bm1},{\bm e}_7+{\bm 1}),
$$%\end{array} $$
  (Note that $sh_1^{-1}\circ sh_2^{-1}\circ sh_3:{\bm e}_7\to{\bm e}_7+{\bf1}.$)

The move of the Riemann scheme as we saw in \S\ref{E2ShiftPQ} for the Gauss equation $E_2$, for example,
$$\begin{array}{l}\quad\ 
\left(\begin{array}{ccccl}
 x=0:& 0&1&2&e_1\dots\\
 x=1:& 0&1&2&e_4\dots\\
 x=\infty:& s&s+1&s+2&e_7\dots
\end{array}\right)
  \underset{x^s}{\to}
\left(\begin{array}{cccl}
  s&s+1&s+2&e_1+s\dots\\
  0&1&2&e_4\dots\\
  0&1&2&e_7-s\dots
\end{array}\right)\\[7mm]
  \underset{\partial}{\to}
\left(\begin{array}{cccl}
  s-1&s&s+1&e_1+s-1\dots\\
  0&1&2&e_4-1\dots\\
  2&3&4&e_7-s+1\dots
\end{array}\right)
\underset{x^{1-s}}{\to}
\left(\begin{array}{cccl}
  0&1&2&e_1-1\dots\\
  0&1&2&e_4-1\dots\\
  s+1&s+2&s+3&e_7+1\dots
\end{array}\right)
\end{array}$$
suggests $P_{0-0}=x\partial +s$ \blue{(refer to Definition \ref{E6PQnotation} for index notation of $P$)}. More generally, 
\begin{thm}\label{shopH6}For every block shift $sh$, the equation $H_6(e,T_{10}=u)$ admits a shift operator $(P,Q,\alpha):$
  $$H_6(sh(e), u-\alpha)\circ P=Q\circ H_6(e,u).$$ For a set of generators $\{sh_1,sh_2,sh_3\},$   the shift operators are given as follows:
  $$\begin{array}{llll}
sh_1: &P_{-00}=(x-1)\partial+s,&Q_{-00}=(x-1)\partial+3+s,&\alpha_1=s_{13}+s_{23}+1,\\[2mm]
sh_2: &P_{0-0}=x\partial+s,    &Q_{0-0}=x\partial+3+s ,   &\alpha_2=0,\\[2mm]
sh_3: &P_{--+}=\partial,       &Q_{--+}=\partial,         &\alpha_3,\end{array}$$
  where
  $$\alpha_3=20-s_{11}^2/3-2s_{11}s_{13}/3+s_{12}^2/3-s_{13}^2/3
  -2s_{11}+7s_{13}+s_{21}-s_{22}+2s_{23}.$$
\end{thm}
\begin{proof}%\noindent{\bf Sketch of the proof: }  
The first one is obtained as follows: Put
\[P=(x-1)\partial+s,\quad Q=(x-1)\partial+ q\]
and solve the equation 
\[H_6(sh_1(e),u-\alpha)\circ P=Q\circ H_6(e,u)\]
with respect to the set of unknowns $\{\alpha,q\}$. 
Solution is
\[\alpha=s_{13}+s_{23}+1,\quad  q = s+3.\]
The second and the third ones are obtained similarly.\end{proof}

%%% deleting comments
 
\subsection{Inverse shift operators and S-values of $H_6$}
We have determined the shift operators of the equation $H_6$ 
for the shifts $sh_1$, $sh_2$ and
$sh_3$ and denoted them as \blue{$(P_{-00},Q_{-00}),\dots, (P_{--+},Q_{--+})$.} Generally,
we introduce notation as follows.
\begin{dfn}\label{E6PQnotation}
%  The shift operator for the shift ${\bm e}_1\to{\bm e}_1\pm{\bf 1}$
%  is denoted by $P_{\pm00}$,
%  for the shift ${\bm e}\to ({\bm e}_1+{\bm 1},{\bm e}_4-{\bm1},{\bm e}_7) $
%  by $P_{+-0}$, and so on.
If $(P,Q,\alpha)$ solves the equation
  $$H_6({\bm e}_1+\epsilon_1{\bf 1},{\bm e}_4+\epsilon_4{\bf 1},
  {\bm e}_7+\epsilon_7{\bf 1},u-\alpha)\circ P
  =Q\circ H_6(e,u),\quad \epsilon_1,\epsilon_4,\epsilon_7 =-1, 0, 1,$$
then the operators $P$ and $Q$ are denoted as
$P_{\delta_1\delta_4\delta_7}$ and $Q_{\delta_1\delta_4\delta_7}$,
where $\delta_i={-}, {0}, {+}$ according as $\epsilon_i=-1, 0, 1$.
For example, 
for the shift 
$e\to({\bm e}_1+{\bf 1},{\bm e}_4+{\bf 1},
{\bm e}_7-{\bf 1})$, the shift operators are $P_{++-}$ and $Q_{++-}$.
\end{dfn}
%\subsubsection{Simple shift operators}\label{simplecases}
\subsubsection{$P_{++-}$ and the S-value $Sv_{--+}=P_{+--}\circ P_{--+}$ for $H_6$}\label{P_0p0}
While the operator $P_{--+}$ defines a map from ${\rm Sol}(H_6({e},u))$
to ${\rm Sol}(H_6({\bm e}_1-{\bf 1},{\bm e}_4-{\bf 1},{\bm e}_7+{\bf 1},u-\alpha))$,
its inverse map is given by the operator $P_{++-}$ evaluated at
$({\bm e}_1-1,{\bm e}_4-1,{\bm e}_7+1,u-\alpha)$ and the composition gives the
S-value; refer to \ref{GenShiftS}. We call the operator $P_{++-}$
itself the inverse of $P_{--+}$ for simplicity in the following.
In view of this property, we see that
\[P_{++-}({\bm e}_1-1,{\bm e}_4-1,{\bm e}_7+1)
=(H_6-p_0)/\partial=x^3(x-1)^3\partial^5+\cdots,\]
where $p_0$ is the constant term of the $(x,\partial)$-form
of $H_6=x^3(x-1)^3\partial^6+p_5\partial^5+\cdots+p_1\partial +p_0$
and that the S-value in this case, which we denote as $Sv_{--+}$, is
\[\begin{array}{ll}Sv_{--+}&=\
  P_{++-}({\bm e}_1-{\bm1},{\bm e}_4-{\bm1},{\bm e}_7+{\bm1})\circ P_{--+}\\
&=\  H_6-p_0 \equiv -p_0=-s(s+1)(s+2)e_7e_8e_9\quad{\rm mod}\ H_6.\end{array}
\]

%This procedure works also for the shift $sh_2$ (\S \ref{P_p00}) and we have
%\[P_{0+0}=x^5 (x-1)^3 \partial^5+\cdots.\]
%%% sasaki equationsG6 %% yoshida inverse-S-value
\subsubsection{$P_{0+0}$ and the S-value $Sv_{0+0}=P_{0-0}\circ P_{0+0}$ for $H_6$}
The inverse of $P_{0-0}$, denoted $P_{0+0}$, is obtained by the relation
$$P_{0-0}({\bm e}_4+{\bf 1})\circ P_{0+0}-U\circ H_6(e)+{\rm constant}$$
 for some differential operator $U$; the constant is the S-value $Sv_{0+0}$. 
In this case, $P_{0-0}=x\partial+s$ and $H_6=x^3(x-1)^3\partial^6+\cdots$; we set
\[P_{0+0}=x^5(x-1)^3\partial^5+\cdots,\quad {\rm and}\quad U=x^3. \]
and solve
\begin{equation}\label{inverse_and_S-value}
P_{0-0}({\bm e}_4={\bm e}_4+{\bm 1})\circ P_{0+0}=x^3H_6+Sv_{0+0},
\end{equation}
to find $P_{0+0}$ and $Sv_{0+0}$. 
The $(\theta,\partial)$-form of $H_6$:
$$H_6=T_0+T_1\partial+T_2\partial^2+T_3\partial^3,$$
implies that $x^3H_6$ has $(x,\theta)$-form as:
$$x^3H_6=x^3T_0+x^2\theta T_1(\theta-1)+x\theta(\theta-1)T_2(\theta-2)+\theta(\theta-1)(\theta-2)T_3(\theta-3).$$
Note that this expression has no constant (independent of $x,\theta,\partial$) term.
\par\noindent
Since $P_{0-0}({\bm e}_4={\bm e}_4+{\bm 1})=\theta+s-1,$ and the composite $(\theta+s-1)P_{0+0}$ differs from $x^3H_6$ only by additive constant, $P_{0+0}$ has $(x,\theta)$-form as
$$P_{0+0}=x^3P_{-3}+x^2P_{-2}+xP_{-1}+P_0.$$
Thus
$$(\theta+s-1)P_{0+0}=x^3(\theta+2+s)P_{-3}+x^2(\theta+1+s)P_{-2}+x(\theta+s)P_{-1}+(\theta+s-1)P_0.$$
Note that the constant term of this expression is the S-value
$$Sv_{0+0}=P_{0-0}\circ P_{0+0}=(s-1)P_0(\theta=0).$$
\par\noindent
Since the $(x,\partial)$-form is unique, we have
$$\begin{array}{ll}T_0&=(\theta+2+s)P_{-3},\\
  \theta T_1(\theta-1)&=(\theta+1+s)P_{-2},\\
  \theta(\theta-1)T_2(\theta-2)&=(\theta+s)P_{-1},\\
  \theta(\theta-1)(\theta-2)T_3(\theta-3)&=(\theta+s-1)P_0-(s-1)P_0(0).\end{array}$$
Since  $T_3=-(\theta+3-e_1)(\theta+3-e_2)(\theta+3-e_3),$ 
$$-\theta(\theta-1)(\theta-2)(\theta-e_1)(\theta-e_2)(\theta-e_3)=(\theta+s-1)P_0-(s-1)P_0(0),$$
and putting $\theta=1-s$, we have the S-value $Sv_{0+0}=P_{0-0}\circ P_{0+0}$:
$$(s-1)P_0(0)=(1-s)(-s)(-1-s)(1-s-e_1)(1-s-e_2)(1-s-e_3)$$
and $P_{0+0}=x^3P_{-3}+x^2P_{-2}+xP_{-1}+P_0,$ where
$$\begin{array}{ll} 
P_{-3}&=(\theta+s+1)(\theta+s)B_0(\theta),\\
P_{-2}&=\theta(\theta+s+1)B_1(\theta-1),\\
P_{-1}&=\theta(\theta-1)B_2(\theta-2),\\
P_0&=\displaystyle{-\frac{\theta(\theta-1)(\theta-2)(\theta-e_1)(\theta-e_2)(\theta-e_3)+Sv_{0-0}}{(\theta+s-1)}}.
\end{array}$$
\comment{
Thus we got
\begin{equation}\label{inverse_and_S-value}
  P_{0-0}({\bm e}_4={\bm e}_4+{\bm 1})\circ P_{0+0}=x^3H_6+Sv_{0+0},\end{equation}
where $P_{0-0}({\bm e}_4={\bm e}_4+{\bm 1})=\theta+s-1$.
}
\newpage
\subsubsection{$P_{+00}$ and the S-value $Sv_{+00}=P_{-00}\circ P_{+00}$ for $H_6$}\label{P_p00}
Perform the coordinate change $x\to1-x$ to \eqref{inverse_and_S-value}:
\begin{itemize}
  \item $P_{0-0}({\bm e}_4={\bm e}_4+{\bm 1})=x\partial+s-1$ changes into
    $$(x-1)\partial +s-1=P_{-00}({\bm e}_1={\bm e}_1+{\bm 1}).$$
  \item
    $x^3H_6({\bm e}_1,{\bm e}_4,{\bm e}_7,T_{10})$ changes into (\S \ref{cochH6})
    $$-(x-1)^3H_6({\bm e}_4,{\bm e}_1,{\bm e}_7,-T_{10}+\alpha(e)),$$
   where 
   $$\alpha(e)=3s^2+(s_{11}+s_{12}-s_{23}+2)s+3s_{11}+3s_{12}-3s_{23}-3s_{33}-21.$$
    \end{itemize}
   Perform next the parameter change ${\bm e}_1\leftrightarrow{\bm e}_4$ and the accessory parameter change $T_{10}\to -T_{10}+\alpha(e)$, to get
      $$ P_{-00}({\bm e}_1={\bm e}_1+{\bm 1})\circ P_{+00}=-(x-1)^3H_6+Sv_{+00},$$
   where $P_{+00}$ is $P_{0+0}$ with the substitution
   $$x\to1-x,\quad \theta\to (x-1)\partial,\quad {\bm e}_1\to {\bm e}_4,\quad{\bm e}_4\to {\bm e}_1,\quad T_{10}\to -T_{10}+\alpha(e),$$
   and
   $$Sv_{+00}=(1-s)(-s)(-1-s)(1-s-e_4)(1-s-e_5)(1-s-e_6).$$
\subsubsection{S-values and reducibility conditions}
We list the S-values for the three simple shifts above:
\begin{prp}\label{SvalueE6}
  The three S-values of the simple shift operators above are given as
\[\def\arraystretch{1.3}\setlength\arraycolsep{3pt} \begin{array}{rcl}
Sv_{--+} &=& P_{++-}({\bm e}_1-{\bf 1}, {\bm e}_4-1, {\bm e}_7+1)\circ P_{--+}
       =-s(s+1)(s+2)e_7e_8e_9,\\
Sv_{-00}&=& P_{+00}({\bm e}_1-{\bf 1})\circ P_{-00}
      =-s(s + 1)(s + 2)(s + e_4)(s + e_5)(s + e_6),\\
      %r(r-1)(r-2)(-r+e_4)(-r+e_5)(-r+e_6),  \\     
%
Sv_{0-0} &=& P_{0+0}({\bm e}_4-{\bf 1})\circ P_{0-0}
      =s(s + 1)(s + 2)(s + e_1)(s + e_2)(s + e_3).\\
\end{array}\]
\end{prp}
Note the order of composition of two maps. The S-value changes
following the rule described in Proposition \ref{twoSvalues}. 
\par\smallskip
Theorem \ref{shopH6} %, Propositions \ref{H_1PQH_0H_0red} and \ref{H_1PQH_0H_1red} 
leads to
\begin{cor}\label{redcondE6}If one of
$$s,\quad e_i+s\ (i=1,\dots,6),\quad e_7,\ e_8, e_9$$
is an integer, then the equation $H_6$ is reducible.
\end{cor}
This can be obtained directly from the Scott theorem (e.g. \cite{Scott,Osh}) since
$${\rm rank}(T_0-{\rm id})+{\rm rank}(T_1-{\rm id})+{\rm rank}(T_\infty-{\rm id})<2\,{\rm rank}(T_0)$$
where $T_x$ denotes the local monodromy around $x\in\{0,1,\infty\}$. 

\subsection{Reducible cases of $H_6$}
\begin{dfn}
  Two operators $H$ and $H'$ with accessory parameters are said to be {\it essentially the same} if $H$ is transformed into $H'$ by
  \begin{enumerate}
  \item changing coordinates by a permutation of $\{x=0,1,\infty\}$,
  \item multiplying a function from the left,
  \item multiplying a factor  $x^*(x-1)^{**}$ from the right,
  \item renaming the local exponents,% $e_1,e_2,\dots$,
    \item and by changing the accessory parameters.
  \end{enumerate}%the parameters $a_0,a_1\dots$.
  Let $G$ be an equation such that its accessory parameters are assigned as functions of local exponents. Two operators $G$ and $G'$ are said to be {\it essentially the same} if $G$ is transformed into $G'$ by the changes $1,\dots,4$ above.% \blue{and 5: the accessory parameters, functions of $e$, change according to the renaming of $e$.}
\end{dfn}
All the statements in this section about $H_6, H_5$ and $H_3$ are valid word to word about $G_6, G_5$ and $G_3$, which will be defined in the next section.

\subsubsection{Factorization when $e_9=0,1$ and when $s=-2,-1,0,1$}\label{factorH6}
\blue{We examine the cases where $e_9=0,1$ and the cases  $s=-2,-1,0,1$.} Recall the $(\theta,\partial)$-form of $H_6$: $T_0+T_1\partial+T_2\partial^2+T_3\partial^3$
 in Proposition \ref{eqwithR6},
\begin{equation}\begin{array}{llll}
   x\partial&=\theta,& \partial x&=\theta+1,\\
   x^2\partial^2&=\theta(\theta-1),& \partial^2x^2&=(\theta+1)(\theta+2),\\
   x^3\partial^3&=\theta(\theta-1)(\theta-2),\quad& \partial^3x^3&=(\theta+1)(\theta+2)(\theta+3),\end{array}\end{equation}
 and
 $$\theta\partial=\partial(\theta-1),\quad \theta\partial^2=\partial^2(\theta-2),\quad\theta\partial^3=\partial^3(\theta-3),\dots$$
 \begin{itemize}
   
 \item When $e_9=0$,\par\noindent
 Since $T_0$ is divisible by $\partial$ from the right, $H_6$ factorizes as
 $$H_6(e_9=0)=H_5\circ \partial,$$ where $H_5=H_6(e_9=0)/\partial,$ which we have explained in \S \ref{G6G5}.
 
\item When $e_9=1$,\par\noindent
Since $\theta+e_9=\theta+1=\partial x$ and $\theta\partial=\partial(\theta-1)$,
$T_0$ is divisible by $\partial$ from the left.   $$ \begin{array}{lcl}
 T_0(e_9=1)&=&\partial (\theta+s+1)(\theta+s)(\theta+s-1)(\theta +e_7-1)(\theta +e_8-1),\\
 T_1(e_9=1)\partial &=&\partial (\theta+s+1)(\theta+s)B_1(\theta -1),\\
 T_2(e_9=1)\partial ^2&=&\partial (\theta+s+1)B_2(\theta -1)\partial ,\\
 T_3(e_9=1)\partial ^3&=&-\partial (\theta+2-e_1)(\theta+2-e_2)(\theta+2-e_3)\partial ^2,
 \end{array}$$
leads to $$H_6(e_9=1)=\partial\circ X_5,$$ where $X_5$ is essentially equal to $H_5$.
  
  \item When $s=1$, the coefficients of $H_6$ change as\par\noindent
  $$\begin{array}{ll}
    T_0(s=1)&=(\theta +3)(\theta +2)(\theta +1)B_0(\theta ,s=1)=\partial^3x^3B_0(\theta ,s=1),\\
    T_1(s=1)\partial &=(\theta +3)(\theta +2)B_1(\theta ,s=1)\partial =\partial(\theta +2)(\theta +1)B_1(\theta -1,s=1)\\
    &=\partial ^3x^2B_1(\theta -1,s=1),\\
    T_2(s=1)\partial ^2&=(\theta +3)B_2(\theta ,s=1)\partial ^2=\partial ^2(\theta +1)B_2(\theta -2,s=1)\\
    &=\partial ^3xB_2(\theta -2,s=1),\\
    T_3(s=1)\partial ^3&=\partial ^3B_3(\theta -3,s=1),\end{array}$$
  which lead to
  $$H_6(s=1)=\partial ^3\circ H_3,$$
  as we have stated in \S \ref{fromG6toG3}.
  
  \item When $s=0$,\par\noindent
  $$\begin{array}{ll}
    T_0(s=0)&=(\theta +2)(\theta +1)\theta B_0(\theta ,s=0)=\partial^2x^2B_0(\theta ,s=0)x\partial ,\\
    T_1(s=0)\partial &=(\theta +2)(\theta +1)B_1(\theta ,s=0)\partial =\partial^2x^2B_1(\theta ,s=0)\partial ,\\
    T_2(s=0)\partial ^2&=(\theta +2)B_2(\theta ,s=0)\partial ^2=(\theta +2)\partial B_2(\theta -1,s=0)\partial \\
    &=\partial (\theta +1)B_2(\theta -1,s=0)\partial =\partial ^2xB_2(\theta -1,s=0)\partial ,\\
    T_3(s=0)\partial ^3&=\partial ^2B_3(\theta -2,s=0)\partial \end{array}$$
  leads to
  $$H_6(s=0)=\partial ^2\circ X_3\circ \partial ,$$
  where $X_3$ is essentially equal to $H_3$.

 \item When $s=-1$,\par\noindent
  $$\begin{array}{ll}
   T_0(s=-1)&=(\theta +1)\theta (\theta -1)B_0(\theta ,s=-1)=\partial x\cdot x^2\partial ^2B_0(\theta ,s=-1)\\
   &=\partial  x^3B_0(\theta +2,s=-1)\partial ^2,\\
   T_1(s=-1)\partial &=(\theta +1)\theta B_1(\theta ,s=-1)\partial =\partial x x\partial  B_1(\theta ,s=-1)\partial \\
   &=\partial  x^2B_1(\theta +1,s=-1)\partial ^2,\\
    T_2(s=-1)\partial ^2&=(\theta +1)B_2(\theta ,s=-1)\partial ^2=\partial  x B_2(\theta ,s=-1)\partial ^2,\\
    T_3(s=-1)\partial ^3&=\partial  B_3(\theta -1,s=-1)\partial ^2\end{array}$$
  lead to
  $$H_6(s=-1)=\partial \circ X_3'\circ \partial ^2,$$
  where $X_3'$ is essentially equal to $H_3$.

\item When $s=-2$,\par\noindent
  $$\begin{array}{ll}
   T_0(s=-2)&=\theta (\theta -1)(\theta -2)B_0(\theta ,s=-2)=x^3\partial ^3B_0(\theta ,s=-2)\\
   &=x^3B_0(\theta +3,s=-2)\partial ^3,\\
   T_1(s=-2)\partial &=\theta (\theta -1)B_1(\theta ,s=-2)\partial =x^2\partial ^2 B_1(\theta ,s=-2)\partial \\
   &=x^2B_1(\theta +2,s=-2)\partial ^3,\\
    T_2(s=-2)\partial ^2&=\theta  B_2(\theta ,s=-2)\partial ^2=x B_2(\theta+1,s=-2)\partial ^3,\\
    T_3(s=-2)\partial ^3&=T_3(s=-2)\partial ^3\end{array}$$
  lead to
  $$H_6(s=-2)=X_3''\circ \partial ^3,$$
  where $X_3''$ is essentially equal to $H_3$.
\end{itemize}

\subsubsection{Factorization when $e_9\in{\mathbb Z}$, $e_1+s\in{\mathbb Z}$ and $s\in{\mathbb Z}$}
The factorizations obtained in \S 6.2.1 and Proposition \ref{FactorType} lead to
\begin{prp}\label{E6red_e9}
  If $e_9\in {\mathbb Z}$, then $H_6$ factorizes as follows: when $e_9$ is
  a non-positive integer, the type of factorization is $[51]$ and, when
  it is a positive integer, $[15]:$
   $$\begin{array}{ccccccccc}
e_9=\ & \cdots &-2    &-1   & 0   & 1     & 2    &3      &\cdots\\   
& &[51]&[51]&[51]A0&[15]A0&[15]&[15]&    
  \end{array}$$
  The notation $A0$ means that the \blue{factors} have no singularity other than   $\{0, 1, \infty\}$.
\end{prp}
When $e_9=-1$, the factors have one apparent singular point and when $e_9=-2$,
two apparent singular points (cf. Proposition \ref{apparentsing}).
%and \S \ref{OQ}).

By the change $x\to1/x$, the condition $e_9\in\mathbb{Z}$ is converted to $e_1+s\in\mathbb{Z}$:
\begin{prp}If $e_1+s\in\mathbb{Z}$, $H_6$ factorizes as follows:
 $$\begin{array}{ccccccccc}
e_1+s=\ & \cdots &-2    &-1   & 0   & 1     & 2    &3      &\cdots\\   
& &[51]&[51]&[51]A0&[15]A0&[15]&[15]&    
  \end{array}$$
  When $e_1+s=0,1$, the factor $[5]$ is essentially equal to $H_5$.
  \end{prp}

\begin{prp}\label{prp1113}
  If $s\in\mathbb{Z}$, $H_6$ is reducible of type $\{3111\}$:
  $$\begin{array}{ccccccccc}
      s=\ & \cdots &-3 &-2  & -1  & 0      & 1      &2       &\cdots\\   
          &        &[3111]&[3111]A0&[1311]A0&[1131]A0&[1113]A0&[1113]&
  \end{array}$$
  %When $s=-2,-1,0,1$, the factor $[3]$ is essentially $G_3$,
  %and no apparent singularities appear. %$G_3$ is defined in \S \ref{G6G3}.
\end{prp}
%\blue{These exhaust all the reducible cases.}
\subsubsection{\blue{Polynomial solutions}}\label{polynomSolH6}
We apply Proposition \ref{polynomSol} to 
  $$H_6=(\theta+s)(\theta+s+1)(\theta+s+2)(\theta+e_7)(\theta+e_8)(\theta+e_9)+(T_1+T_2\partial+T_3\partial^2)\partial.$$

\begin{prp}\label{polynomSol1}If one of $e_j$ $(j=7,8,9)$ and $s$ is a non-positive integer $-m$, then $H_6$ has a polynomial solution of degree $\le m$.
\end{prp}

Moreover, since the symmetry $x\to1/x$ takes ${\bm e}_7\to {\bm e}_1+{ s}$ (see \S 4.2.3), we have
\begin{prp}\label{polynomSol2}If $e_i+s$ $(i=1,2,3)$ is 0 or a negative integer $-m$, then $H_6$ has a solution: a power of $x$ times a polynomial of degree $\le m$. 
\end{prp}

\newpage
  %%%%% sasaki 2023/02/09 equationG6
\section{Equation $G_6$}\secttoc
In this section, we define the equation $G_6$ with Riemann scheme
$R_6$ by replacing the coefficient $T_{10}$ of the equation $H_6$ by a polynomial in the local exponents $e$. The equation $G_6$ admits shift operators for any block shifts of $e$. 
%given in the previous proposition.
%For this purpose, referring to the set of fundamental shift operators 
%for the equation $E_2$ and will try to extend those operators 
%to the present equation $H_6$ to the effect that the parameter $T_{10}$
%is determined as a polynomial of exponents.
\medskip

We prepare an algebraic lemma for later use.

\begin{lemma}The ring of symmetric polynomials in $x_1,\dots,x_n$ invariant under the shift $sh:(x_1,\dots,x_n)\to(x_1+1,\dots,x_n+1)$ is generated by $1$ and the fundamental symmetric polynomials $t_i$ of degree $i$ $(i=2,\dots,n)$ in
  $$y_k := x_k - y_0\ (k=1,2,\dots,n),$$
where $y_0 := (x_1+x_2+\cdots+x_n)/n.$   $\{t_2,\dots,t_n\}$ are algebraically independent.
\end{lemma}
\begin{proof}  $y_1,\dots,y_n$ are stable by the shift $sh$, and $y_0$ changes to $y_0+1$. On the other hand, permutations of $x_1,\dots,x_n$ correspond those of $y_1,\dots,y_n$; $y_0$ does not change.
\end{proof}

We apply this lemma to the ring of polynomials of the variables as $x_1=e_1,x_2=e_2,x_3=e_3$ when $n=3$:% stated as follows.

\begin{cor}\label{t2t3}The ring of polynomials invariant under the shift $(e_1,e_2,e_3)\to(e_1+1,e_2+1,e_3+1)$ is generated by $t_2$ and $t_3$, where
\[
\begin{array}{ll}
    t_2&=(e_1-e_0)(e_2-e_0)+(e_2-e_0)(e_3-e_0)+(e_3-e_0)(e_1-e_0),\\
    &=(-e_1^2 + e_1e_2 + e_1e_3 - e_2^2 + e_2e_3 - e_3^2)/3\\
    &=s_2-s_1^2/3,\\[2mm]
    t_3&=(e_1 - e_0)(e_2 - e_0)(e_3 - e_0)\\
    &=(2e_1 - e_2 - e_3)(2e_2 - e_1 - e_3)(2e_3 - e_1 - e_2)/27\\
    &=2s_1^3/27 - s_1s_2/3 + s_3,    \\
    e_0&=(e_1+e_2+e_3)/3, \\
    s_1&=e_1+e_2+e_3, \quad s_2=e_1e_2+e_1e_3+e_2e_3, \quad s_3=e_1e_2e_3.
  \end{array}
\]
\end{cor}

\subsection{\blue{Definition of the equation $G_6(e,a)$}}\label{theoremG5}
For an equation $G(e)$ with local exponents $e$, we denote
by $G({\bm e}_1\to{\bm e}_1-{\bf 1})$ the equation with exponents
${\bm e}_1$ shifted to ${\bm e}_1-{\bf 1}$ and so on.
Now we can state the main theorem of this paper.

\begin{thm}\label{shiftopG6} Let $G_6$ denote an equation $H_6$ 
with the Riemann scheme $R_6$ and with the accessory parameter $T_{10}$ 
 replaced by a polynomial 
in $e_1,\dots,e_9$. We assume that it admits shift operators relative 
to the shifts of blocks ${\bm e}_i\to{\bm e}_i\pm{\bf 1}\ (i=1, 4, 7)$. 
Namely, for $i=1$, assume that the equation
  \[G_6({\bm e}_1\to {\bm e}_1+{\bf 1})\circ P= Q\circ G_6\]
admits a non-zero solution $(P,Q)$ and similarly for other cases. Then
the term $T_{10}$ is written as
\[T_{10}=S_{10}+R,\]
  where
\[
\begin{array}{l}
  S_{10}:=(-5-s_{21}+s_{22}-5s_{23}+s_{31}-s_{32}-3s_{33})/2 \\[2mm]
\quad\qquad +(s_{11}-7s_{13}+s_{11}s_{13}+s_{11}s_{23}-s_{13}s_{21}+s_{13}s_{22})/3\\[2mm]
\quad\qquad  +(s_{11}^2-s_{12}^2+s_{13}^2-s_{11}s_{21}+s_{12}s_{22}+s_{13}s_{23})/6\\[2mm]
\quad\qquad  +(s_{11}^2-s_{12}^2)s_{13}/9+(s_{11}^3-s_{12}^3)/27,
  \end{array}
\]
and $R$ is any element of the $\mathbb{C}$-algebra generated by
\[t_{2i}:=s_{2i}-s_{1i}^2/3 {\rm\quad and\quad } 
t_{3i}:=2s_{1i}^3/27 - s_{1i}s_{2i}/3 + s_{3i},\quad i=1,2,3.
\]
\end{thm}

%We could not decide whether there exist shift operators for other shifts, such as $e_2\to e_2+3$.

\begin{cor}\label{t2t3bis}
When $T_{10}$ is a polynomial in $e_1,\dots,e_9$ of degree 3, then
\[T_{10}=S_{10}+R,\quad R=a_0+a_1t_{21}+a_2t_{22}+a_3t_{23}+a_4t_{31}+a_5t_{32}+a_6t_{33},\]
where $a_0,\dots,a_6$ are free constants.
\end{cor}

\begin{dfn}The operator $H_6$ with the cubic polynomial $T_{10}$ as above 
in the corollary will be denoted as $G_6(e,a)$. \end{dfn}

\subsection{Proof of Theorem \ref{shiftopG6}}% \ref{shiftopG6}}
Thanks to Theorem \ref{shopH6}, we have only to solve the system for $T_{10}(e)$:

\[
\begin{array}{ll}
  T_{10}(sh_1)-T_{10}&=s_{13}+s_{23}+1,\\[2mm]
T_{10}(sh_2)-T_{10}&=0,\\[2mm]
T_{10}(sh_3)-T_{10}&=20-s_{11}^2/3-2s_{11}s_{13}/3+s_{12}^2/3-s_{13}^2/3\\
&-2s_{11}+7s_{13}+s_{21}-s_{22}+2s_{23}.
\end{array}
\]
One can check that the polynomial $S_{10}$ solves these system 
of three identities. The second identity, for example, 
says that $T_{10}$ is a polynomial of $t_{22}$ and $t_{32}$ 
with coefficients independent of $\{e_4,e_5,e_6\}$. 
Now, the difference $R=T_{10}-S_{10}$ is a polynomial invariant 
under $sh_1$, $sh_2$ and $sh_3$; therefore, we have the theorem
in view of Corollary \ref{t2t3}.

\subsection{Inverse shift operators and S-values of $G_6$}
The shift operators
$$\begin{array}{ll}P_{+00}&=x^3(x-1)^5\partial^5+\cdots,\\
P_{0+0}&=x^5(x-1)^3\partial^5+\cdots,\\
P_{++-}&=x^3(x-1)^3\partial^5+\cdots\end{array}$$
for the equation $G(e,a)$ depends linearly on the parameters $a_0,\dots, a_6$ as follows:\footnote{they are listed in G6PQ.txt in FDEdata mentioned in the end of Introduction.}

\[\begin{array}{ll}
   P_{+00}&=\overline{P}_{+00}+R(x-1)^3\bigl(x\partial^2(s+1)\partial\bigr),\\[1mm]
   P_{0+0}&=\overline{P}_{0+0}+Rx^3\bigl((x-1)\partial^2(s+1)\partial\bigr),\\[1mm]
   P_{++-}&=(H_6-p_0)/\partial\\[1mm]
   &=\overline{P}_{++-}+R\bigl(x(x-1)\partial^2(s+1)(2x-1)\partial+s(s+1)\bigr),\end{array}\]
where $R=a_0+t_{21}a_1+\cdots+t_{33}a_6$,
and $\overline{P}_{+00}$, $\overline{P}_{0+0}$ and $\overline{P}_{++-}$ 
are operators excluding the terms with $a_0, \dots, a_6$. 

The S-values do not depend on the parameter $a$'s, and are exactly the same to those for $H_6$ given in Proposition \ref{SvalueE6}.
\subsection{Adjoint and the coordinate changes $x\to1-x$ and $x\to1/x$}
The operator $G(e,a)$ is symmetric under adjoint and the coordinate changes interchanging $\{0,1,\infty\}$:
\begin{thm}\label{symmetriesG6}\noindent
  \begin{itemize}[leftmargin=*] \setlength{\itemsep}{2pt}
\item Adjoint symmetry: The adjoint of $G_6(e,a)$ is equal to
   $$G_6({\bm 2}-{\bm e}_1,{\bm 2}-{\bm e}_4,{\bm 1}-{\bm e}_7,%r=1-r,
  -a_0,-a_1,-a_2,-a_3,a_4,a_5,a_6).$$
\item $(x\to1-x)$-symmetry: $$G_6({\bm e},a)|_{x\to1-x}=G_6({\bm e}_4,{\bm e}_1,{\bm e}_7,-a_0,-a_2,-a_1,-a_3,-a_5,-a_4,-a_6),$$

\item $(x\to1/x)$-symmetry: $$x^{r-3}G_6({\bm e},a)|_{x\to1/x}\circ x^{-r}=G_6({\bm e}_7-s{\bm1},{\bm e}_4,{\bm e}_1+s{\bm1},-a_0,-a_3,-a_2,-a_1,-a_6,-a_5,-a_4),$$
%  where $r$ is a new parameter defined as $$ r=-s=(e_1+\cdots+e_9-6)/2,$$

%\item Differentiation symmetry: $$\partial E_6({\bm e})=E_6({\bm e}_1-{\bm 1},{\bm e}_4-{\bm 1},{\bm e}_7+{\bm 1})\partial,$$
  where $G_6|_{x\to1-x}$ and $G_6|_{x\to1/x}$ are $G_6$ after the coordinate changes $x\to1-x$ and $x\to1/x$, respectively.
\end{itemize}
 \end{thm}
%If $G_6(e,a)$ enjoys all the symmetries, adjoint and the coordinate changes $x\to1-x$ and $x\to1/x$, then $a_0=\cdots=a_6=0$.

%\subsubsection{Proof of Theorem \ref{symmetriesG6}}
When $T_{10}=S_{10}$, that is, $a_0=\cdots=a_6=0$, a straightforward computation (use $(\theta,\partial)$-form for the adjoint and the coordinate change $x\to1/x$, and $(x,\partial)$-form for $x\to1-x$) leads to the result. 
\comment{In general we have only to notice that for ${\bm e}_{{\rm adj}}=({\bm 2}-{\bm e}_1,{\bm 2}-{\bm e}_4,{\bm 1}-{\bm e}_7),$
$$t_{2j}({\bm e}_{{\rm adj}})=t_{2j}({\bm e}),\ j=1,2,3,\qquad t_{3j}({\bm e}_{{\rm adj}})=-t_{3j}({\bm e}),\ j=4,5,6,$$
for ${\bm e}_{ch01}=({\bm e}_4,{\bm e}_1,{\bm e}_7$),
$$t_{i1}({\bm e}_{ch01})=t_{i2}({\bm e}),\ t_{i2}({\bm e}_{ch01})=t_{i1}({\bm e}),\ t_{i3}({\bm e}_{ch01})=t_{i3}({\bm e}),\ \ i=2,3,$$
for ${\bm e}_{ch0\infty}=({\bm e}_7-s{\bm1},{\bm e}_4,{\bm e}_1+s{\bm1})$,
$$t_{i1}({\bm e}_{ch0\infty})=t_{i3}({\bm e}),\ t_{i2}({\bm e}_{ch0\infty})=t_{i2}({\bm e}),\ t_{i3}({\bm e}_{ch0\infty})=t_{i1}({\bm e}),\ \ i=2,3.$$}

\section{Equation $E_6:=G_6(e,0)$}\secttoc
\begin{dfn}When $a_0=\cdots=a_6=0$, $G_6(e,a)$ is called $E_6(e)$.\end{dfn}
The equation $E_6(e)$ is very symmetric:% in the sense that the following properties hold.
\begin{thm}\label{symmetriesE6}\noindent
  \begin{itemize}[leftmargin=*] \setlength{\itemsep}{2pt}
  \item Shift relations:
  $$  E_6({\bm e}_1\pm{\bm 1},{\bm e}_4,{\bm e}_7)\circ P_{\pm00}=Q_{\pm00}\circ E_6({\bm e}),\qquad  E_6({\bm e}_1,{\bm e}_4\pm{\bm 1},{\bm e}_7)\circ P_{0\pm0}=Q_{0\pm0}\circ E_6({\bm e}),$$
$$  E_6({\bm e}_1\pm{\bm 1},{\bm e}_4\pm{\bm 1},{\bm e}_7\mp{\bm 1})\circ P_{\pm\pm\mp}=Q_{\pm\pm\mp}\circ E_6({\bm e}).$$
\item Differentiation symmetry: $$\partial E_6({\bm e})=E_6({\bm e}_1-{\bm 1},{\bm e}_4-{\bm 1},{\bm e}_7+{\bm 1})\partial,$$
\item Adjoint symmetry: The adjoint of $E_6(e)$ is equal to
   $$E_6({\bm 2}-{\bm e}_1,{\bm 2}-{\bm e}_4,{\bm 1}-{\bm e}_7).$$
\item $(x\to1-x)$-symmetry: $$E_6({\bm e})|_{x\to1-x}=E_6({\bm e}_4,{\bm e}_1,{\bm e}_7),$$
\item $(x\to1/x)$-symmetry: $$x^{-s-3}E_6({\bm e})|_{x\to1/x}\circ x^{s}=E_6({\bm e}_7-s{\bm1},{\bm e}_4,{\bm e}_1+s{\bm1}),$$
  where $E_6|_{x\to1-x}$ and $H_6|_{x\to1/x}$ are $H_6$ after the coordinate changes $x\to1-x$ and $x\to1/x$, respectively.
\end{itemize}
\end{thm}
Since we have adjoint symmetry as in the theorem, Proposition \ref{RelationPQ} is applicable to know the second members of shift operators $(P,Q)$.  %This equation $E_6$ has several nice properties to follow.

\subsection{Interpolative expression of $E_6$ using $V$}\label{E6RedInter}
Let $V:=\partial^3\backslash E_6(e_9=3-e_1-\cdots-e_8)$, that is, $E_6(e_9=3-e_1-\cdots-e_8)=\partial^3\circ V$, as in \S \ref{fromG6toG3}. Put
$$V_1=V,\ V_0=V(e'),\ V_{-1}=V_0(e'),\ V_{-2}=V_{-1}(e'),\  $$
where $e'=(e_1-1,\dots,e_6-1,e_7+1,e_8+1),$ and 
$$U:=\frac{(s-1)s(s+1)(s+2)}6\left\{
\frac{\partial^3 \circ V_1}{s-1}
-3\frac{\partial^2 \circ V_0 \circ \partial}{s}
+3\frac{\partial \circ V_{-1} \circ \partial^2}{s+1}
-\frac{V_{-2} \circ \partial^3}{s+2}
\right\},$$
where $s=2-(e_1+\cdots+e_8+e_9)/3.$ 
%$$T':=E_6(e_1-s-2,\dots,e_6-s-2,e_7+s+2,e_8+s+2,12-e_1-\dots-e_8+s+2),$$
%a slight modification of $E_6$.
Then, by a straightforward computation, we have an interpolative expression of $E_6$ by use of $V$:

  \begin{prp}
    $$E_6-U=-3(s-1)s(s+1)(s+2)\left\{\left(x^2 - x + \frac13\right)\partial^2 + \left(x - \frac12\right)(e_7 + e_8 + 1)\partial + e_7e_8\right\}.$$
    \end{prp}
This expression makes the decomposition of $E_6$
described in Proposition \ref{prp1113} clear.

 \subsection{Explicit expression of the decomposition [1113] when $s=2,3,\dots$}\label{s23}
 By Proposition \ref{prp1113}, when $s=1,2,3,\dots$,
 the equation $H_6$ is reducible of type  $[1113]$.
 In this section, for $E_6$, we find explicit expression of the factors of decomposition [1113],  when $s=2$, $3$, $\dots$.
Recall (\S \ref{fromG6toG3byFactor}) $E_6(s=1)=\partial^3\circ V$, where
$$V=x^3B_0(\theta)+x^2B_1(\theta+1)+\cdots,\quad B_0(\theta)=(\theta+e_7)(\theta+e_8)(\theta+e_9).$$
 
 Assume $e_7$, $e_8$, $e_9\notin\mathbb{Z}$,
 that is, $B_0(\theta=k)\not=0$ $(k\in\mathbb{Z})$.
 Recall the shift relation $E_6(e-u)\circ \partial=\partial\circ E_6(e)$,
 in particular
$$E_6(s=n+1)\circ\partial=\partial\circ E_6(s=n),$$
and set 
$$E^{(n)}=E_6(s=n+1),\quad n=0,1,\dots$$
They satisfy
$$E^{(0)}:=\partial^3\circ V,\quad E^{(n)}\circ \partial^n
=\partial^n\circ E^{(0)},\quad {\it i.e.}, \quad
E^{(n)}:=(\partial^n\circ E^{(0)})/ \partial^{n}.$$

\begin{lemma}\label{lemma:E(1)}
$E^{(n)}(1)$ is a non-zero constant.
\end{lemma}
\begin{proof} The identity
  $$\begin{array}{ll}
    E^{(n)}(1)&=\ E^{(n)} \partial^n (\frac1{n!} x^n) 
= \partial^n E^{(0)} (\frac1{n!} x^n) 
= \partial^n \partial^3 V(\frac1{n!} x^n)\\[2mm]
&=\ \partial^{n+3}(\frac1{n!}B_0(\theta=n)x^{n+3}+\cdots)
=((n+3)!/n!)B_0(\theta=n) \end{array}$$
 asserts the claim.
\end{proof} %\hfill$\square$ \bigbreak 

\begin{lemma}\label{LetQ1Q2}
  Let $Q_1$, $Q_2$ be non-zero differential operators
  with rational function coefficients.
Assume $f:=Q_2(1)$ is a non-zero rational function, and $Q_1 Q_2(1)$ is a non-zero constant.
Then there exist differential operators $\tilde{Q}_1, \tilde{Q}_2$ such that
\begin{align*} 
\tilde{Q}_1 \circ \partial &=\partial \circ Q_1 \circ f, \\
\tilde{Q}_2 \circ \partial &=\partial \circ \frac1{f} \circ Q_2, \\
\tilde{Q}_1 \circ \tilde{Q}_2 \circ \partial &=\partial \circ Q_1 \circ Q_2.
\end{align*}
\end{lemma}
\begin{proof}
Since $\partial(Q_1 (f)) = \partial(Q_1 Q_2(1))=0$
and $\partial(\frac1{f} Q_2(1)) = \partial(1)=0$, 
the right-hand sides of the above two first formulae are divisible from the right by $\partial$.
The last equation is obtained by the combination of first two.
\end{proof} %\hfill$\square$ \bigbreak 

We start by putting
$$Q_1^{(0)}:=\partial^3,\quad Q_2^{(0)}:=V=x^3(x-1)^3\partial^3+\cdots;$$
they satisfy $E^{(0)}=Q_1^{(0)}\circ Q_2^{(0)}$.
Apply Lemma \ref{LetQ1Q2} to
$$f=f^{(n)}:=Q_2^{(n)}(1),\quad Q_1=Q_1^{(n)},\quad Q_2=Q_2^{(n)},\quad Q_1\circ Q_2=E^{(n)}$$
to define  $Q_1^{(n+1)}$ and $Q_2^{(n+1)}$ inductively:
\begin{equation}\label{Q1Q2}
  \begin{aligned} 
 Q_1^{(n+1)}\circ \partial &=\partial\circ Q_1^{(n)}\circ f^{(n)},\\
 Q_2^{(n+1)}\circ \partial &=\partial\circ\frac1{f^{(n)}}\circ Q_2^{(n)},\\
 Q_1^{(n+1)}\circ Q_2^{(n+1)}\circ\partial&=\partial\circ Q_1^{(n)}\circ Q_2^{(n)}.
  \end{aligned}
\end{equation}%\end{align*}
Note that $Q_1^{(n)}\circ Q_2^{(n)}= E^{(n)}$, $Q_1^{(n+1)}\circ Q_2^{(n+1)}= E^{(n+1)}$,
and that $f^{(n)}$ is a non-zero rational function by Lemma~\ref{lemma:E(1)}.
Note also
$$Q_1^{(1)}=f^{(0)}\partial^3+\cdots,\quad \cdots,
\quad Q_1^{(n)}=f^{(0)}\cdots f^{(n-1)}\partial^3+\cdots,$$
$$Q_2^{(1)}=\frac{x^3(x-1)^3}{f^{(0)}}\partial^3+\cdots,\quad \cdots,
\quad Q_2^{(n)}=\frac{x^3(x-1)^3}{f^{(0)}\cdots f^{(n-1)}}\partial^3+\cdots.$$

We define the differential operator $P^{(n)}$ of order $n$ inductively by
$$P^{(n)}:=\partial\circ\frac1{f^{(n-1)}}P^{(n-1)}=\partial\circ\frac1{f^{(n-1)}}\circ \partial\circ\frac1{f^{(n-2)}}\circ\cdots\circ \partial\circ\frac1{f^{(1)}}\circ \partial\circ\frac1{f^{(0)}}.$$
Then by definition,  we have the following lemma:

\begin{lemma}\label{explicit1113} 
\begin{itemize}
\item[{\rm(1)}]
$Q_1^{(n)}\circ P^{(n)}=\partial^{n+3}$.
\item[{\rm(2)}]
${\rm Sol} (P^{(n)})$ is a subspace of $\langle 1,x,\dots,x^{n+2}\rangle$
of dimension $n$.
\item[{\rm(3)}]
The solution space of $Q_1^{(n)}$ is a $3$-dimensional subspace of $\mathbb{C}(x)$.
\end{itemize}
\end{lemma}
\begin{proof}
(1) We use $P^{(n+1)} = \partial \circ \frac{1}{f^{(n)}} \circ P^{(n)}$,
then 
$$\begin{array}{ll}Q_1^{(n+1)}\circ P^{(n+1)} &= Q_1^{(n+1)}\circ\partial
  \circ \frac{1}{f^{(n)}} \circ P^{(n)}\\[2mm]
  & = \partial\circ Q_1^{(n)}\circ f^{(n)}\circ \frac{1}{f^{(n)}}
  \circ P^{(n)} =\partial\circ Q_1^{(n)}\circ P^{(n)}.\end{array}$$
(2) ${\rm Ker\ } \partial^{n+3} = \langle 1,x,\dots,x^{n+2}\rangle.$
\end{proof} % \hfill$\square$ \bigbreak 

We prepare another lemma:
\begin{lemma}\label{LLL}
Let $Q$ be a differential operator over $\mathbb{C}(x)$ of order three 
whose leading term is $\partial^3$,
such that the solution space is a $3$-dimensional vector space in $\mathbb{C}(x)$.
\begin{itemize}
\item[{\rm(1)}]
For linearly independent solutions  $h_1, h_2, h_3 \in \mathbb{C}(x)$, set
\[ \def\arraystretch{1.2}
\begin{array}{ll}
L_3&:=\partial-f_3, \quad f_3={h_3'}/{h_3},\ \quad{\rm put}\quad g_2 := L_3(h_2),\\
L_2 &:= \partial - f_2,\quad f_2={g'_2}/{g_2},\ \quad{\rm put}\quad g_1 := L_2\circ L_3(h_1), \\
L_1 &:= \partial - f_1,\quad f_1={g'_1}/{g_1}.\end{array}
\]
Then we have
\[
Q = L_1\circ L_2\circ L_3.
\]
\item[{\rm(2)}]
Conversely, if $Q$ has an expression $L_1\circ L_2\circ L_3$ such as
$$L_i=\partial - f_i(x),\quad f_i(x)\in\mathbb{C}(x) \quad(i=1,2,3),$$
then 
$$f_3=h_3'/h_3,\quad f_2=g_2'/g_2,\ g_2=L_3(h_2),\quad f_1=g_1'/g_1,\ g_1=L_2\circ L_3(h_1)$$ for some solutions $h_j$ $(i=3,2,1)$.
\end{itemize}

\end{lemma}
\begin{proof}\begin{itemize}
\item[{\rm(1)}] Easy to see that $h_3,h_2$ and $h_1$ solve $L_1\circ L_2\circ L_3$.
\item[{\rm(2)}]  Set
  \[ \def\arraystretch{1.2}
  \begin{array}{ll}W_3 &= \{ u \in \mathbb{C}(x) \mid L_1L_2L_3u=0 \},\\
  W_2 &:= \{ u \in\mathbb{C}(x) \mid L_2 L_3 u = 0 \},\\
  W_1 &:= \{ u \in \mathbb{C}(x) \mid L_3 u = 0\}.
  \end{array}\]
Then $W_1 \subset W_2 \subset W_3$ and $\dim W_i = i$ for $i=1,2,3$.
We take $h_3, h_2, h_1 $ so that
\[
\langle h_3 \rangle = W_1,
\quad
\langle h_2,h_3 \rangle = W_2,
\quad
\langle h_1, h_2,h_3 \rangle = W_3.
\]
\end{itemize}
% Then the differential operators of the assertion in both-hand sides has the same solution space and the same leading coefficient.
\end{proof}

Apply these lemmas to
$$Q=\frac{1}{f^{(0)}\cdots f^{(n-1)}}Q_1^{(n)},$$
and we have the conclusion.

\begin{prp}Define $f^{(n)}$, $Q_1^{(n)}$ and $Q_2^{(n)}$ by $\eqref{Q1Q2}$. Then $E_6(s=n+1)$ $(n=1,2,\dots)$ factors as $Q_1^{(n)}\circ Q_2^{(n)}$. For a basis $\{h_1,h_2,h_3\}$ of the solution space of $Q_1^{(n)}$, define
  the first-order operators $\{L_1,L_2,L_3\}$ as in Lemma $\ref{LLL}$. Then
  $$Q_1^{(n)}=f^{(0)}\cdots f^{(n-1)}L_1\circ L_2\circ L_3.$$
  Though these three operators  $L_1,L_2$ and $L_3$ are not uniquely determined,
  they are controlled by Lemma $\ref{LLL}$.
\end{prp}

\begin{remark}
  The three operators $L_1,L_2$ and $L_3$ have apparent singularities not only at the roots and the poles of $f^{(0)}\cdots f^{(n-1)}$ but also at the points depending on the choice of the basis $\{h_1,h_2,h_3\}$.% We do not know how to control them.
\end{remark}

\section{Shift operators of $H_5$}\label{E5}
%\nostcrule
\secttoc
We find shift operators and reducibility conditions for $H_5$.
Recall 
  \[H_5=H_5(e_1,\dots,e_8):=H_6(e_9=0)/\partial=x\overline{T}_0+\overline{T}_1+\overline{T}_2\partial +\overline{T}_3\partial ^2\]
where
\[\def\arraystretch{1.1} \begin{array}{rcl}
   \overline{T}_0 &=& (\theta-r+1)(\theta-r+2)(\theta-r+3)(\theta+e_7+1)(\theta+e_8+1), \\
   \overline{T}_1 &=& (\theta-r+1)(\theta-r+2)B_{51},\quad B_{51}:=B_1(e_9=0),\\
   \overline{T}_2 &=& (\theta-r+2)B_{52},\quad B_{52}:=B_2(e_9=0), \\
   \overline{T}_3 &=& -(\theta+3-e_1)(\theta+3-e_2)(\theta+3-e_3).
\end{array} \]
Its Riemann scheme is
$$ \left(\begin{array}{ccccc}
  0&1&e_1-1&e_2-1&e_3-1\\
  0&1&e_4-1&e_5-1&e_6-1\\ 
    1-r&2-r&3-r&e_7+1&e_8+1\end{array}\right),\qquad r=-s=(e_1+\cdots+e_8-6)/3.$$
This equation has $(x\to1-x)$-symmetry and adjoint symmetry
    but has no $(x\to1/x)$-symmetry nor differentiation symmetry
    as are summarized in \S 2.1 and \S 2.2.

\subsection{Shift operators of $H_5$, S-values  and reducibility conditions}\label{E5Shift}

\begin{thm}\label{shiftopE5}
  Equation $H_5$ has shift operators relative to the shifts
  of blocks $\{e_1,e_2,e_3\}$ and  $\{e_4,e_5,e_6\}$.
  Explicit form is tabulated in \S \ref{AppenE5}.
\end{thm}
\noindent
\textsc{Notation}: $P_{\pm0}$ denotes the shift operator of $H_5$
for the shift ${{\bm e}_1}\pm{\bm 1}$,
and $P_{0\pm}$ for  ${{\bm e}_4}\pm{\bm 1}$.

\begin{prp}\label{SvalueE5} The S-values for the shifts of blocks:
\[ \def\arraystretch{1.1} \begin{array}{ll}
Sv_{-0}&=P_{+0}({\bm e}_1-1)\circ P_{-0}=(r-1)(r-2)(e_4-r)(e_5-r)(e_6-r),\\
Sv_{0-}&=P_{0+}({\bm e}_4-1)\circ P_{0-}= -(r-1)(r-2)(e_1-r)(e_2-r)(e_3-r).
\end{array}\]
\end{prp}

\begin{thm} \label{redcondE5}
  If one of $r, e_1-r, \dots, e_6-r$ is an integer,
  then the equation $H_5$ is reducible.
  \end{thm}

\noindent
{Proof of Theorem \ref{shiftopE5}}:
Let $sh$ be a shift of blocks ${\bm e}_i\to{\bm e}_i\pm{\bm1}$ $(i=1,4)$,
and $H_{6sh}$  be $H_6$ with shift $sh$. We have the shift relation
  \[ H_{6sh}\circ P = Q\circ H_6.\]
Let us see what happens if we put $e_9=0$ in this relation. We have
$$H_6(e_9=0)=H_5\circ \partial\quad{\rm and}
\quad H_{6sh}(e_9=0)=H_{5sh}\circ \partial,
$$
hence
$$H_{5sh}\circ \partial\circ P = Q\circ H_5\circ \partial.$$
Define $P_1$ by 
\[\partial\circ P = P_1\circ \partial,\]
then we get
\[H_{5sh}\circ P_1 = Q\circ H_5.\]
Divide $P_1$ by $H_5$ on the right:
\[P_1=A\circ H_5+P_2,\quad {\rm deg\ }(P_2)<5={\rm deg\ }(H_5),\]
and we have the shift relation 
\[\pushQED{\qed} H_{5sh}\circ P_2= (Q-H_{5sh}\circ A)\circ H_5.
\qedhere\popQED \]

\begin{example} Shift operator $P_{+0}$ for the shift $sh:{\bm e}_1\to{\bm e}_1+1.$\par\noindent
  In this case, $H_{5sh}=H_5({\bm e}_1+1)$
  and we have $\partial\circ P_{+00}(e_9=0) = P_1\circ \partial$
  for some $P_1$. Let
\[P_1=A\circ H_5+ P_2\quad{\rm and}\quad Q_2=Q_{+00}(e_9=0)-H_{5sh}\circ A.\]
Then, we have the shift relation: $H_{5sh}\circ P_2 = Q_2 \circ H_5$,
where $P_2= x^3(x-1)^4(r+1)\partial^4+\cdots$ and 
$Q_2$ similar. Hence, $P_2=P_{+0}$ and $Q_2=Q_{+0}$ are obtained
as listed in \S \ref{AppenE5}.
\end{example}

\begin{example} Shift operator $P_{-0}$ for the shift $sh:{\bm e}_1\to{\bm e}_1-1.$\par
\par\noindent
In this case, for $H_6$,
\[P_{-00}=(x-1)\partial-r,\quad Q_{-00}=(x-1)\partial+3-r,\]
and $H_{5sh}=H_5({\bm e}_1-1)$.
Defining $P_2$ and $Q_2$ as above,
we have the shift relation  $H_{5sh}\circ P_2 = Q_2 \circ H_5$, where
$$P_2=P_{-0}:=(x-1)\partial+1-r, \quad Q_2=Q_{-0}:=(x-1)\partial+3-r.$$
\end{example}

For the shifts ${\bm e}_4\to{\bm e}_4\pm 1$, we have similar results.
Refer to  \S \ref{AppenE5}.

\begin{remark} The shift relations of $H_6$, which include shift of $e_9$, produce no new relations of $H_5$.
\end{remark}

\subsection{Reducible cases of $H_5$}\label{E5Fact}
When $H_5$ is reducible as in Theorem \ref{redcondE5},
the equation $H_5$ factorizes and $H_4$ and $H_3$ appear as factors:
% \par\smallskip\noindent

\begin{enumerate}[leftmargin=*]
  \renewcommand{\labelenumi}{\arabic{enumi})}
\item When $e_1-r=1$, {\it i.e.,} $e_1=(e_2+\cdots+ e_8-3)/2$,
  we find that $H_5$ factors of type [1,4],
  and the factor [4] has Riemann scheme as
   $$\left(\begin{array}{ccccc}
x=0:&  0 &1 &e_2-1  &e_3-1\\
x=1:&  0 &e_4-1 &e_5-1&e_6-1\\
x=\infty:&  e_7+1&e_8+1&7/2-e_{28}/2&9/2-e_{28}/2\end{array}\right),\quad e_{28}=e_2+\cdots+e_8.$$
After exchanging $x=1$ and $x=\infty$,
we multiply $(x-1)^{7/2-e_{28}/2}$ from the right. Renaming the exponents as
  $$0,\ 1,\ \epsilon_1,\ \epsilon_2;\quad0,\ 1,\ \epsilon_3,\ \epsilon_4;\quad s,\ \epsilon_5,\ \epsilon_6,\ \epsilon_7,$$
we can check that this coincides with $H_4(\epsilon)$,
which is defined in \S 1, %\ref{E4ap1},
and has $7$ $(=8-1)$ independent parameters.
\medskip

\item When $r=2$, $H_5$ factors as $[3,1,1]$. The factor $[1,1]$ is
  just $\partial^2$ and the Riemann scheme of
  $x^{-e_3-2}(x-1)^{-e_6-2}\circ[3]\circ x^{e_3-3}(x-1)^{e_6-3}$ is
  $$\left(\begin{array}{cccc}
x=0:&  0  &e_1-e_3  &e_2-e_3\\
x=1:&  0  &e_4-e_6  &e_5-e_6\\
x=\infty:&  e_3+e_6-3 &e_3+e_6+e_7-3& 9-e_1-e_2-e_4-e_5-e_7\end{array}\right).$$
   Renaming these exponents as 
  $$0,\ \epsilon_1,\ \epsilon_2;\quad 0,\ \epsilon_3,\ \epsilon_4;\quad s,\ \epsilon_5,\ \epsilon_6,$$
   we can check that this coincides with $H_3(\epsilon)$,
   which already appeared as a factor of $H_6$ (\S\ref{factorH6}), 
   and is defined in \S 1. This has $6$ $(=7-1)$ independent parameters.
\end{enumerate}

\noindent Summing up, we have the following proposition.
\begin{prp}\label{factorE5} $1)$ For $i=1,\dots,6$, 
  $$\begin{array}{ccccccc}
    e_i-r= &\cdots&-1&0      &    1  &2&\cdots\\
    & \cdots&[4,1] &[4,1]A0&[1,4]A0&[1,4] &\cdots
  \end{array}$$
  When $e_i+s=0,1$, the factor $[4]$ is {essentially} $H_4$.
%  defined in Section $\ref{E4ap1}$.
    \par\noindent
 $2)$   $$\begin{array}{cccccccc}
    r=& \cdots&-1    & 0&     1&2&3&\cdots\\
    &\cdots&[1,1,3] &[1,1,3]A0&[1,3,1]A0&[3,1,1]A0&[3,1,1]
    &\cdots\end{array}$$
    When $r=0,1,2$, the factor $[3]$ is {essentially} $H_3$.% defined in Section $\ref{E3}$.
\end{prp}

\subsection{Table of shift operators of $H_5$}\label{AppenE5} %, E5PQSdata}
\noindent {\sc Important convention}: {\it For a polynomial $U$ of $\theta$,
we denote by $U[k]$ the polynomial $U(\theta=\theta+k)$; say, $U[-2]$ for
$U(\theta=\theta-2)$. For a polynomial $B$ depending on parameters, $B_s$ denotes
the polynomial $B$ with shifted parameters in question.
}
\renewcommand{\labelenumi}{(\ref{AppenE5}.\arabic{enumi})}
\begin{enumerate}
\item $[-0]\quad({\bm e}_1-{\bm 1} =[e_1-1,e_2-1,e_3-1,r-1])$
\[P_{-0} = (x-1)\partial +1-r,\qquad Q_{-0} = (x-1)\partial +3-r.\]
$[+0]\quad({\bm e}_1+{\bm 1}=[e_1+1,e_2+1,e_3+1,r+1])$
\[\begin{array}{rcl}
P_{+0}&=&x^3P_{nnn}+x^2P_{nn}+xP_n+P_0+P_1\partial,\\ [1mm]
Q_{+0}&=&x^3Q_{nnn}+x^2Q_{nn}+xQ_n+Q_0+Q_1\partial,
\end{array}\]
\[
\hskip-36pt
\begin{array}{lcp{10.8cm}}\hline\noalign{\smallskip}  
  P_{nnn} &=&$(\theta-r+1)(\theta-r+2)(\theta+e_7+1)(\theta+e_8+1)$, \\
  P_{nn} &=& $-(\theta-2r+3)(\theta-r+1)(\theta+e_7+1)(\theta+e_8+1)+(\theta+1-r)B_{51}$,  \\
  P_{n} &=& $r(r-1)(\theta+e_7+1)(\theta+e_8+1) - (\theta-2r+2)B_{51} + \theta B_{52}[-1]$,  \\
  P_{0}&=& $-(\theta+r-1)(\theta+1-e_1)(\theta+1-e_2)(\theta+1-e_3)-(\theta-r+1)B_{52}[-1]$,  \\
  P_{1}&=& $(\theta+2-e_1)(\theta+2-e_2)(\theta+2-e_3)$,  \\
%\end{array} \]  % \hline
%\[  \begin{array}{lcp{10.8cm}}
  Q_{nnn} &=& $(\theta-r+3)(\theta-r+4)(\theta+e_7+3)(\theta+e_8+3)$, \\
  Q_{nn} &=& $-(\theta-2r+2)(\theta-r+3)(\theta+e_7+2)(\theta+e_8+2)+(\theta-r+3)B_{51s}[2]$, \\
  Q_{n} &=& $r(r-1)(\theta+e_7+1)(\theta+e_8+1) - (\theta-2r+2)B_{51s}[1] + (\theta+3)B_{52s}[1]$, \\
  Q_{0} &=& $-(\theta+r+1)(\theta+2-e_1)(\theta+2-e_2)(\theta+2-e_3)-(\theta-r+1)B_{52s}$, \\
  Q_{1}&=& $ (\theta+2-e_1)(\theta+2-e_2)(\theta+2-e_3)$, \\
\noalign{\smallskip}  \hline \noalign{\smallskip}  
&& $B_{51s}=B_{51}({\bm e}_1+{\bm 1}),\quad B_{52s}:=B_{52}({\bm e}_1+{\bm 1})$.
\end{array} \]

\vskip0.5pc
\item $[0-]\quad({\bm e}_4-{\bm 1}=[e_4-1,e_5-1,e_6-1,r-1])$
\[  P_{0-} = x\partial +1-r, \qquad Q_{0n}= x\partial+3-r.\]
  $[0+]\quad({\bm e}_4+{\bm 1}=[e_4+1,e_5+1,e_6+1,r+1])$
\[ \begin{array}{rcl}
  P_{0+}&=&x^3P_{nnn}+x^2P_{nn}+xP_n+P_0+P_1\partial,\\ [1mm]
  Q_{0+}&=&x^3Q_{nnn}+x^2Q_{nn}+xQ_n+Q_0+Q_1\partial,
\end{array} \]  
\[
\hskip-36pt
\begin{array}{lcp{10.8cm}}\hline\noalign{\smallskip}  
  P_{nnn} &=& $(\theta-r+1)(\theta-r+2)(\theta+e_7+1)(\theta+e_8+1)$,  \\ 
  P_{nn}  &=& $(\theta-r+1)B_{51}$,  \\ 
  P_{n}  &=&  $\theta B_{52}[-1]$,  \\ 
  P_{0}  &=&  (see\ below)  \\
  Q_{nnn}  &=& $(\theta-r+3)(\theta-r+4)(\theta+e_7+3)(\theta+e_8+3)$, \\
  Q_{nn}  &=& $(\theta-r+3)B_{51s}[2]$, \\
  Q_{n}  &=&  $(\theta+2)B_{52s}[1]$, \\
  Q_{0} &=& $P_{0}[2]$ \\
\noalign{\smallskip}    \hline \noalign{\smallskip}  
&& $B_{51s}=B_{51}({\bm e}_4+1),\quad B_{52s}=B_{52}({\bm e}_4+1)$.
\end{array}
\]
\[
\begin{array}{lcl}
  P_{0} &=& -\theta^4-(r+2-e_1-e_2-e_3)\theta^3-(r^2+(2-e_1-e_2-e_3)r-e_1 \\
  &&\qquad -e_2-e_3+e_1e_2+e_1e_3+e_2e_3)\theta^2 -(r^3+(2-e_1-e_2-e_3)r^2 \\
  &&\quad -(e_1+e_2+e_3-e_1e_2-e_1e_3-e_2e_3)r-2+e_1\\
  &&\quad +e_2+e_3-e_1e_2e_3)\theta-(r-1)(r-e_1+1)(r+1-e_2)(r+1-e_3).
\end{array}
\]

\end{enumerate}

\section{Shift operators of $H_4$ }\label{E4ap1}
%\nostcrule
\secttoc
In this section, we study the equation $H_4$.
As is stated in Proposition \ref{factorE5},
this equation appears as a factor of $H_5$, when $e_1-r=1$. It is also obtained from $H_3$ via middle convolution: practically, express $\partial\circ H_3$ as a linear combination of $\theta^i\partial^j$ ($0\le i+j\le4$) and replace $\theta$ by $\theta-u$.
%$$\partial^s\circ\partial \circ H_3\circ\partial^{-s}.$$

\subsection{A shift operator of $H_4$}
The equation $H_4=H_4(e_1,\dots,e_7)$ is defined in \S\ref{TheTenTable}.
Its $(x,\partial)$-form is as follows:
  $$H_4=x^2(x-1)^2\partial^4+\cdots+p_0,\quad p_0= e_5e_6e_7e_8.$$
It is easy to check that 
\[H_4(e') \circ \partial = \partial\circ H_4(e),
\quad e'=(e_1-1,\dots,e_4-1,e_5+1,e_6+1,e_7+1),
\]
which, in particular, implies $H_4$ has differentiation symmetry.
Thus, $\partial$ is the shift operator for the shift $e\to e'.$
Set $R=x^2(x-1)^2\partial^3+p_3\partial^2+p_2\partial+p_1$. Then we have
  \[R\circ \partial = H_4-p_0 \equiv -p_0\quad{\rm mod}\ H_4.\]
  This implies that $R$ gives the inverse of the map
  $\partial:{\rm Sol}(H_4(e)\longrightarrow {\rm Sol}(H_4(e'))$,
  and that the corresponding S-value is $p_0$.

\begin{prp}\label{E4red}If one of
  $$e_5,\ e_6,\ e_7,\ e_8(=s=4-(e_1+\cdots+e_7))$$ is an integer,
  then the equation $H_4$ is reducible.
  \end{prp}

We could not find other shift operator than $\partial$.

\subsection{Reducible cases of $H_4$}\label{E4ap1Red} 

\begin{prp}\label{E4redE3}
\[\def\arraystretch{1.1} \begin{array}{cccccccc}
e_5,\dots,e_8=\ & \cdots   &-1    & 0     & 1     & 2        &\cdots\\
  &   [31]    & [31] &[31]A0&[13]A0 &[13]&  [13]  
\end{array}
\]
  In particular, when $e_7=0,1$, we have
\[\def\arraystretch{1.1} \begin{array}{lcl}H_4(e_7=0)&=&
    H_3(e_1-1,\dots,e_4-1,e_5+1,e_6+1)\circ\partial,\\
    H_4(e_7=1)&=& \partial\circ H_3(e).
\end{array}\]
\end{prp}

\begin{proof} When $e_7=1$, $H_4$ factors as $[\partial, F_1]$.
  The local exponents of $F_1=x^2(x-1)^2\partial^3+\cdots$ are
$$[0,e_1,e_2],\ [0,e_3,e_4],\ [e_5,e_6,3-e_1-\cdots-e_6].$$
$F_1$ coincides with $H_3$ without modification. 
%Refer to Section \ref{E3}  for $H_3$.
\par\noindent
When $e_7=0$, $H_4$ factors as $[F_0,\partial]$. The local exponents of $F_0=x^2(x-1)^2+\cdots$ are
$$[0,e_1-1,e_2-1],\ [0,e_3-1,e_4-1],\ [e_5+1,e_6+1,5-e_1-\cdots-e_6],$$
and  $F_0=H_3(e_1-1, \dots, e_4-1, e_5+1, e_6+1)$.
\end{proof} % \hfill$\square$

\newpage
\bibliographystyle{amsalpha}

\bigskip

\noindent 
Yoshishige Haraoka

Josai University, Sakado 350-0295, Japan

haraoka@kumamoto-u.ac.jp

%Supported by the JSPS grant-in-aid for scientific research B, No.20H01810

\medskip \noindent
Hiroyuki Ochiai

Department of Mathematics, Kyushu University, Fukuoka 819-0395, Japan 

ochiai@imi.kyushu-u.ac.jp

\medskip \noindent 
Takeshi Sasaki

Kobe University, Kobe 657-8501, Japan 

yfd72128@nifty.com

\medskip \noindent 
Masaaki Yoshida 

Kyushu University, Fukuoka 819-0395, Japan 

myoshida1948@jcom.home.ne.jp
\end{document}